\documentclass[reqno]{amsart}

\RequirePackage{fix-cm}
\usepackage{graphicx}
\usepackage{mathptmx}      
\usepackage{latexsym}
\RequirePackage[OT1]{fontenc}

\usepackage{latexsym,amsmath}
\usepackage{amsmath,
amscd}
\usepackage{amsfonts}
\usepackage[psamsfonts]{amssymb}
\usepackage{enumerate}

\usepackage{url}
\usepackage{tocvsec2}

\usepackage[final]
{showkeys}
\usepackage{color}

\usepackage{epstopdf}
\usepackage{float}
\usepackage{array}
\usepackage{hyperref}%

\newtheorem{theorem}{Theorem}[section]
\newtheorem*{theorem*}{Theorem (Lehmann--Scheff\'e)}
\newtheorem{lemma}[theorem]{Lemma}
\newtheorem{corollary}[theorem]{Corollary}
\newtheorem{proposition}[theorem]{Proposition}

\theoremstyle{definition}

\newtheorem{remark}[theorem]{Remark}

\theoremstyle{definition} 
\newtheorem*{remark*}{Remark}

\numberwithin{equation}{section}

\newcommand{\R}{\mathbb{R}}
\newcommand{\N}{\mathbb{N}}

\newcommand{\CC}{\mathbb{C}}

\renewcommand{\H}{\mathcal{H}}

\newcommand{\T}{\mathcal{T}}

\newcommand{\tga}{\tilde\gamma}
\newcommand{\trho}{\tilde\rho}

\newcommand{\ffrown}{\text{\raisebox{3pt}[0pt][0pt]{$\frown$}}}

\renewcommand{\O}{\underset{\ffrown}{<}}
\newcommand{\OG}{\underset{\ffrown}{>}}
\newcommand{\OO}{\mathrel{\text{\raisebox{2pt}{$\O$}}}}
\newcommand{\OOG}{\mathrel{\text{\raisebox{2pt}{$\OG$}}}}

\newcommand{\al}{\alpha}
\newcommand{\be}{\beta}
\newcommand{\ga}{\gamma}
\newcommand{\de}{\delta}
\newcommand{\De}{\Delta}
\newcommand{\vp}{\varepsilon}
\newcommand{\ka}{\kappa}
\newcommand{\la}{\lambda}
\newcommand{\La}{\Lambda}

\newcommand{\si}{\sigma}
\newcommand{\Si}{\Sigma}
\renewcommand{\th}{\theta}

\newcommand{\om}{\omega}

\newcommand{\kk}{\mathbf{k}}
\newcommand{\uu}{\mathbf{u}}

\newcommand{\dd}{\operatorname{d}\!}

\newcommand{\erf}{\operatorname{erf}}

\newcommand{\ii}{\operatorname{I}}

\newcommand{\Alt}{\mathsf{Alt}}
\newcommand{\EM}{\mathsf{EM}}
\newcommand{\Ra}{\mathsf{Ra}}

\newcommand{\gaeu}{\operatorname{\ga_{\,\mathsf{Eu}}}}
\newcommand{\der}{\mathsf{der}}
\newcommand{\ev}{\mathsf{ev}}
\newcommand{\od}{\mathsf{od}}

\newcommand{\tT}{\tilde T}

\newcommand{\pd}[2]{\frac{\partial^{#1}}{{\partial#2\,}^{#1}}}
\newcommand{\fd}[2]{\frac{\dd^{#1}}{{\dd#2\,}^{#1}}}

\newcommand{\ce}{\centering}

\begin{document}

\title[Approximating sums by integrals only]{An alternative to the Euler--Maclaurin formula: \\ Approximating sums by integrals only}



\author{Iosif Pinelis}
\address{Department of Mathematical Sciences, Michigan Technological University}
\email{ipinelis@mtu.edu}


%
%
%

 	

\begin{abstract} 
The Euler--Maclaurin (EM) summation formula is used in many theoretical studies and numerical calculations. It approximates the sum 
$\sum_{k=0}^{n-1} f(k)$ of values of a function $f$ by a linear combination of a corresponding integral of $f$ and values of its higher-order derivatives $f^{(j)}$. 
An alternative (Alt) summation formula is proposed, which approximates the sum by a linear combination of integrals only,  
without using high-order derivatives of $f$. 
Explicit and rather easy to use bounds on the remainder are given. 
Possible extensions to multi-index summation are suggested. 
Applications to summing possibly divergent series are presented. 
It is shown that the Alt formula will in most cases outperform, or greatly outperform, the EM formula in terms of the execution time and memory use. One of the advantages of the Alt calculations is that, in contrast with the EM ones, they can be almost completely parallelized.  
Illustrative examples are given. 
In one of the examples, where an array of values of the Hurwitz generalized zeta function is computed with high accuracy, it is shown that both our implementation of the EM formula and, especially, the Alt formula perform much faster than the built-in Mathematica command \texttt{HurwitzZeta[]}. 
\end{abstract}

\maketitle

\setcounter{tocdepth}{1}
\tableofcontents

\section{Introduction}\label{intro}

The Euler--Maclaurin (EM) summation formula
can be written as follows: 
\begin{equation*}\label{eq:EMintro}
	\sum_{k=0}^{n-1} f(k)\approx 
	\int^n_0\dd x\, f(x)  
    +
    \sum_{j=1}^{2m-1}\frac{B_j}{j!}[f^{(j-1)}(n) - f^{(j-1)}(0)], \tag{EM}      
\end{equation*} 
where $f$ is a smooth enough function, 
$B_{j}$ is the $j$th Bernoulli number, and $n$ and $m$ are natural numbers. The remainder/error term of the approximation provided by this formula has a certain explicit integral expression, which depends linearly on $f^{(2m-1)}$; in particular, the EM approximation is exact when $f$ is a polynomial of degree $<2m-1$. 
The integral $\int^n_0\dd x\, f(x)$ in \eqref{eq:EMintro} may be considered the main term of the approximation, whereas the summands $\frac{B_j}{j!}[f^{(j-1)}(n) - f^{(j-1)}(0)]$ may be referred to as the correction terms.  
 
The EM formula has been used 
in a large number of theoretical studies and numerical calculations, even outside mathematical sciences -- see e.g.\ applications of this formula in problems of renormalization in
quantum electrodynamics and quantum field theory \cite[\S2.15.6]{kleinert} and \cite[\S3.1.3]{sadovskii}. 
The EM formula is 
used, in particular, for numerical calculations of sums in such mathematical software packages as Mathematica (command \verb!NSum[]! with option \verb!Method->"EulerMaclaurin"!), Maple (command \verb!eulermac()!\,), PARI/GP (command 
\verb!sumnum()!), 
and Matlab (\verb!numeric::sum()!). EM is also used internally in other built-in commands in such software packages requiring calculations of sums. 
%

Clearly, to use the EM formula in a theoretical or computational study, one will usually need to have an  antiderivative $F$ of $f$ and the derivatives $f^{(j-1)}$ for $j=1,\dots,2m-1$ in tractable or, respectively, computable form. 

In this paper, an alternative summation formula (Alt) is offered, which approximates the sum $\sum_{k=0}^{n-1} f(k)$ by a linear combination of values of an antiderivative $F$ of $f$ only, without using values of any derivatives of $f$: 
\begin{equation*}\label{eq:intro}
	\sum_{k=0}^{n-1} f(k)\approx 
	\sum_{j=1-m}^{m-1}\tau_{m,1+|j|}\,\int_
	{-1/2-j/2}^{n-1/2-j/2}\dd x\, f(x), \tag{Alt}  
\end{equation*}
where $f$ is again a smooth enough function, 
the coefficients $\tau_{m,r}$ are certain rational numbers not depending on $f$ and such that $\sum_{j=1-m}^{m-1}\tau_{m,1+|j|}=1$, and $n$ and $m$ are natural numbers. 
Similarly to the case of the EM formula, the remainder/error term of the approximation provided by the Alt formula has a certain explicit integral expression, which depends linearly on $f^{(2m)}$; in particular, the Alt approximation is exact when $f$ is a polynomial of degree $<2m$. 

Even though the exact expressions of the remainders in the Alt and EM formulas involve higher-order derivatives $f^{(2m-1)}$ and $f^{(2m)}$ of the function $f$, such derivatives need not be computed to compute tight enough bounds on the remainders in typical applications, where 
the function $f$ has rather natural analyticity properties; in such cases, in view of Cauchy's integral formula, it is enough to have appropriate bounds just on the function $f$ itself. 

The coefficients $\tau_{m,r}$ in the Alt formula are significantly easier to compute than the Bernoulli numbers $B_j$, which are the coefficients in the EM formula. Moreover, it will be shown (see \eqref{eq:=B_p}) that the Bernoulli numbers $B_j$ can be expressed as certain linear combinations of $\tau_{m,r}$'s. 

As was noted, to use either the EM formula or the Alt one, one needs to have an antiderivative $F$ of $f$ in tractable/computable form anyway. Then $f$ (being the derivative of $F$) will usually be of complexity no less than that of $F$. On the other hand, the complexity of higher-order derivatives of $f$ will usually be much greater than that of $F$ or $f$. This is the main advantage of the Alt summation formula over the EM one. 

Closely related to this is another important advantage of Alt over EM, to be detailed in Subsection~\ref{par}: the Alt calculations can be almost completely parallelized, whereas the calculation of the derivatives $f^{(j-1)}$ for $j=1,\dots,2m-1$ is non-parallelizable -- except for a few special cases when these derivatives are available in simple closed form, rather than only recursively. 

Even in such special cases, comparatively least favorable for the Alt formula, it will outperform the EM one when the needed accuracy is high enough. As explained in Subsection~\ref{euler}, this will happen because of the comparatively large time needed to compute the Bernoulli numbers. 

Summarizing these considerations, we see that the Alt formula should be usually expected to outperform the EM one.  
Alt's advantage will usually be the greater, the greater accuracy of the result is required -- because then one will need to use a greater number $\asymp m$ of correction terms $\frac{B_j}{j!}[f^{(j-1)}(n) - f^{(j-1)}(0)]$, and thus a higher order of the derivatives of $f$ in the EM formula will be needed. 
In Section~\ref{examples}, we shall consider specific examples illustrating such general expectations for high-accuracy calculations.
\big(Here and in what follows, for any two positive expressions $a$ and $b$, we write $a\asymp b$ to mean that $a\OO b\ \&\ a\OOG b$; in turn,
$a\OO b$ means $a\le Cb$ for some 
constant $C$, and $a\OOG b$ means $b\OO a$. We also write $a\sim b$ for $a/b\to1$.\big) 

Concerning very high accuracy, one may ask -- as it was done, rhetorically,  
in the Foreword by D.~H.~Bailey to \cite{100digit}: 
``[...] why should anyone care about finding any answers to 10,000-digit
accuracy?'' 
An answer to this question was given in the same Foreword: 
``[...] recent work in experimental mathematics has provided an important venue
where numerical results are needed to very high numerical precision, in some cases to
thousands of decimal digits. In particular, precision on this scale is often required when
applying integer relation algorithms to discover new mathematical identities. [...] 
Numerical quadrature (i.e., numerical evaluation of definite integrals), series
evaluation, and limit evaluation, each performed to very high precision, are particularly
important. These algorithms require multiple-precision arithmetic, of course, but often also
involve significant symbolic manipulation and considerable mathematical cleverness as well.''
The footnote on the same page in \cite{100digit} 
adds this information: ``Such algorithms were ranked among \emph{The Top 10} [...]
of algorithms `with the greatest influence on the development and practice of science and engineering in the 20th
century.' '' 
See also \cite{bailey-borwein} for some specifics on this, as well as  
\cite{bornemann} for a recent update on  \cite{100digit} -- where, in particular, an example is cited, which is now part of Mathematica's documentation, stating that the largest positive root of Riemann's prime counting function $R$ is ``shocking[ly]'' small, about $1.83\times10^{-14828}$. 


The EM formula was introduced about $282$ years ago. Accordingly, a large body of literature has been produced with further developments in the theory and applications of this tool. 
In contrast, the Alt summation formula appears to have no precedents in the literature. The idea behind the formula and techniques used in its proof also appear to be new.  
Given the mostly superior performance of the Alt formula in calculations of sums and the 
already uncovered theoretical connections with the EM formula, 
new developments in the theory of the Alt formula and its applications do not seem unlikely, and they are  certainly welcome. 

\medskip

The rest of this paper is organized as follows. 

In Section~\ref{EU}, a rigorous review of the EU formula is given, with an application to the Faulhaber formula for the sums of powers, to be subsequently used in the paper. 

In Section~\ref{result}, the Alt formula is rigorously stated, with discussion. The main idea behind this formula is presented. The mentioned expression of the Bernoulli numbers (which are the coefficients in the EM formula) in terms of the coefficients in the Alt formula is stated as well. Possible extensions of the Alt formula to the case of multi-index sums are suggested. 

In Section~\ref{R bounds}, explicit and rather easy to use bounds on the remainders in the Alt and EM formulas are presented, especially in the case when the function $f$ has rather natural analyticity properties. 

In Section~\ref{series}, applications of the Alt and EM formulas to summing possibly divergent series are given. It is shown that the corresponding generalized sums, $\sum_{k\ge0}^\Alt f(k)$ and $\sum_{k\ge0}^\EM f(k)$ are equal to each other, and that they both differ from the Ramanujan sum $\sum_{k\ge0}^\Ra f(k)$ by the additive constant $F(0)$, where $F$ is the chosen antiderivative of $f$. 
A simple shift trick is presented, which allows one to make the remainders in the Alt and EM formulas arbitrarily small; this trick is an extension of the method used by Knuth~\cite{knuth62} to compute an approximation to the value of the Euler constant. 

In Section~\ref{c,m}, it is discussed how to choose the number $m$ ``of the correction terms'' in the Alt and EM formulas and the value (say $c$) of the shift mentioned in the previous paragraph -- in order to obtain the desired number $d$ of digits of accuracy of the (generalized) sum in a nearly optimal time. This is significantly more difficult to do for the EM formula, mainly because it is difficult to assess, especially in general terms, the time needed to compute the values of the higher-order derivatives of $f$. 

A general and rather straightforward way to parallelize the calculations by the Alt formula is detailed in Section~\ref{par}. The matter of memory use is also considered there. In most cases, the Alt calculations will require substantially less memory than the corresponding EU ones. 

To illustrate the above comparisons between the Alt and EM formulas' performance, 
four specific examples of a function $f$ are considered in Section~\ref{examples}, with various levels of the complexity of $f$, its antiderivative $F$, and its higher-order derivatives $f^{(j-1)}$. Measurements of the execution time and memory use for the Alt and EM formulas are presented there for various values of the desired number $d$ of digits of accuracy. Similar measurements are also presented for the Richardson extrapolation process (REP) in the two of the examples where this kind of extrapolation seems applicable. 
It appears that the REP cannot compete with either the EM or Alt formula as far as high-accuracy calculations of sums are concerned. 
It also turns out, as shown in the example in Subsection~\ref{zeta}, where an array of values of the Hurwitz generalized zeta function is computed, that both our implementation of the EM formula and, especially, the Alt formula significantly outperform the built-in Mathematica command \texttt{HurwitzZeta[]} in terms of the execution time, which suggests that our code is rather well optimized in that respect. 

One should note here that the mentioned comparisons of the performance of the Alt summation with the EM one and the REP concern only problems of summation. Of course, the REP has many other uses in which the Alt formula is not applicable at all. Also, the EM formula may be used to approximate integrals by sums, which cannot be done in general with the Alt formula.  


The necessary proofs are deferred to Section~\ref{proof}. 

\medskip

At the end of this introduction, let us fix the notation to be used in the rest of the paper:  
For any natural number $\al$, let $C^{\al-}$ denote the set of all functions $f\colon\R\to\R$ such that $f$ has continuous derivatives $f^{(i)}$ of all orders $i=0,\dots,\al-1$ and the derivative $f^{(\al-1)}$ is absolutely continuous, with a Radon--Nikodym derivative denoted here simply by $f^{(\al)}$. 
As usual, $f^{(0)}:=f$. 

Suppose that $n$ is a nonnegative integer, $m$ is a natural number, and $f\in C^{2m-}$.

\section{
The EM summation formula}\label{EU}
Here is an exact form of the formula \eqref{eq:EMintro} stated in the Introduction  
(see e.g.\ 
\cite{knopp51
}): 
\begin{equation}\label{eq:EM}
\begin{aligned}
	\sum_{k=0}^{n-1} f(k) =A^\EM_m + R^\EM_m, 
\end{aligned}    
\end{equation}
where 
\begin{equation}\label{eq:A^EM}
	   A^\EM_m:= \int^{n-1}_0\!\!\!\dd x\, f(x) 
    + \frac{f(n-1) + f(0)}2  
    +
    \sum_{j=1}^{m-1}\frac{B_{2j}}{(2j)!}[f^{(2j - 1)}(n-1) - f^{(2j - 1)}(0)],   
\end{equation} 
$B_{j}$ is the $j$th Bernoulli number, $R^\EM_m$ is the remainder given by the formula 
\begin{equation}\label{eq:R^EM}
	R^\EM_m:=\frac1{(2m-1)!}\,\int_0^{n-1} \dd x\,f^{(2m-1)}(x)\,
	B_{2m-1}(x-\lfloor x\rfloor),  
\end{equation}
and $B_j(x)$ is the $j$th Bernoulli polynomial, defined recursively by the conditions 
$B_0(x)=1$, $B_j'(x)=jB_{j - 1}(x)$, and $\int_0^1 \dd x\,B_j(x) = 0$ for $j=1,2,\dots$ and real $x$. 
In particular, for all $j=2,3,\dots$ the $j$th Bernoulli number coincides with the value of the $j$th Bernoulli polynomial at $0$: $B_j=B_j(0)$.  
Here and in what follows, we assume the standard convention, according to which the sum of an empty family is $0$. 
So, the sum in \eqref{eq:A^EM} equals $0$ if $m=1$. 
It is known that for all real $x\in[0,1]$ 
\begin{equation*}
	|B_{2m-1}(x)|\le\frac{2(2m-1)!}{(2\pi)^{2m-1}}\,\zeta(2m-1);    
\end{equation*}
see e.g.\ \cite[page~525]{knopp51}.  
Therefore, 
\begin{equation}\label{eq:R^EM<}
	|R^\EM_m|\le\frac{2\zeta(2m-1)}{(2\pi)^{2m-1}}\,\int_0^{n-1} \dd x\,|f^{(2m-1)}(x)|. 
\end{equation}
Here $\zeta$ is the Riemann zeta function, so that 
$\zeta(2m-1)<1.01$ for $m\ge4$.  

In \cite{butzer-etal11}, it is shown that the Abel-Plana summation formula, the Poisson summation formula, and the approximate sampling formula are in a certain sense equivalent to the EM summation formula.   

For $m=0$, the Euler--MacLaurin formula takes the form 
$$\sum_{k=0}^{n-1} f(k) =
    \int^{n-1}_0\dd x\, f(x)+ \frac{f(n-1) + f(0)}2+R^\EM_0.$$ 
Therefore, the general formula \eqref{eq:EM} can be viewed as a higher-order extension of the trapezoidal quadrature formula.   

One may note that the Euler--MacLaurin formula implies the Faulhaber formula for the sums of powers. Indeed, take any natural $p$ and $n$. 
Then, 
using \eqref{eq:EM} with $f(x)=x^p$, $m=\lceil p/2\rceil+1$, and $n$ in place of $n-1$, and recalling that $B_3=B_5=\dots=0$, one has $f^{(2m-1)}=0$ and hence $R^\EM_m=0$ and 
\begin{equation*}
	\sum_{k=0}^{n-1}k^p=-n^p+\sum_{k=0}^{n}k^p=\frac{n^{p+1}}{p+1}-\frac{n^p}2
	+\frac1{p+1}\,\sum_{\al=2}^p
	B_\al \binom{p+1}\al n^{p+1-\al};    
\end{equation*}
here we also use the simple observation that $f^{(p)}(n) - f^{(p)}(0)=0$ if $f(x)=x^p$. 
Therefore and because $B_0=1$ and $B_1=-1/2$, we have the following version of the Faulhaber formula:    
\begin{equation}\label{eq:faulhaber}
	\sum_{k=0}^{n-1}k^p=\frac1{p+1}\,\sum_{\al=0}^p B_\al \binom{p+1}\al n^{p+1-\al},   
\end{equation}
for any natural $p$ and $n$; 
this conclusion holds for any nonnegative integers $p$ and $n$ as well -- assuming the convention $0^0:=1$, which will be indeed assumed throughout this paper. 
We shall use \eqref{eq:faulhaber} in the proof of Proposition~\ref{lem:=B_p}. 

\section{
An alternative (Alt) to the EM formula}\label{result}

%
The following is the main result of this paper: 

\begin{theorem}\label{th:}
%
One has 
\begin{equation}\label{eq:}
	\sum_{k=0}^{n-1}f(k)
	=A_m-R_m,
\end{equation}

\newcommand{\bx}[2]{\makebox[#1in][l]{$\displaystyle{#2}$}}

where 
\begin{align}
	A_m&:=\bx{1.6}{\sum_{j=1}^m\ga_{m,j}\sum_{i=0}^{j-1}\int_{i-j/2}^{n-1+j/2-i}}
	\bx{1.5}{=\sum_{j=1}^m\ga_{m,j}\sum_{i=0}^{j-1}\int_{-1+j/2-i}^{n-1+j/2-i} } 
	\label{eq:A_m} \\ 
	&=\bx{1.6}{\sum_{\al=1-m}^{m-1}\tau_{m,1+|\al|}\,\int_{\al/2-1/2}^{n-1/2-\al/2} }
		\bx{1.7}{=\sum_{\al=1-m}^{m-1}\tau_{m,1+|\al|}\,\int_{-1/2-\al/2}^{n-1/2-\al/2} }
	\label{eq:A_m,alt1} \\ 
	&
	\begin{aligned}
	=&\bx{3.5}{\;\tau_{m,1}\,\int_{-1/2}^{n-1/2}
	+\sum_{\al=1}^{m-1}\tau_{m,1+\al}\,
	\Big(\int_{\al/2-1/2}^{n-1/2-\al/2}
	+\int_{-\al/2-1/2}^{n-1/2+\al/2}
	\Big) }\\ 
		=&\bx{3.5}{\;\tau_{m,1}\,\int_{-1/2}^{n-1/2}
	+\sum_{\al=1}^{m-1}\tau_{m,1+\al}\,
	\Big(\int_{-1/2-\al/2}^{n-1/2-\al/2}
	+\int_{-1/2+\al/2}^{n-1/2+\al/2} 
	\Big)	} 
	\end{aligned} \label{eq:A_m,alt2}
\end{align}  
is the integral approximation to the sum $\sum_{k=0}^{n-1}f(k)$, 
\begin{equation}\label{eq:int}
	\int_a^b:=\int_a^b \dd x\,f(x):=F(b)-F(a),   
\end{equation}
$F$ is any antiderivative of $f$ \big(so that $\int_a^b=\int_{(a,b]}$\, if $a\le b$\big),  
\begin{equation}\label{eq:ga_j}
	\ga_{m,j}:=(-1)^{j-1}\,\frac2j\,\binom{2m}{m+j}\Big/ \binom{2m}{m}, 
\end{equation}
\begin{equation}\label{eq:tau_j}
	\tau_{m,r}:=
	\sum_{\be=0}^{\lfloor m/2-r/2\rfloor}\ga_{m,r+2\be}
	\,=\sum_{\be=0}^\infty\ga_{m,r+2\be}, 
\end{equation}
and 
$R_m$ is the remainder given by the formula 
\begin{equation}\label{eq:R_m}
	R_m
	 :=\frac1{(2m-1)!\,2^{2m+1}}\,\int_0^1\dd s\,(1-s)^{2m-1}\int_{-1}^{1}\dd v\,v^{2m}
	\sum_{j=1}^m\ga_{m,j}j^{2m+1} \sum_{k=0}^{n-1}f^{(2m)}(k+jsv/2). 
\end{equation}
The sum of all the coefficients of the integrals in each of the expressions in \eqref{eq:A_m}, \eqref{eq:A_m,alt1}, and \eqref{eq:A_m,alt2} of $A_m$ is   
\begin{equation}\label{eq:sum ga_j}
	\sum_{j=1}^m\ga_{m,j}\sum_{i=0}^{j-1}1=\sum_{j=1}^m\ga_{m,j} j
	=\sum_{\al=1-m}^{m-1}\tau_{m,1+|\al|}=1. 
\end{equation}

If $M_{2m}$ is a real number such that 
\begin{equation}\label{eq:<M}
	\Big|\sum_{k=0}^{n-1}f^{(2m)}(k+w)\Big|\le M_{2m}\quad\text{for all}\quad w\in[-m/2,m/2], 
\end{equation}
then the remainder $R_m$ can be bounded as follows: 
\begin{align}
	|R_m|&\le\frac{M_{2m}}{(2m+1)!\,2^{2m}}\,\sum_{j=1}^m|\ga_{m,j}|j^{2m+1}.  \label{eq:|R|<} 
\end{align}
\end{theorem}

Recall the convention that the sum of an empty family is $0$. 
In particular, if $n=0$, then $\sum_{k=0}^{n-1}f(k)=0 
=A_m=R_m$.  

A formal proof of Theorem~\ref{th:} will be given in 
Section~\ref{proof}. At this point, let us just present the idea leading to representation \eqref{eq:}. 
First here, one may note that, in accordance with \eqref{eq:}--\eqref{eq:tau_j}, the first approximation $A_1=\int_{-1/2}^{n-1/2}$ of the sum $\sum_{k=0}^{n-1}f(k)$ is 
obtained by approximating each summand $f(k)$ by $\int_{k-1/2}^{k+1/2}$.  
Next, formally integrating the Taylor expansion 
\begin{equation}\label{eq:taylor}
f(x)=f(k)+f'(k)(x-k)+f''(k)(x-k)^2/2+\cdots	
\end{equation}
in $x$ from $k-1/2$ to $k+1/2$ and then summing in $k=0,\dots,n-1$, one sees that 
\begin{equation}\label{eq:int1}
	\sum_{k=0}^{n-1}f(k)=\int_{-1/2}^{n-1/2}-\sum_{k=0}^{n-1}f''(k)/24-\cdots,  
\end{equation}
with the ellipsis $\cdots$ standing for a ``negligible'' remainder.  
Integrating now the Taylor expansion \eqref{eq:taylor} in $x$ from $k-1$ to $k+1$ and then summing again in $k=0,\dots,n-1$, one has 
\begin{equation}\label{eq:int2}
	\sum_{k=0}^{n-1}f(k)=\frac12\,\Big(\int_{-1}^n+\int_0^{n-1}\Big)-\sum_{k=0}^{n-1}f''(k)/6-\cdots,  
\end{equation}
because 
\begin{equation}\label{eq:sum-ints}
	\sum_{k=0}^{n-1}\int_{k-1}^{k+1}=\sum_{k=0}^{n-1}\Big(\int_{k-1}^k+\int_k^{k+1}\Big)
	=\int_{-1}^{n-1}+\int_0^n=\int_{-1}^n+\int_0^{n-1}. 
\end{equation}
Multiplying now both sides of the identities in \eqref{eq:int1} and \eqref{eq:int2} by $4/3$ and $-1/3$, respectively, and then adding the resulting identities, we eliminate the term containing the second derivatives $f''(k)$: 
\begin{equation}\label{eq:A_2-dots}
	\sum_{k=0}^{n-1}f(k)=\frac43\,\int_{-1/2}^{n-1/2}
	-\frac16\Big(\int_{-1}^n+\int_0^{n-1}\Big)-\cdots=A_2-\cdots,   
\end{equation}
where $A_m$ is as in \eqref{eq:A_m}--\eqref{eq:A_m,alt2}. The last identity, $\sum_{k=0}^{n-1}f(k)=A_2-\cdots$, is a non-explicit and non-rigorous version of identity \eqref{eq:} for $m=2$. 

Similarly eliminating higher-order derivatives in an explicit and rigorous version of the Taylor expansion and generalizing the calculations in \eqref{eq:sum-ints}, we shall derive identity \eqref{eq:} for all natural $m$.  

The idea described above may remind one the idea of the mentioned earier Richardson extrapolation process (REP), which results in an elimination of higher-order terms of an 
asymptotic expansion and thus in a higher rate of convergence; see e.g.\ \cite{richardson,sidi}. An important difference between the two ideas is that Richardson's elimination is iterative, in distinction with our representation \eqref{eq:}, and so, one has to be concerned with the stability of the REP and its generalizations; cf.\ e.g.\ \cite[Sections~0.5.2 and 1.6]{sidi}. 
Moreover, it will be shown in Section~\ref{par} that the computation by formula \eqref{eq:} can be almost completely parallelized, whereas the iterative character of Richardson's method makes that problematic, if at all possible. 
However, 
one may note the following. 

\begin{remark}\label{rem:ga}
Let 
\begin{equation}\label{eq:rho:=} 
	\rho_j:=\rho_{m,j}:=\ga_{m,j} j \text{\ \ for\ \ } j=1,\dots,m. 
\end{equation}
Then one has the (downward) recursion 
\begin{equation}\label{eq:rho iter}
	\rho_{m,j-1}=\rho_{m,j}\frac{m+j}{j-m-1} \text{ for } j=m,\dots,2,\quad\text{with}\quad
	\rho_{m,m}=(-1)^{m-1}2\Big/ \binom{2m}{m}. 
\end{equation}
Of course, this can be rewritten as an ``upward'' recursion. However, it is the downward recursion that will be used for parallelization, to be described in detail in Section~\ref{par}. 
\hfill\qed
\end{remark}

\begin{remark}\label{rem:intrs}
In each of the formulas \eqref{eq:A_m}, \eqref{eq:A_m,alt1}, and \eqref{eq:A_m,alt2}, the first expression is a linear combination of integrals over intervals centered at the point $(n-1)/2$, whereas the endpoints of each of the intervals corresponding to the second expression differ by $n$.\hfill\qed
\end{remark}

\begin{remark}\label{rem:A_m-alt} 
The expressions for $A_m$ in \eqref{eq:A_m,alt1} are obtained from those in \eqref{eq:A_m} by grouping the summands in the double sum with the same integral $\int_{i-j/2}^{n-1+j/2-i}$ or $\int_{-1+j/2-i}^{n-1+j/2-i}$ \big(and then the expressions for $A_m$ in the two lines of \eqref{eq:A_m,alt2} are obtained from the corresponding expressions in \eqref{eq:A_m,alt1} by grouping the summands with the same value of $\tau_{m,1+|\al|}$\big). So, each of the expressions in \eqref{eq:A_m,alt2} requires the calculation of $2m-1$ integrals, which is much fewer for large $m$ than  
the $(m+1)m/2$ integrals in \eqref{eq:A_m}.  
On the other hand, the coefficients $\ga_{m,j}$ in \eqref{eq:A_m} are a bit easier to compute than the coefficients $\tau_{m,j}$ in \eqref{eq:A_m,alt1} or \eqref{eq:A_m,alt2}. 
\hfill\qed
\end{remark}

\begin{remark}\label{rem:vect}
Instead of assuming that the function $f$ is real-valued, one may assume, more generally, that $f$ takes values in any normed space. In particular, one may allow $f$ to take values in the $q$-dimensional complex space $\CC^q$, for any natural $q$. Such a situation will be considered in the example in Subsection~\ref{zeta}. 
An advantage of dealing with a vector function such as the one defined by \eqref{eq:f,vect} (rather than separately with each of its components) is that this way one has to compute the coefficients -- say $\tau_{m,\be}$ in \eqref{eq:A_m,alt2} and $B_{2j}/(2j)!$ in \eqref{eq:A^EM} -- only once, for all the components of the vector function.  
\hfill\qed
\end{remark}

We have the following curious and useful identity, whereby the Bernoulli numbers $B_p$, which are the coefficients in the EM approximation \eqref{eq:A^EM}, are expressed as linear combinations of the coefficients $\ga_{m,j}$ and $\tau_{m,r}$ in the Alt approximation \eqref{eq:A_m}--\eqref{eq:A_m,alt2} to the sum $\sum_{k=0}^{n-1} f(k)$. 


\begin{proposition}\label{lem:=B_p}
Take any $p=0,\dots,2m-1$.  
Then  
\begin{equation}\label{eq:=B_p}
	B_p=\sum_{j=1}^m\ga_{m,j}\,\sum_{i=0}^{j-1}\Big(\frac j2-i-1\Big)^p
	=\frac1{2^p}\,\Big(\tau_{m,1}(-1)^p
	+\sum _{\be=2}^m \tau_{m,\be}\big[(\be-2)^p+(-\be)^p\big]\Big).   
\end{equation} 
\end{proposition}

The last expression in \eqref{eq:=B_p} is obtained from the preceding one by grouping the summands in the double sum with the same value of $\frac j2-i-1$; cf.\ Remark~\ref{rem:A_m-alt}. 

Identity \eqref{eq:=B_p} will be used in the proofs of Propositions~\ref{prop:series} and \ref{prop:same}. 
In fact, this identity was discovered in numerical experiments suggesting the limit relation \eqref{eq:same} in Proposition~\ref{prop:same}.
The special case of \eqref{eq:=B_p} for $p=0$ is 
\eqref{eq:sum ga_j}.

Identities of a kind somewhat similar to \eqref{eq:=B_p} have been known. A survey of them was given in \cite{gould}. In particular, identity \cite[(1)]{gould} can be written as 
\begin{equation}\label{eq:gould}
	B_p=\sum _{j=0}^p \frac1{j+1}\sum _{i=0}^j (-1)^i \binom{j}{i} i^p
	=\sum _{j=0}^m \frac1{j+1}\sum _{i=0}^j (-1)^i \binom{j}{i} i^p
\end{equation}
for any $p=0,1,\dots$ and any $m=p,p+1,\dots$; the second equality in \eqref{eq:gould} holds because \break $\sum _{i=0}^j (-1)^i \binom{j}{i} i^p=0$ for all $j=p+1,p+2,\dots$. 

A notable distinction between identities \eqref{eq:=B_p} and \eqref{eq:gould} is that, in view of the definition \eqref{eq:ga_j} of $\ga_{m,j}$, the binomial coefficients involved in \eqref{eq:=B_p} are of the form $\binom{2m}\cdot$, with the constant choose-from index $2m[\ge p+1]$, whereas \eqref{eq:gould} involves all the first $p+1$ rows of the Pascal binomial triangle.

Recall the notation introduced in \eqref{eq:int}. 
The first 
three approximations $A_m$ of the sum $\sum_{k=0}^{n-1}f(k)$ are as follows: 
\begin{align}
A_1=& \int_{-1/2}^{n-1/2}, \notag \\
A_2=&  \frac{4}{3} \int_{-1/2}^{n-1/2}
-\frac{1}{6} \Big(\int_{-1}^n+\int_0^{n-1}\Big)\quad \text{(cf.\ \eqref{eq:A_2-dots})}, 
\notag \\
A_3=&  \frac{3}{2} \int_{-1/2}^{n-1/2}-\frac{3}{10} \Big(\int_{-1}^n+\int_0^{n-1}\Big)+\frac{1}{30} \Big(\int_{-3/2}^{n+1/2}+\int_{-1/2}^{n-1/2} +\int_{1/2}^{n-3/2}\Big) \label{eq:A_3}\\
=&  \frac{23}{15} \int_{-1/2}^{n-1/2}-\frac{3}{10} \Big(\int_{-1}^n+\int_0^{n-1}\Big)+\frac{1}{30} \Big(\int_{-3/2}^{n+1/2}+
\int_{1/2}^{n-3/2}\Big),   \label{eq:A_3-alt} 
\end{align}
using the first 
expression for $A_m$ in \eqref{eq:A_m}; cf.\ Remark~\ref{rem:intrs}. 

If (as usually will be the case) $n\ge m-1$, then the integral approximation $A_m$ can be written as just one integral, as follows:
\begin{equation}\label{eq:one int}
	A_m=\int_{-m/2}^{n-1+m/2}\dd x\,f(x)h_m(x), 
\end{equation}
where 
\begin{equation}\label{eq:h_m}
	h_m:=\sum_{j=1}^m\ga_{m,j}\sum_{i=0}^{j-1}\ii_{(i-j/2,n-1+j/2-i]}
	=\sum_{\al=1-m}^{m-1}\tau_{m,1+|\al|}\,\ii_{(\al/2-1/2,n-1/2-\al/2]} 
\end{equation}
and $\ii_A$ denotes the indicator function of a set $A$. 

The integral approximation of the sum $\sum_{k=0}^{n-1}f(k)$ is illustrated in Figure~\ref{fig:TeX}, for $n=10$ and $m=3$. In the left panel of the figure, each of the six integrals $\int_a^b$ in the expression \eqref{eq:A_3} for $A_3$ is represented by a rectangle whose projection onto the horizontal axis is the interval $(a,b]$ and whose height equals the absolute value of the coefficient of the integral in that expression for $A_3$. The rectangle is placed above or below the horizontal axis depending on whether the respective coefficient is positive or negative. 
Thus, each such rectangle also represents a summand of the form $\ga_{m,j}\ii_{(i-j/2,n-1+j/2-i]}$ in the expression \eqref{eq:h_m} of $h_m$. 
The rectangles of the same height are shown in the same color.  E.g., the two green rectangles represent the integrals $\int_{-1}^n=\int_{-1}^{10}$ and $\int_0^{n-1}=\int_0^9$; the height of each of these green rectangles is $\frac3{10}$, the absolute value of the coefficient $-\frac3{10}$ of these integrals, and these rectangles are ``negative'' (that is, below the horizontal axis), since the coefficient $-\frac3{10}$ is negative.   

The 
resulting function $h_3$, which is a sort of sum of all the ``positive'' and  ``negative'' rectangles or, more precisely, the sum of the corresponding functions 
$\ga_{m,j}\ii_{(i-j/2,n-1+j/2-i]}$ (for $n=10$), is shown in the right panel of Figure~\ref{fig:TeX}. 
In accordance with \eqref{eq:A_3}--\eqref{eq:A_3-alt}, the middle blue rectangle has the same base as, and hence can be ``absorbed into'', the red rectangle. 

\begin{figure}[htbp]
	\centering
		\includegraphics[width=1.00\textwidth]{
		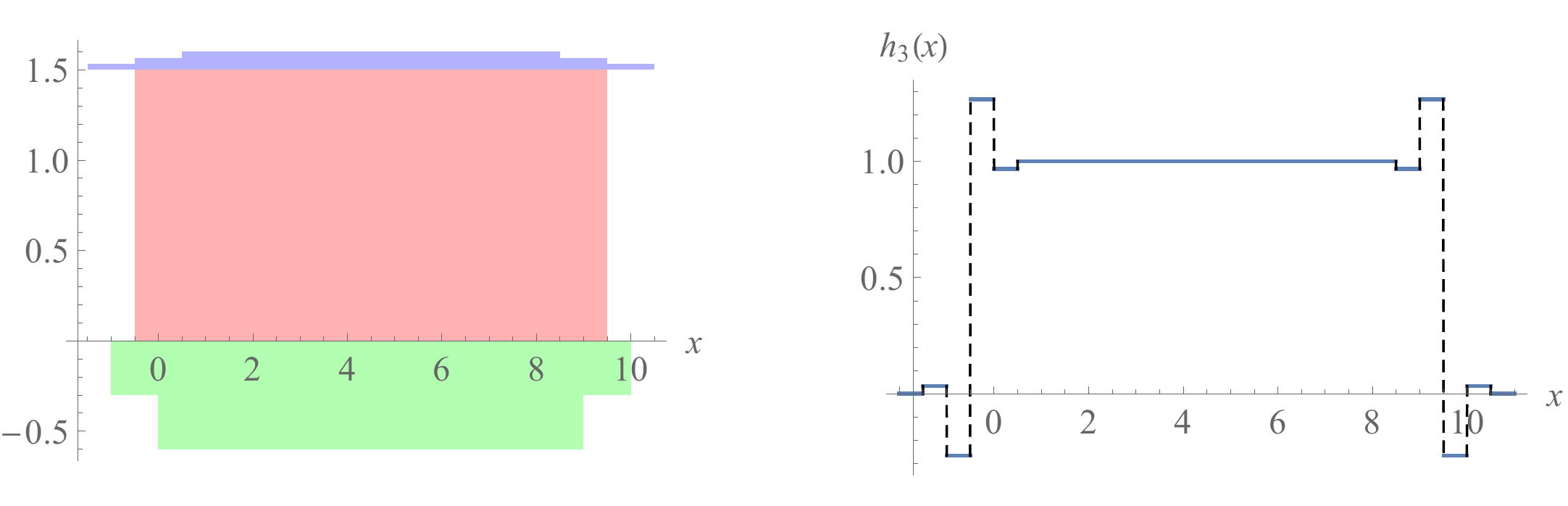}
	\caption{Left panel: Graphical representation of the integral approximation $A_3$ for $n=10$. Right panel: Graph of the function $h_3$ for $n=10$. }
	\label{fig:TeX}
\end{figure}

One can see that the proposed integral approximation of the sum $\sum_{k=0}^{n-1}f(k)$ works by (i)~``borrowing'' information about how the function $f$ integrates in left and right neighborhoods of, respectively, the left and right endpoints of the interval $[0,n-1]$ and (ii) taking into account boundary effects near the endpoints both inside and outside the interval $[0,n-1]$. 

\begin{remark}\label{rem:infty}
For real $a$, let $\int_a^{\infty-}\dd x\,f(x):=\lim\big(\int_a^{r/2}\dd x\,f(x)\colon r\in\N,\,r\to\infty\big)$, if this limit exists and is finite. 
If such an ``improper'' integral $\int_{-m/2}^{\infty-}\dd x\,f(x)$ exists and is finite and if  
the series \break 
$\sum_{k=0}^\infty f^{(2m)}(k+v)$ converges uniformly in $v\in[-m/2,m/2]$, then \eqref{eq:} will hold if the instances of $\sum_{k=0}^{n-1}$, $\int_{i-j/2}^{n-1+j/2-i}$, 
$\int_{-1+j/2-i}^{n-1+j/2-i}$, $\int_{\al/2-1/2}^{n-1/2-\al/2}$, $\int_{-1/2-\al/2}^{n-1/2-\al/2}$, $\int_{-\al/2-1/2}^{n-1/2+\al/2}$, and $\int_{-1/2+\al/2}^{n-1/2+\al/2}$ in \eqref{eq:}, \eqref{eq:A_m}, \eqref{eq:A_m,alt1}, \eqref{eq:A_m,alt2}, and \eqref{eq:R_m} are replaced respectively by $\sum_{k=0}^\infty$, $\int_{i-j/2}^{\infty-}$, $\int_{-1+j/2-i}^{\infty-}$, $\int_{\al/2-1/2}^{\infty-}$, $\int_{-1/2-\al/2}^{\infty-}$, $\int_{-\al/2-1/2}^{\infty-}$, and $\int_{-1/2+\al/2}^{\infty-}$. 
\hfill\qed
\end{remark}

\begin{remark}\label{rem:ep g(ep)}
Suppose that the function $f$ in Theorem~\ref{th:} is given by the formula $f(x)=g(\vp x)\vp$ for some function $g$, some real $\vp>0$, and all real $x$. So, if $\vp$ is a small number, the sum $\sum_{k=0}^{n-1}f(k)=\sum_{k=0}^{n-1}g(\vp k)\vp$ may be thought of as an integral sum for the function $g$ over a fine partition of an interval. Suppose now that, for instance, the function $|g^{(2m)}|$ is nondecreasing on the interval $[-\vp-\vp m/2,\infty)$ and let $\tilde M_{2m}:=\int_{-\vp-\vp m/2}^\infty\dd y\,|g^{(2m)}(y)|$. Then for all $v\in[-m/2,m/2]$
\begin{equation*}
	\sum_{k=0}^{n-1}|f^{(2m)}(k+v)|=\vp^{2m}\sum_{k=0}^{n-1}|g^{(2m)}\big((k+v)\vp)|\,\vp
	\le\vp^{2m}\tilde M_{2m}, 
\end{equation*}
so that, by \eqref{eq:|R|<}, 
\begin{equation*}
		|R_m|\le\vp^{2m} \frac{\tilde M_{2m}}{(2m+1)!\,2^{2m}}\,\sum_{j=1}^m|\ga_{m,j}|j^{2m+1}, 
\end{equation*}
which provides a justification for referring to $R_m$ as the remainder. 
Another, less trivial but hopefully more convincing justification will be provided in Sections~\ref{R bounds}, \ref{series}, and \ref{c,m}. 

Also, it is clear that $R_m=0$ if the function $f$ is any polynomial of degree at most $2m-1$. 
\hfill\qed
\end{remark}

\begin{remark}
The alternative summation formula given in Theorem~\ref{th:} can be generalized to multi-index sums. Indeed, one can similarly approximate multi-index sums $$\sum_{k_1=0}^{n_1-1}\cdots\sum_{k_r=0}^{n_r-1}f(k_1,\dots,k_r)$$ 
of values of functions $f$ of several variables by linear combinations of corresponding integrals of $f$ over rectangles in $\R^r$, using the multivariable Taylor expansions 
$f(\kk+\uu)=f(\kk)+
f'(\kk)\cdot\uu+\cdots$, where $\kk:=(k_1,\dots,k_r)$ and $\uu:=(u_1,\dots,u_r)$. 
See \cite{sums-via-ints-multivar}, where an extension to sums over lattice polytopes is also given. 
\hfill\qed
\end{remark}

\section{Bounds on the remainders}\label{R bounds}

\begin{remark}\label{rem:R}
If on the interval $(-m/2-1/2,n-1/2+m/2)$ one has $|f^{(2m)}|\le g_{2m}$ for some nonnegative convex function $g_{2m}\colon\R\to\R$, then \eqref{eq:<M} will hold with 
\begin{equation}\label{eq:M,eu-const}
	M_{2m}=\int_{-m/2-1/2}^{n-1/2+m/2}\dd x\, g_{2m}(x)\le\int_{-m/2-1/2}^\infty\dd x\, g_{2m}(x);
\end{equation}
here we used the simple observation that $g(a)\le\int_{a-1/2}^{a+1/2}\dd x\,g(x)$ for any real $a$ and any function $g$ that is convex on the interval $(a-1/2,a+1/2)$.  
\hfill\qed
\end{remark}

In typical applications, the function $f$ will admit an 
analytic extension of at most a polynomial growth into a large enough region of the complex plane $\CC$ containing $[0,\infty)$. Then the higher-order derivatives of $f$ can be nicely bounded without an explicit evaluation of them, and then one can use Remark~\ref{rem:R} and 
\eqref{eq:|R|<} to easily bound the remainder $R_m$. 
Such a bounding tool is provided by 

\begin{proposition}\label{prop:cauchy}  
For real $a\ge(m+3)/2$ and $\th_0\in(0,\pi/2]$, consider the 
sector 
\begin{equation*}
	S:=\{-a+re^{i\th}\colon r\ge0,\,|\th|\le\th_0\}    
\end{equation*} 
of the complex plane $\CC$ with vertex at the point $-a$, so that $S\supset[0,\infty)$; 
of course, 
here $i$ denotes the imaginary unit.  
Suppose that the function $f\colon\R\to\R$ can be extended to a function (which we shall still denote by $f$) mapping $\R\cup S$ into $\CC$ such that $f$ is continuous on the set $S$, 
holomorphic in the interior of $S$, and 
\begin{equation}\label{eq:|f|<}
	|f(z)|\le\mu\,
	|z+a+1|^\la
\end{equation}
for some real $\mu\in[0,\infty)$ and $\la\in[0,2m-1)$ and for all $z\in S$. Then 
\begin{equation}\label{eq:R_m<}
	|R_m|\le\frac{C(\mu,\la,\th_0,m)}{(a-m/2-1/2)^{2m-1-\la}}, 
\end{equation}
where 
\begin{equation}\label{eq:C:=}
	C(\mu,\la,\th_0,m):=	\frac{\mu\,(
	2+\sin\th_0)^\la}{(2m+1)(2m-1-\la)(2\sin\th_0)^{2m}}\,\sum_{j=1}^m|\ga_{m,j}|j^{2m+1}. 
\end{equation}
Similarly, for $m\ge4$, $a>0$, and $\la\in[0,2m-2)$,  
\begin{equation}\label{eq:R_m^EM<}
	|R^\EM_m|\le\frac{C^\EM(\mu,\la,\th_0,m)}{a^{2m-2-\la}}, 
\end{equation}
where 
\begin{equation}\label{eq:C^EM:=}
	C^\EM(\mu,\la,\th_0,m):=	\frac{2.02\mu\,(
	2+\sin\th_0)^\la}{(2m-2-\la)(\sin\th_0)^{2m-1}}\,
	\frac{(2m-1)!}{(2\pi)^{2m-1}}.  
\end{equation}
\end{proposition} 

In applications, $m$ will be rather large, and $a$ will be substantially greater than $m$, to make the upper bounds on $|R_m|$ and $|R^\EM_m|$ in \eqref{eq:R_m<}  and \eqref{eq:R_m^EM<} very small. 

\begin{remark}\label{rem:mu} 
As in Proposition~\ref{prop:cauchy}, suppose that $f$ is continuous on the set $S$ and 
holomorphic in the interior of $S$. 
Consider the function $h$ defined by the formula $h(z):=f(z)/(z+a+1)^\la$ for $z\in S$. 
In view of the Phragm\'en--Lindel\"of extension of the maximum modulus principle 
(applied to $h$), for the condition \eqref{eq:|f|<} to hold for all $z\in S$, it is enough that \eqref{eq:|f|<} hold for all $z$ on the boundary of $S$ -- provided that $f$ does not grow too fast on $S$, that is, provided that $|f(z)|\le Ce^{|z+a|^\al}$ for some $\al\in\big[0,\frac\pi{2\th_0}\big)$,  some real $C$, and all $z\in S$; see e.g.\ \cite[\S 5.6.1]{titchmarsh}. 
\end{remark}

Bounds \eqref{eq:|R|<} and \eqref{eq:R_m<}--\eqref{eq:C:=} are complemented by 

\begin{proposition}\label{prop:sum gaj}
If $m\ge2$ then 
\begin{equation}\label{eq:sum gaj}
	\sum_{j=1}^m|\ga_{m,j}|j^{2m+1}\le1.001\pi\La_*^m m^{2m+1}, 
\end{equation}
where 
\begin{equation}\label{eq:La}
	\La_*:=\max_{0<t<1}\La(t)=0.3081\dots,\quad 
	\La(t):=(1-t)^{t-1} (1+t)^{-1-t} t^2. 
\end{equation}
For $m=1$, \eqref{eq:sum gaj} holds with $1.0331$ in place of $1.001$. 
\end{proposition}

It appears that in most practical situations it will be possible to use Proposition~\ref{prop:cauchy} with $\th_0=\pi/2$. In such a case, the expression for the constant $ $ in \eqref{eq:C:=} can be simplified, and we immediately obtain the following corollary of Propositions~\ref{prop:cauchy} and \ref{prop:sum gaj}. 

\begin{corollary}\label{cor:R bound}
Suppose that $m\ge2$ and the conditions in Proposition~\ref{prop:cauchy} hold with $\th_0=\pi/2$. Then 
\begin{equation}\label{eq:R bound}
	|R_m|\le \frac{1.001\pi\mu 3^\la}{(2m+1)(2m-1-\la)}\,\Big(\frac{\La_*}4\Big)^m 
	\frac{m^{2m+1}}{(a-m/2-1/2)^{2m-1-\la}}. 
\end{equation}
For $m=1$, \eqref{eq:R bound} holds with $1.0331$ in place of $1.001$. 
\end{corollary}

\section{Application to summing (possibly divergent) series}\label{series}

The alternative summation formula presented in Theorem~\ref{th:} can be used for summing (possibly divergent) series, as follows.  

\begin{proposition}\label{prop:series}
Let $m_0$ be a natural number, and suppose that $m\ge m_0$. 
Suppose that 
\begin{equation}\label{eq:f^ to0}
	f^{(2m_0-1)}(x)\underset{x\to\infty}\longrightarrow0 
\end{equation}
and the series 
\begin{equation}\label{eq:R unif}
	\sum_{k=0}^\infty f^{(2m)}(k+w) \ \text{converges uniformly in $w\in[-m/2,m/2]$. }
\end{equation}
Let $F$ be any antiderivative of $f$, so that $F'=f$. 
Then 
\begin{equation}\label{eq:series}
\sum_{k\ge0}^\Alt f(k):=\lim_{n\to\infty}	\Big(\sum_{k=0}^{n-1} f(k)-G_{m_0,F}(n)\Big)=
	-G_{m,F}(0)-R_{m,f}(\infty), 
\end{equation}
where (cf.\ \eqref{eq:A_m}, \eqref{eq:A_m,alt1},  and \eqref{eq:A_m,alt2})   
\begin{align}
G_{m,F}(n):=& \sum_{j=1}^m\ga_{m,j}\sum_{i=0}^{j-1}F(n-1+j/2-i) \label{eq:G}\\ 
	=&\sum_{\al=1-m}^{m-1}\tau_{m,1+|\al|}\,F(n-1/2-\al/2)
	\label{eq:G_m,alt1} 
	\\ 
	=&\tau_{m,1}\,F(n-1/2)
	+\sum_{\al=1}^{m-1}\tau_{m,1+\al}\,
	\big[F(n-1/2-\al/2)
	+F(n-1/2+\al/2)\big]
	\label{eq:G_m,alt2}
\end{align}
and (cf.\ \eqref{eq:R_m}) 
\begin{equation}\label{eq:R_m,f}
	R_{m,f}(\infty):=
	\frac1{(2m-1)!\,2^{2m+1}}\,\int_0^1\dd s\,(1-s)^{2m-1}\int_{-1}^{1}\dd v\,v^{2m}
	\sum_{j=1}^m\ga_{m,j}j^{2m+1} \sum_{k=0}^\infty f^{(2m)}(k+jsv/2). 
\end{equation} 
\end{proposition}

The limit $\sum_{k\ge0}^\Alt f(k)$ in \eqref{eq:series} may be referred to as the (generalized) sum of the possibly divergent series $\sum_{k=0}^\infty f(k)$ by means of the Alt formula \eqref{eq:}. 

\begin{remark}\label{rem:m_0=1}
The centering/stabilizing term $G_{m_0,F}(n)$ in \eqref{eq:series} equals $F(n-1/2)$ for all natural $n$ if one can take $m_0=1$. 
Thus, if $F(\infty):=\lim_{x\to\infty}F(x)$ exists and is finite, and if $F$ is chosen so that $F(\infty)=0$, then we will have $\sum_{k\ge0}^\Alt f(k)=\sum_{k=0}^\infty f(k)$. 
\end{remark}

The key point in the proof of 
Proposition~\ref{prop:series} is that, under the conditions of Proposition~\ref{prop:series},  
$G_{m,F}(n)-G_{m_0,F}(n)\underset{n\to\infty}\longrightarrow0$. 

The EM summation formula, too, can be used for summing possibly divergent series. The following proposition is rather similar to 
Proposition~\ref{prop:series}. 

\begin{proposition}\label{prop:seriesEM}
Let $m_0$ be a 
natural number, and suppose that $m\ge m_0$. 
Suppose that 
\begin{equation}\label{eq:f^^ to0}
	f^{(2j-1)}(x)\underset{x\to\infty}\longrightarrow0 \text{\quad for $j=m_0,\dots,m-1$ }
\end{equation}
and 
\begin{equation}\label{eq:REM unif}
	\int_0^{\infty} \dd x\,|f^{(2m-1)}(x)|<\infty. 
\end{equation}
Let $F$ be any antiderivative of $f$, so that $F'=f$. 
Then 
\begin{equation}\label{eq:seriesEM}
\sum_{k\ge0}^\EM f(k):=	\lim_{n\to\infty}\Big(\sum_{k=0}^{n-1} f(k)-G^\EM_{m_0,F}(n)\Big)=
	f(0)-G^\EM_{m,F}(1)+R^\EM_{m,f}(\infty), 
\end{equation}
where (cf.\ \eqref{eq:A^EM})  
\begin{align}
G^\EM_{m,F}(n):=& F(n-1) 
    + \frac{F'(n-1)}2  
    +
    \sum_{j=1}^{m-1}\frac{B_{2j}}{(2j)!}F^{(2j)}(n-1)
	\label{eq:G^EM_m} 
\end{align}
and (cf.\ \eqref{eq:R^EM}) 
\begin{equation*}
	R^\EM_{m,f}(\infty):=
	\frac1{(2m-1)!}\,\int_0^{\infty} \dd x\,f^{(2m-1)}(x)\,
	B_{2m-1}(x-\lfloor x\rfloor). 
\end{equation*} 
\end{proposition}

The limit $\sum_{k\ge0}^\EM f(k)$ in \eqref{eq:seriesEM} may be referred to as the (generalized) sum of the possibly divergent series $\sum_{k=0}^\infty f(k)$ by means of the EM formula \eqref{eq:EM}. 

\begin{remark}\label{rem:m_0}
Suppose that the function $f$ satisfies conditions of Proposition~\ref{prop:cauchy}. 
Then, in view of the bound \eqref{eq:der-bound} in the proof of Proposition~\ref{prop:cauchy}, conditions \eqref{eq:f^ to0}, \eqref{eq:R unif}, \eqref{eq:f`to0}, and \eqref{eq:REM unif} will all hold if $2m_0>1+\la$ and $m>m_0$. 
\end{remark}

A curious fact is that the centering/stabilizing terms, $G_{m_0,F}(n)$ in \eqref{eq:series} and $G^\EM_{m_0,F}(n)$ in \eqref{eq:seriesEM}, are asymptotically interchangeable. More specifically, one has 
 
\begin{proposition}\label{prop:same}
Let the conditions of Propositions~\ref{prop:series} and \ref{prop:seriesEM} hold. Then  
\begin{equation}\label{eq:same}
	G_{m_0,F}(n)-G^\EM_{m_0,F}(n)\underset{n\to\infty}\longrightarrow0,  
\end{equation}
and hence the generalized sums in \eqref{eq:series} and \eqref{eq:seriesEM} are equal to each other: 
\begin{equation}\label{eq:Alt=EM}
\sum_{k\ge0}^\Alt f(k)=\sum_{k\ge0}^\EM f(k). 	
\end{equation}
Moreover, $G_{m,P}(n)=G^\EM_{m,P}(n)$ for any polynomial $P$ of degree $\le2m_0-1$ and any $n$. 
\end{proposition} 

Proposition~\ref{prop:same} suggests some curious ``objectivity'' in summing possibly divergent series, when the two seemingly quite different methods of summation yield the same result. 
Note that each of the generalized sums $\sum_{k\ge0}^\Alt f(k)$ and $\sum_{k\ge0}^\EM f(k)$ in \eqref{eq:series} and \eqref{eq:seriesEM} depends on the choice of the additive constant in the expression of an antiderivative $F$ of $f$ and is thus similar to the indefinite integral $\int f(x)\dd x$; yet, these two generalized sums are equal to each other for any choice of an antiderivative $F$ of $f$, provided that the conditions of Propositions~\ref{prop:series} and \ref{prop:seriesEM} hold. 

Further, one may note that 
\begin{equation*}
	\Big[\sum_{k\ge0}^\Alt f(k)=\Big]\sum_{k\ge0}^\EM f(k)=F(0)+\sum_{k\ge0}^\Ra f(k), 
\end{equation*}
again provided that the conditions of Propositions~\ref{prop:series} and \ref{prop:seriesEM} hold, where 
$\sum_{k\ge0}^\Ra f(k)$ denotes the Ramanujan constant of the series $\sum_{k=0}^\infty f(k)$. 
Cf.\ e.g.\ the expression for the Ramanujan constant $\sum_{k\ge1}^{\mathcal R} f(k)$ of the series $\sum_{k=1}^\infty f(k)$ in formula (23) in \cite{candel-etal}; to match the expression on the right-hand side of formula \eqref{eq:seriesEM} in the present paper with that in \cite[(23)]{candel-etal}, one should accordingly replace there $N$, $\partial^{k-1}f(1)$, and $\int_1^\infty$ by $2m-1$,  $\partial^{k-1}f(0)
[=f^{(k-1)}(0)]$, and $\int_0^\infty$, respectively, and also recall that $B_1=-1/2$ and $B_3=B_5=\dots=0$. 
So, choosing the antiderivative $F$ of $f$ determined by the condition $F(0)=0$, we would have $\sum_{k\ge0}^\Alt f(k)=\sum_{k\ge0}^\EM f(k)=\sum_{k\ge0}^\Ra f(k)$. 

However, such a choice of $F$ may not always be the most natural one. 
For instance, by \eqref{eq:Alt=EM}, \eqref{eq:seriesEM}, and \eqref{eq:zeta=Z}, 
$\sum_{k\ge0}^\Alt f(k)=\sum_{k\ge0}^\EM f(k)=\zeta(p,\de)$ for $f=f_{p,\de}$ as in \eqref{eq:f_p,de}, $F=F_{p,\de}$ as in \eqref{eq:F_p,de}, $p\in\CC\setminus\{1\}$, and $\de\in\CC\setminus(-\infty,0]$, whereas $\sum_{k\ge0}^\Ra f_{p,\de}(k)=\zeta(p,\de)+F_{p,\de}(0)=\zeta(p,\de)+\tfrac1{1-p}\,\de^{1-p}\ne\zeta(p,\de)$ (cf.\ \cite[formula~(52)]{candel-etal}. 
Cf.\ also Remark~\ref{rem:m_0=1}, which suggests that, in the case when the series $\sum_{k=0}^\infty f(k)$ converges, the natural choice of $F$ will usually be given by the condition $F(\infty)=0$, rather than $F(0)=0$. 


\medskip

To compute the generalized sums 
$\sum_{k\ge0}^\Alt f(k)$ and $\sum_{k\ge0}^\EM f(k)$ effectively, one has to make sure that the remainders $R_{m,f}(\infty)$ and $R^\EM_{m,f}(\infty)$ can be made arbitrarily small. It can be seen that 
the bounds in \eqref{eq:|R|<}--\eqref{eq:<M} and \eqref{eq:R^EM<} are rather tight. Therefore, usually the only way to ensure that the remainders $R_{m,f}(\infty)$ and $R^\EM_{m,f}(\infty)$ be small will be to make the high-order derivatives $f^{(2m)}$  and $f^{(2m-1)}$ small. This can be achieved by the following simple trick. 

For any function $h\colon\R\to\R$ and any real $c$, let $h_c$ denote the $c$-shift of $h$ defined by the formula 
\begin{equation*}
	h_c(x):=h(x+c)
\end{equation*}
for all real $x$. Let now $c$ be any natural number, and suppose the conditions in Proposition~\ref{prop:series} hold. Note that 
\begin{equation}\label{eq:G_c}
G_{m,F}(n+c)=G_{m,F_c}(n)	
\end{equation}
for all natural $n$. So, 
\begin{equation*}
\sum_{k=0}^{n+c-1} f(k)-G_{m_0,F}(n+c)
=\sum_{k=0}^{c-1} f(k)+\Big(\sum_{k=0}^{n-1} f_c(k)-G_{m_0,F_c}(n)\Big).   
\end{equation*}
Letting now $n\to\infty$ and using \eqref{eq:G_c} again (now with $0$ in place of $n$), we see 
that Proposition~\ref{prop:series} immediately yields 

\begin{corollary}\label{cor:series}
Under the conditions in Proposition~\ref{prop:series}, for any natural $c$ 
\begin{equation}\label{eq:series,c}
\sum_{k\ge0}^\Alt f(k)
=
	\sum_{k=0}^{c-1} f(k)-G_{m,F}(c)-R_{m,f_c}(\infty).  
\end{equation}
\end{corollary}

That is, \eqref{eq:series} holds if $G_{m,F}(0)$ and $R_{m,f}(\infty)$ are replaced there by $-\sum_{k=0}^{c-1} f(k)+G_{m,F}(c)$ and $R_{m,f_c}(\infty)$, respectively. 
Under the condition \eqref{eq:R unif}, one can indeed make the remainder $R_{m,f_c}(\infty)$ arbitrarily small by taking a large enough $c$. The extra price to pay for this is the need to compute the additional term $\sum_{k=0}^{c-1} f(k)$.

Similarly, Proposition~\ref{prop:seriesEM} immediately yields 

\begin{corollary}\label{cor:seriesEM}
Under the conditions in Proposition~\ref{prop:seriesEM}, for any natural $c$ 
\begin{equation}\label{eq:series,c,EM}
\sum_{k\ge0}^\EM f(k)
=
	\sum_{k=0}^{c
	} f(k)-G^\EM_{m,F}(c+1)+R^\EM_{m,f_c}(\infty).  
\end{equation}
\end{corollary}
Under the condition \eqref{eq:REM unif}, one can make the remainder $R^\EM_{m,f_c}(\infty)$ arbitrarily small by taking a large enough $c$. 

One can view \eqref{eq:series,c} and \eqref{eq:series,c,EM} as extrapolation formulas (in a broad enough sense), with the correction terms $-G_{m,F}(c)$ and $-G^\EM_{m,F}(c+1)$ added to the partial sums $\sum_{k=0}^{c-1} f(k)$ and $\sum_{k=0}^c f(k)$ of the series $\sum_{k=0}^\infty f(k)$ 
to obtain better approximations to its generalized sum $\sum_{k\ge0}^\Alt f(k)=\sum_{k\ge0}^\EM f(k)$.  

The special case of 
Corollary~\ref{cor:seriesEM} with $f(x)=\frac1{x+1}$ for $x\ge0$ was, essentially, the basis of Knuth's method in \cite{knuth62} to compute a decimal approximation to Euler's constant; cf.\ \cite[formula~(7)]{knuth62}. In that case, one can take 
$m_0=1$ and $F(x)=\ln(x+1)$ for $x\ge0$. Then the ``stabilizers'' in \eqref{eq:series,c} and  \eqref{eq:series,c,EM} will be $G_{m_0,F}(n)=\ln(n+1/2)$ and $G^\EM_{m_0,F}(n)=\ln n+\frac1{2n}
$, 
respectively, which in particular illustrates \eqref{eq:same}. 


\section{Choosing 
\texorpdfstring{$c$}{c} and \texorpdfstring{$m$}{m} 
for a desired accuracy}\label{c,m}


As was noted, the integer $c$ should be taken to be large enough to ensure that the remainders $R_{m,f_c}(\infty)$ and $R^\EM_{m,f_c}(\infty)$ in \eqref{eq:series,c} and  \eqref{eq:series,c,EM} be small. How large $c$ must be depends on the choice of $m$. In turn, one can see that the order $m$ of the approximation formulas \eqref{eq:} and \eqref{eq:EM}
should also be large enough for the remainders to be small. 

In the following two subsections, we shall consider these theses in some detail. As we shall see, there is a certain balance between the required values of $m$ and $c$. 
If the value of $m$ is too small, then the required value of $c$ must be too large to ensure the desired accuracy. 
Thus, trying to decrease $m$ to reduce the volume of calculations to compute $G_{m,F}(c)$ or $G^\EM_{m,F}(c+1)$ in \eqref{eq:series,c} and \eqref{eq:series,c,EM} according to \eqref{eq:G}--\eqref{eq:G_m,alt2} and \eqref{eq:G^EM_m} may result in an increased volume of calculations to compute the sums $\sum_{k=0}^{c-1} f(k)$ and $\sum_{k=0}^{c} f(k)$ in \eqref{eq:series,c} and \eqref{eq:series,c,EM}. 

Suppose that $m\ge2$ and all the conditions in Proposition~\ref{prop:cauchy} hold with $\th_0=\pi/2$ -- except possibly the condition $a\ge(m+3)/2$, so that here 
\begin{equation}\label{eq:S=Pi}
S=\Pi_{-a}^+:=\{z\in\CC\colon\Re z\ge-a\}. 	
\end{equation}
Let $c$ be any natural number such that $a+c\ge(m+3)/2$ -- this latter condition replacing the condition $a\ge(m+3)/2$. 
Condition \eqref{eq:|f|<} will hold (now for all $z\in\Pi_{-c-a}^+$) with $f_c$ and $a_c:=a+c$ in place of $f$ and $a$, respectively. 

\subsection{Choosing \texorpdfstring{$c$}{c} and \texorpdfstring{$m$}{m} nearly optimally in \texorpdfstring{\eqref{eq:series,c}}{} 
}\label{c,m,alt}

In view of Corollary~\ref{cor:R bound}, 
\begin{align}
	|R_{m,f_c}(\infty)|\le R^*_{m,c}:=R^*_{m,c;\mu,\la,a}
	&:=
	\frac{1.001\pi\mu3^\la}{(2m+1)(2m-1-\la)}\,\Big(\frac{\La_*}4\Big)^m 
	\frac{m^{2m+1}}{(c+a-m/2-1/2)^{2m-1-\la}} \label{eq:R_c,infty}\\ 
	&=\Big(\frac{\ka m}{c-m/2}\Big)^{(2+o(1))m} \notag
\end{align}
for $m\to\infty$ and $c-m/2\ge(\ka+\vp)m$ for some fixed real $\vp>0$, where 
\begin{equation}\label{eq:ka}
	\ka:=\sqrt{\frac{\La_*}4}=0.27754\dots. 
\end{equation}
So, to ensure that 
\begin{equation*}
	|R_{m,f_c}(\infty)|\le\tfrac12\,10^{-d} 
\end{equation*}
for a large enough natural $d$, an appropriate choice of $c$ will be 
as follows: 
\begin{equation}\label{eq:c=}
	c=\lceil c_{d,m}\rceil\approx m/2+\ka m10^{d/(2m)},  
\end{equation}
where $c_{d,m}=c_{d,m;\mu,\la,a}$ is the root $c$ of the equation $R^*_{m,c}=\frac12\,10^{-d}$. 

Let now $T_f=T_f(d)$ denote the time needed to compute one value of the function $f$. We suppose here that this time does not depend significantly on the value of the argument of $f$ in the range $\{0,\dots,c-1\}$; otherwise, take the average over the range. 

However, $T_f=T_f(d)$ will depend on the working accuracy needed to attain the desired accuracy of $\tfrac12\,10^{-d}$ of the ultimate result. We shall see at the end of this subsection that, for large $d$, this working accuracy will have to be $10^{-d_1}$, where $d_1$ exceeds $d$ only by a summand \,$\asymp\ln d$. 
  
Similarly introduce $T_F=T_F(d)$, the time ``cost'' per value of the antiderivative $F$ of $f$, and $T_\tau=T_\tau(d)$, the time ``cost'' per value of $\tau_{\cdot,\cdot}$ in \eqref{eq:G_m,alt2}. 
Then the total time needed to compute the approximate value $\sum_{k=0}^{c-1} f(k)-G_{m,F}(c)$ of the generalized sum $\sum_{k\ge0}^\Alt f(k)$ in \eqref{eq:series,c} is 
\begin{equation}\label{eq:T}
	\T\approx T_f c+T_\tau m+T_F\times2m
	\approx \T(m):=K m\big(1+\om\,10^{d/(2m)}\big)
\end{equation}
in view of \eqref{eq:c=}, where 
\begin{equation}\label{eq:om}
	K:=T_f/2+T_\tau+2T_F\quad\text{and}\quad \om:=\frac{\ka T_f}K. 
\end{equation}
We can now find an approximately optimal value of $m$ by minimizing $\T(m)$ in $m$, and then choose $c$ in accordance with \eqref{eq:c=}. 
Assuming that $T_F\ge T_f$ and hence $K>2.5T_f$, 
we will have 
\begin{equation}\label{eq:om<.1}
	\om<\frac\ka{2.5}\approx0.1. 
\end{equation}

It is easy to see that $\T(m)$ is strictly convex in $m\ge1$, $\T(m)\to\infty$ as $m\to\infty$, and 
\begin{equation*}
	\T'(1)/K=1+\om10^{d/2}(1-\tfrac d2\,\ln10)<1-10^{252}\om<0
\end{equation*}
provided that $d\ge500$ (which will be the case in the examples to be considered in this paper) and $\om>10^{-252}$. The latter condition will hold in all realistic situations -- when the time ``cost'' per value of the antiderivative $F$ is not quite prohibitively high. 
Therefore, $\T(m)$ attains its minimum in $m\ge1$ at the unique root $m=m_\om\in(1,\infty)$ of the equation $\T'(m)=0$, which can be rewritten as 
\begin{equation}\label{eq:T'(m)=0}
1+\om 10^{d/(2m)}(1-\tfrac d{2m}\,\ln10)=0. 	
\end{equation}
This root is given by the formula 
\begin{equation}\label{eq:m_om}
	m_\om=\frac{\ln10}{1+L(\tfrac{1}{e \om})}\,\frac d2
	\approx
	\frac{\ln10}{2\ln\tfrac{1}{\om}}\,d,   
\end{equation}
where $L$ is the Lambert product-log function, so that for all positive real $z$ and $u$ one has $z=L(u)\iff u=ze^z$; 
the approximate equalities here involving $\om$ hold under the natural assumption that $\om$ is small; recall \eqref{eq:om<.1} and the corresponding discussion.  
For $m=m_\om$, it follows from \eqref{eq:T'(m)=0} that 
\begin{equation*}
	\om 10^{d/(2m)}=\frac1{\frac d{2m}\,\ln10-1}
	=\frac1{L\big(\tfrac{1}{e \om}\big)}
	\approx
	\frac1{\ln\frac1{e\om}}.  
\end{equation*}
So, by \eqref{eq:c=}, \eqref{eq:m_om}, and \eqref{eq:T}, an appropriate choice of $c$ will be
\begin{equation}\label{eq:c approx}
c
\approx m_\om\Big(\frac12+\frac\ka{\om L\big(\tfrac{1}{e \om}\big)}\Big)
\approx \frac{\ln10}{1+L(\tfrac{1}{e \om})}\,\Big(\frac12+\frac\ka{\om L\big(\tfrac{1}{e \om}\big)}\Big)\,\frac d2
\approx\frac{\ka\ln10}{2\om\ln\frac1\om\,\ln\frac1{e\om}}\, d 	
\approx\frac{\ka\ln10}{2\om\ln^2\frac1\om}\, d, 
\end{equation}
whence  
\begin{equation}\label{eq:T sim}
\T(m)
=\frac{K\ln 10}{L(\tfrac{1}{e \om})}\,\frac d2
\approx \frac{K\ln 10}{2\ln\frac1\om}\,d. 	
\end{equation}

One can see that it is rather straightforward to find nearly optimal values of $m$ and $c$ for  \eqref{eq:series,c}. Formulas \eqref{eq:m_om} and \eqref{eq:c approx} show that such values of $m$ and $c$ are both approximately proportional to the desired number $d$ of digits of accuracy, and the proportionality coefficients are rather moderate in size unless $\om$ is very small. 

Usually, $T_\tau$ will be small compared to $T_f$ and $T_F$; then, $\om$ will be very small only if $T_F$ is much greater than $T_f$. 
In particular, in the rather typical case when $T_F\approx T_f>>T_\tau$, nearly optimal values of $m$ and $c$ are as follows: $m\approx0.55\,d$ and $c\approx1.5\,d$. 

Also, according to \eqref{eq:T sim} and \eqref{eq:om}, the needed time-per-digit $\T(m)/d$ when using 
formula \eqref{eq:series,c} will be approximately proportional to the ``aggregated'' time ``cost-per-value'' $K=T_f/2+T_\tau+2T_F$ for values of $f$, $\tau_{\cdot,\cdot}$, and $F$ -- again with a proportionality coefficient \big($\frac{\ln 10}{2\ln\frac1\om}$\big) that is moderate in size unless $\om$ is very small.  

Again, by \eqref{eq:m_om} and \eqref{eq:c approx}, for optimal values of $m$ and $c$ we have $m\asymp d$ and $c\asymp d$. Hence, the needed working accuracy $d_1$ is $d+O(\log_{10}d)$, which is $\sim d$ for large $d$. 
Mathematica keeps a rigorous record of the accuracy in its calculations with arbitrary-precision numbers. 
From the Mathematica tutorial/ArbitraryPrecisionNumbers:

\begin{quote}
[...] the Wolfram Language keeps track of which digits in your result could be affected by unknown digits in your input. It sets the precision of your result so that no affected digits are ever included. This procedure ensures that all digits returned by the Wolfram Language are correct, whatever the values of the unknown digits may be. 
\end{quote}  

\subsection{Choosing \texorpdfstring{$c$}{c} and \texorpdfstring{$m$}{m} in \texorpdfstring{\eqref{eq:series,c,EM}}{} for a desired accuracy}\label{c,m,EM}
In this subsection, we still be assuming the conditions stated in the paragraph containing \eqref{eq:S=Pi}. 
First here, instead of Corollary~\ref{cor:R bound}, let us use \eqref{eq:R_m^EM<} to get 
\begin{align}
	|R^\EM_{m,f_c}(\infty)|\le R^{\EM,*}_{m,c}:=R^{\EM,*}_{m,c;\mu,\la,a}
	&:=
	\frac{2.02\mu\,3^\la}{2m-2-\la}\,\frac{(2m-1)!}{(2\pi)^{2m-1}}\,
	\frac1{(c+a)^{2m-2-\la}} \label{eq:REM_c,infty} \\ 
	&=\Big(\frac{\ka^\EM m}{c}\Big)^{(2+o(1))m}  \notag 
\end{align}
for $m\to\infty$ and $c\ge(\ka^\EM+\vp)m$ for some fixed real $\vp>0$, 
where 
\begin{equation*}
	\ka^\EM:=\frac1{\pi e}\approx0.12. 
\end{equation*}
So, to ensure that 
\begin{equation*}
	|R^\EM_{m,f_c}(\infty)|\le\tfrac12\,10^{-d} 
\end{equation*}
for a large enough natural $d$, an appropriate choice of $c$ will be 
as follows: 
\begin{equation}\label{eq:c=,EM}
	c=\big\lceil c^\EM_{d,m}\big\rceil \approx\ka^\EM m10^{d/(2m)},  
\end{equation}
where 
$c^\EM_{d,m}=c^\EM_{d,m;\mu,\la,a}$ is the root $c$ of the equation $R^{\EM,*}_{m,c}=\frac12\,10^{-d}$. 

Let $T_f=T_f(d)$ still denote the time needed to compute one value of the function $f$. 

In comparison with dealing with formula \eqref{eq:series,c}, 
it will usually be much more difficult to find nearly optimal choices of $m$ and $c$ for the EM calculations according to \eqref{eq:series,c,EM} -- mainly because it is difficult to assess, especially in general terms, the time needed to compute the values of the derivatives 
\begin{equation}\label{eq:ders}
F^{(2j)}(c)=f^{(2j-1)}(c)\quad\text{for}\quad j=1,\dots,m-1,	
\end{equation}
needed in \eqref{eq:series,c}. 

For instance, suppose that the expression for the function $f$ contains the product $gh$ of two non-polynomial functions $g$ and $h$. Even if the time to compute the values $g^{(i)}(c)$ and $h^{(i)}(c)$ is not growing with $i$, still\ \;$\asymp j$ multiplications will be needed to compute the value $(gh)^{(j)}(c)=\sum_{i=0}^j\binom ji g^{(i)}(c)h^{(j-i)}(c)$ by the Leibniz formula. 
Here we are not taking into account the efforts to compute the $g^{(i)}(c)$'s, $h^{(i)}(c)$'s, and the binomial coefficients $\binom ji$.  
So, the time to compute all the derivatives \eqref{eq:ders} will be $\OG m^2$ -- which may be compared with the time $Km
$ to compute the values of one function, $F$, needed in \eqref{eq:series}.  

If the expression for the function $f$ contains the composition $g\circ h$ of two functions $g$ and $h$, then the derivatives of $g\circ h$ of required orders can be computed by a recursive version of the Fa\`a di Bruno formula: 
$(g\circ h)^{(j)}(c)=P_j$, where $P_j:=P_j(g_0,\dots,g_j,h_1,\dots,h_j)$ is a polynomial in $g_0,\dots,g_j,h_1,\dots,h_j$ of degree $j+1$ given recursively by the formulas $P_0(g_0):=g_0$ and 
\begin{equation*}
	P_{j+1}:=\Big(\sum_{i=0}^j \frac{\partial P_j}{\partial g_i}g_{i+1}\Big)h_1
	+\sum_{i=1}^j \frac{\partial P_j}{\partial h_i}h_{i+1}  
\end{equation*}
for $j=0,1,\dots$, 
and $g_i:=g^{(i)}(h(c))$ and $h_i:=h^{(i)}(c)$ for all $i$. 
So, if the functions $g$ and $h$ are non-polynomial, here as well the time to compute all the derivatives \eqref{eq:ders} will be $\OG m^2$ -- counting neither the time to compute the $g^{(i)}(h(c))$'s and $h^{(i)}(c)$'s, nor the time to compute the partial derivatives of the polynomials $P_j$. Note that the number of monomials in the polynomial $P_j$ equals the number of partitions of $j$, which is known \cite{hardy-ramanuj} to be asymptotic to $\frac1{4j\sqrt3}\,e^{c_0\sqrt j}$ as $j\to\infty$, where $c_0:=\pi\sqrt{2/3}=2.56\dots$. 
So, here the time to compute $f^{(j)}(c)$ will grow for large $j$ much faster than any power of $j$. Of course, such a difficult situation can quickly become much worse when the expression for $f$ is more complicated than just the product or the composition of two non-polynomial functions. 

To use the EM formula \eqref{eq:EM}--\eqref{eq:A^EM}, one also needs to compute the Bernoulli numbers $B_2,\dots,B_{2m-2}$. In Mathematica, for large enough $m$ this is done using the Fillebrown algorithm \cite{fillebrown}, which requires $\;\asymp m^2/\ln m$ multiplications of integers represented by $\O m\ln m$ bits, according to the analysis in \cite{fillebrown}; note the typo in the abstract in \cite{fillebrown}, where it should be $m^2/\log m$ in place of $m^2\log m$. 

We can now see that calculations by the EM formula will in most cases involve different computational strands of different degrees of growth in complexity, and those strands can be mixed in different proportions. 
The rates of growth within the different strands and the proportions between the strands will of course strongly depend on the function $f$, as well as on the desired number $d$ of digits of accuracy and the choice of $m$ and $c$. These rates and proportions may also significantly vary with $j$, the index in $f^{(j)}(c)$ and $B_{2j}$. 
However, in most case the time needed to compute the Bernoulli numbers $B_2,\dots,B_{2m-2}$ will be much less than that for the derivatives \eqref{eq:ders}.  

It should be clear from the above discussion that we can only offer a rough, tentative analysis pertaining to the choice of nearly optimal values of $m$ and $c$ in \eqref{eq:series,c,EM} in general -- assuming that the time to compute the derivatives \eqref{eq:ders} and the Bernoulli numbers $B_2,\dots,B_{2m-2}$ can be modeled in the considered range of values of $d$  
by an expression of the form $T_\der m^{2+\vp}$ for some positive real ``constants'' $\vp$ and $T_\der$;  $T_\der$ usually will, and $\vp$ may, depend on the desired accuracy $\frac12\,10^{-d}$.

So, the total time needed to compute the approximate value $\sum_{k=0}^{c-1} f(k)-G^\EM_{m,F}(c)$ of the generalized sum $\sum_{k\ge0}^\EM f(k)$ in \eqref{eq:series,c,EM} will be as follows:  
\begin{equation}\label{eq:T^EM}
\begin{aligned}
	\T^\EM&\approx \T^\EM(m):=\tT_f m10^{d/(2m)}+T_\der m^{2+\vp}, 
\end{aligned}	
\end{equation}
in view of \eqref{eq:c=,EM}, where $T_B$ is some positive constant and 
\begin{equation*}
	\tT_f:=\ka^\EM T_f.  
\end{equation*}

We can now find an approximately optimal value of $m$ by minimizing $\T^\EM(m)$ in $m$, and then choose $c$ in accordance with \eqref{eq:c=,EM}. 
It is easy to see that $\T^\EM(m)$ is strictly convex in $m\ge1$, $\T^\EM(m)\to\infty$ as $m\to\infty$, and 
\begin{equation}\label{eq:T',EM}
	\big(\T^\EM\big)'(m)=\tT_f 10^{d/(2m)}(1-d\ln10/(2m))+(2+\vp)T_\der m^{1+\vp}
\end{equation}
tends to $-\infty$ if $d\to\infty$ but $m$ stays bounded. So, $\T^\EM(m)$ is minimized in $m$ only when $m$ is the root of of the equation $\big(\T^\EM\big)'(m)=0$, and the minimizer $m$ tends to $\infty$; here and in the rest of this somewhat informal analysis, we assume that $d\to\infty$. 
If $m\to\infty$ but $d/m$ stays bounded, then, by \eqref{eq:T',EM}, $\big(\T^\EM\big)'(m)\to\infty$. So, for the minimizer $m$ of $\T^\EM(m)$, we have $m\to\infty$ and $d/m\to\infty$; hence, again by \eqref{eq:T',EM}, 
\begin{equation}\label{eq:10^}
	\tT_f 10^{d/(2m)}\,d\ln10/(2m)
	\sim (2+\vp)T_\der m^{1+\vp}, 
\end{equation}
whence $d/(2m)\sim
(1+\vp)\log_{10}m$ and  
\begin{equation}\label{eq:m sim}
	m\sim\frac{d}{2(1+\vp)\log_{10}d}.  
\end{equation}
It also follows from \eqref{eq:T^EM}, \eqref{eq:10^}, and \eqref{eq:m sim} that 
\begin{equation}\label{eq:T^EM sim}
	\T^\EM(m)\sim T_\der m^{2+\vp}\sim
	T_\der \Big(\frac{d}{2(1+\vp)\log_{10}d}\Big)^{2+\vp}
	\asymp \Big(\frac{d}{\log_{10}d}\Big)^{2+\vp},  
\end{equation}
which may be compared with $\T(m)	\asymp d$ according to \eqref{eq:T sim}. 

\section{Parallelization and memory use}\label{par}
It was made clear in Section~\ref{c,m} that both $m$ and $c$ should be large enough in order for the remainders $R_{m,f_c}(\infty)$ and $R^\EM_{m,f_c}(\infty)$ in \eqref{eq:series,c} and  \eqref{eq:series,c,EM} to be small. It is therefore important to consider whether calculations of the terms $\sum_{k=0}^{c-1} f(k)$ and $G_{m,F}(c)$ in \eqref{eq:series,c} and $\sum_{k=0}^{c}f(k)$ and $G^\EM_{m,F}(c+1)$ in \eqref{eq:series,c,EM} can be parallelized, for such large values of $m$ and $c$. 

Let $k_*$ denote the number of computer cores available for parallel computation. 
The terms $\sum_{i=0}^{c-1} f(i)$ in \eqref{eq:series,c} and $\sum_{i=0}^{c}f(i)$ in \eqref{eq:series,c,EM} can be easily computed about $k_*$ times as fast in parallel on the $k_*$ cores as on one such core or on a single-core CPU. To do such a parallel calculation, one can partition the index range, say $R:=\{0,\dots,c-1\}$, into $k_*$ sub-ranges $R_1,\dots,R_{k_*}$, each approximately of same size $\approx c/k_*$, compute each of the corresponding parts $S_k:=\sum_{i\in R_k} f(i)$ of the sum $\sum_{i=0}^{c-1} f(i)$ 
on (say) core $k$ of the $k_*$ parallelly engaged cores, and then quickly add the partial sums $S_1,\dots,S_{k_*}$. In Mathematica, such a calculation can be done by issuing the command \\ 
\verb!ParallelSum[SetAccuracy[f[i],d1],{i,0,c-1},Method->"CoarsestGrained"]!, \\ 
where \verb!d1! is the accuracy set for each summand $f(i)$. 
In distinction with \verb!"FinestGrained"!, the \verb!"CoarsestGrained"! method minimizes the overhead caused by data interchange between the controlling process (master kernel) and the subordinate processes (subkernels, running on available parallel cores). 

However, there seems to be no way in general to parallellize the calculation of the consecutive derivatives  
$F^{(2j)}(c)=f^{(2j-1)}(c)$ for $j=1,\dots,m-1$ 
in the expression 
of the term 
$G^\EM_{m,F}(c+1)$ in \eqref{eq:series,c,EM}, defined by \eqref{eq:G^EM_m}. 
Usually, this will be the bottleneck in using the EM summation formula for high-precision calculations. 

It is possible to compute the Bernoulli numbers in parallel \cite{harvey10}. However, such parallelization is not done in Mathematica, and it has not been done in the computer experiments to be described in the examples in Section~\ref{examples} -- 
mainly because, as was noted, except for a very narrow set of functions $f$, the time needed to compute the Bernoulli numbers $B_2,\dots,B_{2m-2}$ will be much less than that for the derivatives \eqref{eq:ders}. 
The usually very large amount of time needed to compute those derivatives certainly precludes values of $m>\frac12\,10^4$. On the other hand, for values of $m\le\frac12\,10^4$, the only execution time reported in \cite{harvey10} is for $B_{2m}=B_{10^4}$, with $m=\frac12\,10^4$ -- only for a one-core calculation, which took $0.25$ sec (on a 16-core 2.6 GHz AMD Opteron (64-bit) machine with
96 GB RAM, running Ubuntu Linux). In comparison, it took Mathematica just about $0.05$ sec 
to compute the same number, $B_{10^4}$ (on a roughtly comparable 12-core 2.30 GHz Intel Xeon (64-bit) machine with
128 GB RAM, running Windows 7). 
The programming in \cite{harvey10} was done in C++. 
Relevant here may be the following 
quote from \cite{c++vsMathca} : 
`` 
Mathematica is only about three times slower than C++, but only after a considerable
rewriting of the code to take advantage of the peculiarities of the language. The baseline
version of our algorithm in Mathematica is considerably slower.''  
The code in \cite{harvey10} may be significantly more complicated than the code for the Bernoulli numbers used in Mathematica, so that the overhead caused by the complexity of the code used in \cite{harvey10} may be relatively too large for not too large values of $m$.  
Note also that it usually takes about $2$ to $4$ sec to launch several kernels in Mathematica. For all these reasons, it seems to make little (if any) sense to parallelize the calculation of the Bernoulli numbers $B_2,\dots,B_{2m-2}$ when $m$ is not very large.  

In contrast with the term $G^\EM_{m,F}(c+1)$ in \eqref{eq:series,c,EM}, the calculation of the term $G_{m,F}(c)$ in \eqref{eq:series,c} can be almost fully parallelized. For large $m$, of the three expressions \eqref{eq:G}--\eqref{eq:G_m,alt2} for $G_{m,F}(c)$, 
one should use the one in \eqref{eq:G_m,alt2} as containing the fewest number ($m$) of summands, versus $\sim2m$ summands in \eqref{eq:G_m,alt1} and $\sim m^2/2$ summands in \eqref{eq:G}; cf.\ Remark~\ref{rem:A_m-alt}. 
Next, the values of the function $F$ in
\eqref{eq:G_m,alt2} (with $n=c$) can of course be easily computed in parallel, on several cores. 

Also, with a little trick, the calculations of the coefficients 
$$\tau_j:=\tau_{m,j}$$ 
in \eqref{eq:G_m,alt2} can be almost entirely parallelized, and this can be done so that very little memory space is needed. 
Indeed, assume for simplicity that the large natural number $m$ is even. Let $0=m_0,m_1,\dots,m_{k_*}=m$ be even numbers such that 
\begin{equation}\label{eq:ell}
	\ell_k:=m_k-m_{k-1}\approx m/k_*, 
\end{equation}
so that $\ell_k$ is even; here and in what follows, $k$ is an arbitrary number in the set $\{1,\dots,k_*\}$. 
\big(Note that in this section the meaning of $m_0$ is quite different from that in other parts of this paper -- such as formula \eqref{eq:series,c}, for example.\big) 
 
A particular way to specify values of the $\ell_k$'s and $m_k$'s is as follows. Let $q$ and $r$ denote the nonnegative integers that are, respectively, the quotient and the remainder of the division of $m/2$ by $k_*$, so that $m/2=k_* q+r$ and $r<k_*$. 
Let  
\begin{equation}\label{eq:ell_k,m_k}
	\ell_k:=2\big(q+\ii\{k\le r\}\big)\quad\text{and}\quad m_k:=2(kq+k\wedge r),  
\end{equation}
where $\ii\{A\}$ denotes the indicator of an assertion $A$.  
Then indeed the $m_k$'s are even numbers, $m_{k_*}=m$ and condition \eqref{eq:ell} holds. Let us also assume the quite natural condition that $m$ is large enough so as 
\begin{equation}\label{eq:ell ge 2}
	m\ge4k_*\quad\text{and hence}\quad \ell_k\ge4\text{\quad and\quad }\ell_k/2-1\ge1.  
\end{equation}

By \eqref{eq:G_m,alt2} with $\tau_j:=\tau_{m,j}$, 
\begin{gather}
	G_{m,F}(c)=\sum_{k=1}^{k_*}\Si_k,\quad\text{where}\quad
	\Si_k:=\sum_{j=m_{k-1}+1}^{m_k}\tau_j H_j=\Si_k^\ev+\Si_k^\od, \label{eq:G(c)=} \\ 
	H_j:=F(c-j/2)+F(c+j/2-1)\ii\{j\ge2\}, \label{eq:H_j=} \\ 
	\Si_k^\ev:=\sum_{i=0}^{\ell_k/2-1}\tau_{j_{k,i}}H_{j_{k,i}}, \quad
	\Si_k^\od:=\sum_{i=0}^{\ell_k/2-1}\tau_{j_{k,i}-1}H_{j_{k,i}-1}, \notag \\ 
	j_{k,i}:=m_k-2i; \label{eq:j} 
\end{gather}
here, the superscripts ${}^\ev$ and ${}^\od$ allude to ``even'' and ``odd'', respectively. 

The key in computing $\Si_k^\ev$ and $\Si_k^\od$ is the simple identities  
\begin{equation}
	\Si_k^\ev=\tau_{m_k+2}\,\eta_{k,\ell_k/2-1}^\ev +\si_{k,\ell_k/2-1}^\ev\quad\text{and}\quad
	\Si_k^\od=\tau_{m_k+1}\,\eta_{k,\ell_k/2-1}^\od +\si_{k,\ell_k/2-1}^\od, \label{eq:key}
\end{equation}
where 
\begin{gather}
\eta_{k,s}^\ev:=\sum_{i=0}^s H_{j_{k,i}},\quad \si_{k,s}^\ev:=\sum_{i=0}^s\th_{k,i}^\ev H_{j_{k,i}}, \label{eq:eta,si ev}
\\ 
\eta_{k,s}^\od:=\sum_{i=0}^s H_{j_{k,i}-1},\quad \si_{k,s}^\od:=\sum_{i=0}^s\th_{k,i}^\od H_{j_{k,i}-1}, \label{eq:eta,si od} \\ 
	\th_{k,i}^\ev:=\tau_{j_{k,i}}-\tau_{m_k+2}=\ga_{j_{k,i}}+\ga_{j_{k,i}+2}+\dots+\ga_{m_k}, \label{eq:thev:=} \\
	\th_{k,i}^\od:=\tau_{j_{k,i}-1}-\tau_{m_k+1}=\ga_{j_{k,i}-1}+\ga_{j_{k,i}+1}+\dots+\ga_{m_k-1},   \label{eq:thod:=}  
\end{gather}
and $\ga_j:=\ga_{m,j}$; 
the last equalities in the last two lines of the above display follow 
by \eqref{eq:tau_j}, which in particular implies $\tau_{m_k+2}=0=\tau_{m_k+1}$ for $k=k_*$.  
So, 
\begin{gather}
	\th_{k,0}^\ev=\tga_{k,0}^\ev;\quad \th_{k,i}^\ev=\th_{k,i-1}^\ev+\tga_{k,i}^\ev\quad \forall i=1,\dots,\ell_k/2-1, \label{eq:thev} \\ 
	\th_{k,0}^\od=\tga_{k,0}^\od;\quad \th_{k,i}^\od=\th_{k,i-1}^\od+\tga_{k,i}^\od\quad \forall i=1,\dots,\ell_k/2-1, \label{eq:thod}	
\end{gather}
where  
\begin{gather}\label{eq:gaev,gaod}
	\tga_{k,i}^\ev:=\ga_{j_{k,i}}=\frac{\trho_{k,i}^\ev}{j_{k,i}},\quad  
	\tga_{k,i}^\od:=\ga_{j_{k,i}-1}=\frac{\trho_{k,i}^\od}{j_{k,i}-1}, \\ 
	\trho_{k,i}^\ev:=\rho_{j_{k,i}}\quad \trho_{k,i}^\od:=\rho_{j_{k,i}-1},  \notag 
\end{gather}
with $\rho_\cdot$ defined by \eqref{eq:rho:=}. 
Recalling that $m_k$ is even and using also \eqref{eq:ga_j}, \eqref{eq:j}, and \eqref{eq:rho iter}, we have 
\begin{gather}
	\trho_{k,0}^\ev=-2\binom{2m}{m+m_k}\Big/ \binom{2m}{m} \quad\text{and}\quad  
	\quad \trho_{k,0}^\od=\trho_{k,0}^\ev\frac{m+m_k}{m_k-m-1}, 
	\label{eq:rho0}\\ 
	\trho_{k,i}^\ev=\trho_{k,i-1}^\od\frac{m+j_{k,i}+1}{j_{k,i}-m}
	\quad\text{and}\quad 
	\trho_{k,i}^\od=\trho_{k,i}^\ev\frac{m+j_{k,i}}{j_{k,i}-m-1}
\quad\forall i=1,\dots,\ell_k/2-1. \label{eq:rho} 
\end{gather}

For each $k=1,\dots,k_*$, one can compute $\th_{k,\ell_k/2-1}^\ev, \eta_{k,\ell_k/2-1}^\ev,  \si_{k,\ell_k/2-1}^\ev, \th_{k,\ell_k/2-1}^\od, 
\eta_{k,\ell_k/2-1}^\od,\si_{k,\ell_k/2-1}^\od$ on core $k$. This is done in $\ell_k/2$ steps, indexed by $i=0,\dots,\ell_k/2-1$. At the initial step $i=0$, one can compute $\trho_{k,0}^\ev$ and  $\trho_{k,0}^\od$ by formula \eqref{eq:rho0}, then $\tga_{k,0}^\ev=\trho_{k,0}^\ev/m_k$ and  $\tga_{k,0}^\od=\trho_{k,0}^\od/(m_k-1)$ by \eqref{eq:gaev,gaod} and \eqref{eq:j}, then $\th_{k,0}^\ev$ and  $\th_{k,0}^\od$ by \eqref{eq:thev} and \eqref{eq:thod}, and finally 
\begin{equation}\label{eq:step 0}
\eta_{k,0}^\ev=H_{m_k},\quad \eta_{k,0}^\od=H_{m_k-1},\quad 
\si_{k,0}^\ev=\tga_{k,0}^\ev H_{m_k},\quad \si_{k,0}^\od=\tga_{k,0}^\od H_{m_k-1}	
\end{equation}
in accordance with \eqref{eq:eta,si ev}, \eqref{eq:eta,si od}, \eqref{eq:thev}, \eqref{eq:thod}, and \eqref{eq:j}. 
After that, at each step \break 
$i=1,\dots,\ell_k/2-1$, one computes $\trho_{k,i}^\ev$ and  $\trho_{k,i}^\od$ (in this order) by formula \eqref{eq:rho}, then $\tga_{k,i}^\ev=\trho_{k,i}^\ev/j_{k,i}$ and	$\tga_{k,i}^\od=\trho_{k,i}^\od/(j_{k,i}-1)$ by \eqref{eq:gaev,gaod}, then $\th_{k,i}^\ev$ and  $\th_{k,i}^\od$ by the recursions in \eqref{eq:thev} and \eqref{eq:thod}, and finally, in accordance with \eqref{eq:eta,si ev} and \eqref{eq:eta,si od}, with $j:=j_{k,i}$,  
\begin{equation}\label{eq:step i}
\begin{gathered}
\eta_{k,i}^\ev=\eta_{k,i-1}^\ev+H_{j},\quad 
\si_{k,i}^\ev=\si_{k,i-1}^\ev+\th_{k,i}^\ev H_{j}, \\  
\eta_{k,i}^\od=\eta_{k,i-1}^\od+H_{j-1},\quad 
\si_{k,i}^\od=\si_{k,i-1}^\od+\th_{k,0}^\od H_{j-1}. 	
\end{gathered}
\end{equation} 

Then, for each $k=1,\dots,k_*$, the six computed values of 
\begin{equation}\label{eq:output}
\th_{k,\ell_k/2-1}^\ev,\quad \th_{k,\ell_k/2-1}^\od,\quad \eta_{k,\ell_k/2-1}^\ev,\quad \eta_{k,\ell_k/2-1}^\od,\quad \si_{k,\ell_k/2-1}^\ev,\quad  
\si_{k,\ell_k/2-1}^\od	
\end{equation}
are transmitted from kernel $k$ (running on core $k$) to the master kernel. 
By \eqref{eq:thev:=}, \eqref{eq:thod:=}, \eqref{eq:j}, and \eqref{eq:ell}, 
$\th_{k,\ell_k/2-1}^\ev=\tau_{m_{k-1}+2}-\tau_{m_k+2}$ and 
$\th_{k,\ell_k/2-1}^\od=\tau_{m_{k-1}+1}-\tau_{m_k+1}$. 
So, the values of  
\begin{equation}\label{eq:taus}
\tau_{m_k+2}=\th_{k+1,\ell_{k+1}/2-1}^\ev+\dots+\th_{k_*,\ell_{k_*}/2-1}^\ev
\quad\text{and}\quad
	\tau_{m_k+1}=\th_{k+1,\ell_{k+1}/2-1}^\od+\dots+\th_{k_*,\ell_{k_*}/2-1}^\od
\end{equation} 
for all $k=1,\dots,k_*$ 
can be very quickly computed in the master kernel. 
Now the calculation of $G_{m,F}(c)$ can be very quickly completed by \eqref{eq:G(c)=} and 
\eqref{eq:key}. 

We see that the bulk of this calculation is to compute -- for each $k=1,\dots,k_*$, on core $k$ -- the values \eqref{eq:output}, which takes $\ell_k/2\approx m/(2k_*)$ steps, in view of \eqref{eq:ell}. 

Each of these steps -- except for the initial step, corresponding to $i=0$ -- involves just 6 operations of multiplication/division and 8 operations of addition/subtraction 
of real numbers of a given accuracy (say $d_1$), not counting the additions/subtractions of $1$ in \eqref{eq:gaev,gaod} and \eqref{eq:rho}. 
Moreover, at each step $i=1,\dots,\ell_k/2-1$, there is no need to keep in memory the values $\th_{k,i-1}^\ev,\th_{k,i-1}^\od,\trho_{k,i-1}^\od,\eta_{k,i-1}^\ev,\eta_{k,i-1}^\od,\si_{k,i-1}^\ev,\si_{k,i-1}^\od$ computed at the previous step $i-1$. 
For instance, the recursion in \eqref{eq:thev} can be realized in programming code as $\th_{k,\cdot}^\ev\leftarrow\th_{k,\cdot}^\ev+\tga_{k,i}^\ev$, with the single value of $\th_{k,\cdot}^\ev$ updated in the memory for each $i=1,\dots,\ell_k/2-1$. 

Therefore, all the steps $i=1,\dots,\ell_k/2-1$ together require memory storage of only $\asymp k_*$ real numbers of the accuracy $d_1$ -- in addition to the memory needed to compute and store the current value of $H_\cdot$, defined in \eqref{eq:H_j=}. 


The most significant computational difference of the 
initial step $i=0$ from steps $i=1,\dots,\ell_k/2-1$ is the calculation of the binomial coefficients in \eqref{eq:rho0}. This can be done very quickly using a version of the divide-and-conquer algorithm, as it is done e.g.\ by Mathematica. However, then the needed memory storage will be much greater than $\asymp k_*$ of real numbers of the accuracy $d_1$. 

Yet, it is not hard to figure out how to compute the values of $\trho_{k,0}^\ev$ in \eqref{eq:rho0} very fast, in an amount of time negligible as compared with the total time to compute $G_{m,F}(c)$, and still requiring memory storage of only $\asymp k_*$ real numbers of the accuracy $d_1$. 
Indeed, letting $\trho_{0,0}^\ev:=-2$, it is easy to see that for each $k=1,\dots,k_*$ 
\begin{equation*}
	\trho_{k,0}^\ev=\trho_{k-1,0}^\ev\, N_k, 
\end{equation*}
where 
\begin{equation}\label{eq:N_k}
N_k:=N_{k,\ell_k}:=\prod_{j=m_{k-1}+1}^{m_k}\nu_j=\prod_{i=1}^{\ell_k}\nu_{m_{k-1}+i}\quad\text{with}\quad \nu_j:=\frac{j-m-1}{m+j},  	
\end{equation}
can be computed recursively in kernel $k$ in $\ell_k\approx m/k_*$ steps, as follows: 
\begin{equation*}
	N_{k,1}=\nu_{m_{k-1}+1};\quad
	N_{k,i}=N_{k,i-1}\nu_{m_{k-1}+i}\quad\forall i=2,\dots,\ell_k. 
\end{equation*}
Similarly to a previous comment, note here that there is no need to keep in memory the value $N_{k,i-1}$ computed at the previous step $i-1$. 

Thus -- in addition to the calculation of the values of $H_1,\dots,H_m$ in accordance with \eqref{eq:H_j=} --  
the calculation of $G_{m,F}(c)$ requires only $\asymp m/k_*$ arithmetical operations (a.o.'s) in each of the parallel $k_*$ kernels plus only $\asymp k_*$ a.o.'s in the master kernel, and the required memory storage is only for $\asymp k_*$ real numbers of the accuracy $d_1$.

\section{Examples}\label{examples}

Four examples will be considered in this section, to illustrate applications of the Alt formula to the summation of possibly divergent series and measure its performance in terms of the execution time and memory use, in comparison with the performance of the EM formula -- and also with the performance of Mathematica and the Richardson extrapolation process (REP) where the latter tools are applicable. 
Applications of the REP to summation rely on asymptotic expansions, which appear to have been designed on an ad hoc basis, depending on the function $f$, as was e.g.\ done in \cite[Appendix~E.2]{sidi}. 

In the first of these examples, considered in Subsection~\ref{sqrt}, both the function $f$ and its antiderivative $F$ are given by rather simple expressions, but high-order derivatives of $f$ are significantly more complicated. As expected, here the Alt formula greatly outperforms the EM one, both in the execution time and memory use. 

In the example considered in Subsection~\ref{erfInv}, both the function $f$ and its antiderivative $F$ are complicated, and high-order derivatives of $f$ are significantly more complicated yet. Here the advantages of the Alt formula over the EM one, especially in the execution time, are even more pronounced. In particular, Alt calculations yield $500$ (decimal) digits of the results in about $0.8$ sec, whereas it takes the EM formula over half an hour to produce just $150$ digits.  

In the example considered in Subsection~\ref{zeta} -- where an array of values of the Hurwitz generalized zeta function is computed with various degrees of accuracy -- 
the function $f$, its antiderivative $F$, and high-order derivatives of $f$ are more or less equally complicated. Here the execution time and memory use numbers of the Alt and EM formulas are within a factor of $2$ from each other for up to $d=16\times10^3$ digits of accuracy of the result of the calculations, with the EM formula being better for the calculations without parallelization, and the comparison reversed when the calculations are parallelized. In this example, calculations of the coefficients $\tau_{m,r}$ and $B_{2j}$ in the Alt and EM formulas take relatively small fractions of the execution time and memory use. However, for very large values of $d$ it is expected that here too the Alt formula will significantly outperform the EM one even without parallelization, because of the comparatively fast growth in $d$ of the execution time and memory use needed to compute the Bernoulli numbers $B_{2j}$. 
As for the comparison with Mathematica, both the EM formula and, especially,
the Alt one in the parallelized version perform much faster than the (non-parallelizable) built-in Mathematica command \texttt{HurwitzZeta[]};  
however, \texttt{HurwitzZeta[]} is better in terms of memory use than our implementations of the EM and Alt  formulas. 
The Richardson extrapolation scheme (REP) is also applicable here, but its performance in this situation is much inferior even to that of Mathematica's \texttt{HurwitzZeta[]}, let alone the EM and Alt formulas.   

Finally, in the example considered in Subsection~\ref{euler}, the function $f$ and its high-order derivatives of $f$ are simple, but the antiderivative $F$ of $f$ is significantly more complicated. Here the comparison between the EM and Alt formulas is somewhat similar to that in the previous example; see Subsection~\ref{euler} for details. Again, for very, very large values of the number $d$ of the required digits of accuracy of the result it is expected that in this example too the Alt formula will significantly outperform the EM one even without parallelization. In this situation as well, the REP's performance is much inferior to that of the EM and Alt formulas. 

\bigskip

Recall now again that, to use the EM formula, one will usually need to have an antiderivative $F$ of $f$ in tractable/computable form. Then $f$ (being the derivative of $F$) will usually be of complexity no less than that of $F$. So, the latter two of the four mentioned examples are among the relatively few settings least favorable to the Alt formula, in comparison with the EM one. 
It may therefore be noted that such unfavorable examples are rather disproportionally represented among the four examples considered in Section~\ref{examples}. 
Yet, even then, the Alt formula can be expected to significantly outperform the EM one when very high accuracy is needed. In other cases, Alt will usually be significantly better than EM even for moderately high accuracy.     

The annotated code and details of the calculations in the mentioned examples are given in the Mathematica notebooks and their pdf images in 
the zip file AltSum.zip at the Selected Works site \url{https://works.bepress.com/iosif-pinelis/}. 
Our calculations were verified by comparing with each other the computed decimal approximations to the values of the corresponding generalized sums, $\sum_{k\ge0}^\Alt f(k)$ and $\sum_{k\ge0}^\EM f(k)$, and also, in the examples in Subsections~\ref{zeta} and \ref{euler}, by comparing those decimal approximations with the ones computed using the built-in Mathematica commands \texttt{HurwitzZeta[]} and \texttt{EulerGamma}. 

\subsection{Example: simple \texorpdfstring{$f$}{f} and \texorpdfstring{$F$}{f}, complicated derivatives \texorpdfstring{$f^{(2j)}$}{}}\label{sqrt}

Here we consider summing the divergent series $\sum_{k=0}^\infty f(k)$ according to Corollaries~\ref{cor:series} and \ref{cor:seriesEM}, where 
\begin{equation}\label{eq:sqrt f}
	f(x):=\frac{3x^3}{\sqrt{x^2 + 1}}.  
\end{equation}
Then for an antiderivative $F$ of $f$ one can take the function given by the formula 
\begin{equation}\label{eq:sqrt F}
	F(x)=(x^2-2)\sqrt{x^2 + 1}. 
\end{equation}

The remainder $R_{m,f_c}(\infty)$ in \eqref{eq:series,c} will be bounded according to \eqref{eq:R_c,infty}. Let 
\begin{equation}\label{eq:a,la,mu--ex2}
	a=-2\quad\text{and}\quad\la=2. 
\end{equation}
Then for all $z\in\Pi_{-a}^+$ (recall here the definition in \eqref{eq:S=Pi}) we have $|z|\ge\Re z\ge-a=2$ and $|z-1|\ge|z|-1\ge2-1=1$, whence 
$$|z^2+1|\ge|z|^2-1=|z|^2(1-1/|z|^2)\ge|z|^2(1-1/2^2)=\tfrac34\,|z|^2,$$ 
$|z|\le|z-1|+1\le2|z-1|$, 
and 
$$|f(z)|\le\frac{3|z|^3}{\sqrt{\frac34}\,|z|}=2\sqrt{3}|z|^2
\le8\sqrt{3}|z-1|^2
=8\sqrt{3}|z+a+1|^2,$$ so that condition \eqref{eq:|f|<} holds with $\mu=8\sqrt{3}=13.85\dots$. 
Using Remark~\ref{rem:mu}, this possible value for $\mu$ can be a bit improved, to the optimal value 
\begin{equation}\label{eq:mu=}
	\mu=\frac{24}{\sqrt5}=10.73\dots;  
\end{equation}
for details of this calculation, see Mathematica notebook 
\verb!\sqrt\mu.nb! and its pdf image \verb!\sqrt\mu.pdf!;  
\begin{center}
\framebox{
	\parbox{4.in}{all the files and folders mentioned here and in what follows are contained in 
the mentioned earlier zip file AltSum.zip at the Selected Works site \url{https://works.bepress.com/iosif-pinelis/}.
}
}
\end{center}
The value $\mu$ in \eqref{eq:mu=} will be assumed in the rest of the consideration of this example. One may note that with this value of $\mu$ the inequality in \eqref{eq:|f|<} turns into the equality for $z=-a=2$. 


So, \eqref{eq:R_c,infty} will hold if 
\begin{equation}\label{eq:m,c>}
\text{$m\ge2$\quad and\quad $c\ge2+(m+3)/2=(m+7)/2$.}	
\end{equation}

By Remark~\ref{rem:m_0}, Corollaries~\ref{cor:series} and \ref{cor:seriesEM} will hold with 
\begin{equation}\label{eq:m>m0=2}
	m>m_0=2, 
\end{equation}
which will also be assumed in this example. 
Then, by \eqref{eq:G} and \eqref{eq:G^EM_m}, one has the following, rather complicated expressions for $G_{m_0,F}(n)$ and $G^\EM_{m_0,F}(n)$: 
\begin{equation*}
	G_{m_0,F}(n)=
	\frac{(8 n^2-8n-14) \sqrt{ n^2- n+5/4}-(n^2-2 n-1)\sqrt{n^2-2 n+2}-(n^2-2) \sqrt{n^2+1}}6
\end{equation*}
and 
\begin{equation*}
	G^\EM_{m_0,F}(n)=\frac{4 n^6-18 n^5+32 n^4-22 n^3-15 n^2+34 n-23}{4 (n^2-2 n+2)^{3/2}}. 
\end{equation*} 
Writing the terms of the form $(n^2+Bn+C)^p$ in these expressions for $G_{m_0,F}(n)$ and $n^i G^\EM_{m_0,F}(n)$ as \break 
$n^{2p}(1+B\vp+C\vp^2)^p$ with $\vp:=\frac1n$, and then expanding $(1+B\vp+C\vp^2)^p$ into the powers of $\vp$,
%
we have 
\begin{equation*}
	G_{m_0,F}(n)=n^3-\frac{3 n^2}{2}-n+\frac{3}{4}-\frac{9}{8 n}-\frac{9}{16 n^2}+\frac{1}{8 n^3}+\frac{15}{32
   n^4}+\frac{31}{128 n^5}+O\Big(\frac{1}{n^6}\Big) 
\end{equation*}
and 
\begin{equation*}
G^\EM_{m_0,F}(n)=	n^3-\frac{3 n^2}{2}-n+\frac{3}{4}-\frac{9}{8 n}-\frac{9}{16 n^2}+\frac{1}{8 n^3}+\frac{15}{32
   n^4}+\frac{19}{128 n^5}+O\Big(\frac{1}{n^6}\Big). 
\end{equation*}
So, for this example, we confirm the asymptotic relation \eqref{eq:same} and see that here one can replace $G_{m,F}(n)$ and $G^\EM_{m,F}(n)$ in \eqref{eq:series,c} and \eqref{eq:series,c,EM} by 
$$n^3-\frac{3 n^2}{2}-n+\frac{3}{4}.$$ 
We can also see that here in fact $G_{m_0,F}(n)-G^\EM_{m_0,F}(n)=O\big(\frac{1}{n^5}\big)$.

In this example, concerning the quantities $T_f$, $T_F$, and $T_\tau$ introduced in Subsection~\ref{c,m,alt}, we have $T_f\approx T_F>>T_\tau$, so that, by \eqref{eq:om}, \eqref{eq:om<.1},  \eqref{eq:m_om}, and \eqref{eq:c approx}, 
\begin{equation*}
\om\approx\frac\ka{2.5}\approx0.1,\quad m=m_\om\approx0.53d,\quad\text{and}\quad 
c=c_{d,m_\om}\approx m_\om\Big(\frac12+\frac\ka{\om L\big(\tfrac{1}{e \om}\big)}\Big)\approx1.55 d. 
\end{equation*}

Accordingly, we will take here 
\begin{equation}\label{eq:m=}
	m=2\lceil \tfrac12\,0.53d\rceil,  
\end{equation}
for $m$ to be an even integer, as was assumed in Section~\ref{par}. 
However, $c$ will be taken to be, more accurately, the least integer majorizing the solution for $c$ of the equation $R^*_{m,c}=\frac12\,10^{-d}$, where $R^*_{m,c}$ is as in \eqref{eq:R_c,infty} and $d$ is the desired number of known digits of the generalized sum $\sum_{k\ge0}^\Alt f(k)$ in \eqref{eq:series,c} after the decimal point: 
\begin{equation}\label{eq:cc=}
	c=\Big\lceil \frac{m+1}2-a
	+\Big(2\times10^d\,\frac{1.001\pi\mu3^\la \ka^{2m}m^{2m+1}}{(2m+1)(2m-1-\la)}\Big)^{1/(2m-1-\la)} 
	\Big\rceil, 
\end{equation}
with $\ka$ as defined in \eqref{eq:ka}. 
Then conditions \eqref{eq:m,c>} and \eqref{eq:m>m0=2} will hold if $d\ge6$ (say), which will of course be the case. 

Table~\ref{tab:sqrt} presents the wall-clock time (in sec) and the required memory use (in bytes) to compute $d$ digits of the generalized sum $\sum_{k\ge0}^\Alt f(k)$ in \eqref{eq:series,c}, based on the Alt summation formula \eqref{eq:}, for each of the selected values of $d$, ranging from $2\times10^3$ to $32\times10^4$. 
It appears sensible enough to present the memory use per digit of accuracy. For instance, the recorded memory use $20\,d$ in Table~\ref{tab:sqrt} for $d=2\times10^3$ and a one-core calculation means that only about $20$ bytes of memory were used per one decimal digit of the $2\times10^3$ correct computed digits of $\sum_{k\ge0}^\Alt f(k)$. One may recall here that one byte can represent about $\log_{10}(2^8)\approx2.41$ decimal digits.  
The calculations were done  
on one core and also on $12$ parallel cores of a 12-core 2.30GHz Intel Xeon (64-bit) machine with 128 GB RAM, running Windows 7. The values of $m$ and $c$ in \eqref{eq:series,c} were chosen, for each selected value of $d$, in accordance with the above discussion. 
Deployment of multiple parallel cores may take a few seconds; so, it is justified only if the execution time is at least a couple of seconds. Therefore, the very small amount of time of the calculation for $d=2\times10^3$ on parallel cores recorded in Table~\ref{tab:sqrt} shows that, for $f$ as in this example, it may make sense to use parallelization only for $d>2\times10^3$. 
One can see that the ratios of the one-core times to the corresponding 12-core times in Table~\ref{tab:sqrt} tend to get close to $12$, the number of parallel cores, as $d$ increases from $2\times10^3$ to $16\times10^4$. 
The calculation on one core for $d=32\times10^4$ has not been done, since it is expected to take a very long time, about $1600\ \text{sec}\times12\approx5$ hrs. 
The annotated code and details of these calculations are given in Mathematica notebooks \verb!\sqrt\ASqrt.nb! and \verb!\sqrt\AParSqrt.nb! (for the one-core and $12$-core Alt calculations, respectively) and their pdf images \verb!\sqrt\ASqrt.pdf! and \verb!\sqrt\AParSqrt.pdf!.   

\setlength\tabcolsep{3.pt}
\renewcommand{\arraystretch}{1.5}
\begin{table}[h]
	\centering
  \begin{tabular}{c|m{.8cm}||c|c|c|c|c|c|c}
\raisebox{-6pt}[-5pt][-5pt]{\large Alt} & \multicolumn{1}{c||}{\raisebox{-6pt}[-5pt][-5pt]{\# cores}} &\multicolumn{7}{c}{$d$} \\
\cline{3-9}
    \ce $ $ & & $2\times 10^3$ & $10^4$ & $2\times 10^4$ & $4\times 10^4$ & $8\times 10^4$ & $16\times 10^4$ & $32\times 10^4$  \\ 
    \hline\hline
     \ce time & \ce 1 & $0.24$ & $5.5$ & $28$ & $150$ & $730$ & $3600$ & $-$  
    \\ 
   \hline
     \ce memory & \ce 1  & $20\,d$  & $12\,d$ & $11\,d$ & $11\,d$ & $11\,d$ & $11\,d$ & $-$  
        \\  
   \hline\hline          
     \ce time & \ce 12  & $0.10$  & $0.72$ & $3.1$ & $15$ & $68$ & $330$ & $1600$  
    \\ 
    \hline    
   \ce memory & \ce 12  & $360\,d$  & $240\,d$ & $220\,d$ & $210\,d$ & $220\,d$ & $210\,d$ & $210\,d$ 
    \\ 
  \end{tabular}
  	\caption{The wall-clock time (in sec) and memory use (in bytes) to compute $d$ digits of the generalized sum $\sum_{k\ge0}^\Alt f(k)$ in \eqref{eq:series,c}, on one core and on $12$ parallel cores, 
  	for 
  	$f$ as in \eqref{eq:sqrt f} and 
  	each of the listed values of $d$ -- by using \eqref{eq:series,c}.}
	\label{tab:sqrt}
\end{table}

\setlength\tabcolsep{3.pt}
\renewcommand{\arraystretch}{1.5}
\begin{table}[h]
	\centering
  \begin{tabular}{c|m{.8cm}||c|c|c|c}
\raisebox{-6pt}[-5pt][-5pt]{\large EM}  & \multicolumn{1}{c||}{\raisebox{-6pt}[-5pt][-5pt]{\# cores}} &\multicolumn{4}{c}{$d$} \\
\cline{3-6}
    \ce $ $ & & $\frac12\times10^3$ & $10^3$ & $2\times10^3$ & $4\times 10^3$ \\ 
    \hline\hline
     \ce time & \ce 1 & $0.69$ & $3.0$ & $81$ & $1600$  
    \\ 
   \hline
     \ce memory & \ce 1  & $13\times10^3\,d$  & $21\times10^3\,d$ & $29\times10^3\,d$ & $24\times10^3\,d$  
        \\  
   \hline\hline          
     \ce time & \ce 12  & $0.47$  & $1.7$ & $38$ & $850$ 
    \\ 
    \hline    
   \ce memory & \ce 12  & $8.9\times10^3\,d$  & $15\times10^3\,d$ & $27\times10^3\,d$ & $24\times10^3\,d$ 
    \\ 
  \end{tabular}
  	\caption{The wall-clock time (in sec) and memory use (in bytes) to compute $d$ digits of the generalized sum $\sum_{k\ge0}^\EM f(k)$ in \eqref{eq:series,c,EM} on one core and on $12$ parallel cores, for $f$ as in \eqref{eq:sqrt f} and 
  	each of the listed values of $d$ -- by using \eqref{eq:series,c,EM}.}
	\label{tab:sqrt-EM}
\end{table}

Table~\ref{tab:sqrt-EM} presents the wall-clock times (in sec) and the required memory use (in bytes) to compute $d$ digits of the generalized sum $\sum_{k\ge0}^\EM f(k)$ in \eqref{eq:series,c,EM}, based on the EM summation formula \eqref{eq:EM}, for each of the selected values of $d$, ranging from $\frac12\times10^3$ to $4\times10^3$.   
The calculations were done on the same 12-core machine.  
The values of $m$ and $c$ in \eqref{eq:series,c,EM} were chosen, for each selected value of $d$, in accordance with the discussion in Subsection~\ref{c,m,EM}. 
However, the value of the exponent $2+\vp$ in \eqref{eq:T^EM} is not so easy to measure, and it may depend on $d$. Anyhow, reasonably strong efforts were exerted to bracket the optimal value of $2+\vp$ and then to
choose values of $m$ and $c$ accordingly. 
The annotated code and details of the calculations pertained to Table~\ref{tab:sqrt-EM} are given in Mathematica notebooks \verb!\sqrt\AEMSqrt.nb! and \verb!\sqrt\AEMParSqrt.nb! (for the one-core and $12$-core Alt calculations, respectively) and their pdf images \verb!\sqrt\AEMSqrt.pdf! and \verb!\sqrt\AEMParSqrt.pdf!.  
The ratios of the one-core times to the corresponding 12-core times in Table~\ref{tab:sqrt-EM} are at best about $2$, which reflects the fact that only the calculation of the term 
$\sum_{k=0}^{c} f(k)$ in the approximation $\sum_{k=0}^{c} f(k)-G^\EM_{m,F}(c+1)$ to the generalized sum $\sum_{k\ge0}^\EM f(k)$ in \eqref{eq:series,c,EM} can be easily parallelized. 

One can see that, for $f$ as in \eqref{eq:sqrt f}, the calculations based on the Alt summation formula \eqref{eq:} are much faster than those based on the EM formula. For instance, the one-core times for $d=2\times10^3$ in Tables~\ref{tab:sqrt} and \ref{tab:sqrt-EM} are about $0.24$ sec and $81$ sec, respectively, with the ratio $81/0.24\approx340$. Such ratios seem to grow fast with $d$. For instance, 
the Alt summation formula \eqref{eq:} allows one to compute $2.5$ times as many digits ($10^4$ vs.\ $4\times10^3$) in a $0.72/850\approx1/1200$ fraction of the time
needed by the EM formula.  

As for the memory use, one can see from Table~\ref{tab:sqrt} that the corresponding numbers for the calculations in accordance with the Alt formula \eqref{eq:} are remarkably small, especially for the one-core calculations, and they go down (per digit of the result) and then stabilize as the number $d$ of computed digits grows. The memory use for the $12$-core calculations is, as expected, more than $12$ times the memory use for the corresponding one-core calculation; here one may also note that the parallelized code is more complicated than its non-parallelized counterpart. 

It is also seen that the calculations here by the EM formula take much more memory than those by the Alt summation formula \eqref{eq:}: for $d=2\times10^3$, it is about $29\times10^3/20\approx1500$ times as much for one core and $27\times10^3/360=75$ times as much for $12$ cores. 

Overall, the Alt summation formula here works on a rather different, higher scale than the EM one.    

The labels {\large Alt} and {\large EM} in the upper left cells in 
Tables~\ref{tab:sqrt} and \ref{tab:sqrt-EM}, as well as in the similar tables in the examples to be given in Subsections~\ref{erfInv}--\ref{euler}, 
indicate the use of the Alt summation formula \eqref{eq:} and the EM formula \eqref{eq:EM}, respectively. 

\subsection{Example: complicated \texorpdfstring{$f$}{f} and \texorpdfstring{$F$}{F}, very complicated derivatives \texorpdfstring{$f^{(2j)}$}{}}\label{erfInv}

Here we 
consider summing the convergent series $\sum_{k=0}^\infty f(k)$ according to Corollaries~\ref{cor:series} and \ref{cor:seriesEM}, where 
\begin{equation}\label{eq:f,erfInv}
	f(x):=\frac x{(x^2+2)\sqrt{x^2+1}}\,\erf^{-1}\Big(\arctan\frac1{\sqrt{x^2+1}}\Big)   
\end{equation}
for real $x\ge0$ and $\erf^{-1}$ is the function inverse to the error function $\erf$, given by the formula $\erf w:=\frac2{\sqrt\pi}\int_0^w e^{-v^2}\dd v$ for $w\in\CC$. 
Then for an antiderivative $F$ of $f$ one can take the function given by the formula 
\begin{equation}\label{eq:F,erfInv}
	F(x)=\frac1{\sqrt\pi}\,\exp\Big\{\erf^{-1}\Big(\arctan\frac1{\sqrt{x^2+1}}\Big)^2\Big\}. 
\end{equation}
One can see that Corollaries~\ref{cor:series} and \ref{cor:seriesEM} may be applicable, at least in principle, even when the function $f$ is rather far from elementary. 

Let 
\begin{equation*}
a:=-3,\quad \vp_1:=\tfrac13,\quad \vp_2:=\tfrac{38}{100},\quad
\de:=\tfrac{48}{100},\quad \ka_1:=\tfrac2{\sqrt\pi}\,(1-\de^2 e^{\de^2})\approx0.80.  	
\end{equation*}
For any real $r>0$, let $D_r:=\{z\in\CC\colon|z|\le r\}$. For all $v\in D_\de$ we have 
$|1-e^{-v^2}|= 
\big|\int_0^{-v^2}e^w\dd w\big|\le|v|^2 e^{|v|^2}
\le\de^2 e^{\de^2}$. So, for any $w_1$ and $w_2$ in $D_\de$ 
\begin{multline*}
	\tfrac{\sqrt\pi}2\,|\erf w_1-\erf w_2|=\Big|\int_{w_1}^{w_2} e^{-v^2}\dd v\Big| 
	\ge|w_1-w_2|-\Big|\int_{w_1}^{w_2} (1-e^{-v^2})\dd v\Big| \\ 
	\ge|w_1-w_2|(1-\de^2 e^{\de^2})=\tfrac{\sqrt\pi}2\,\ka_1|w_1-w_2|,  
\end{multline*}
so that the function $\erf$ is one-to-one on $D_\de$. 
It also follows that 
$|\erf w|\ge\ka_1\,|w|$\quad for all $w\in D_\de$, 
so that all the points in the image of the boundary of the disk $D_\de$ under the map $\erf$ are at distance $\ge\ka_1\de>\vp_2$ from $0$. 
It now follows by Darboux--Picard theorem (see e.g.\ Subsection~2.1 in \cite{dichot_published} or the more general Jordan Filling principle there) that $\erf(D_\de)\supseteq D_{\vp_2}$. 
Hence, the inverse function $\erf^{-1}$ can be extended to $D_{\vp_2}$, and 
\begin{equation}\label{eq:subset}
	\erf^{-1}(D_{\vp_2})\subseteq D_\de. 
\end{equation}
Next, for any $v\in\CC$ we have 
\begin{equation}\label{eq:<vp_2}
	|v|\le\vp_1\implies|\arctan v|=\Big|\int_0^v\frac{\dd u}{1+u^2}\Big|\le\frac{|v|}{1-|v|^2}
	\le\frac{\vp_1}{1-\vp_1^2}<\vp_2. 
\end{equation}

Further, 
recall here \eqref{eq:S=Pi} and take any $z\in\Pi_{-a}^+=\Pi_3^+$, so that $\Re z\ge3$. 
Let $s:=(\Re z)^2[\ge a^2=9]$ and $t:=(\Im z)^2[\ge0]$. 
Then $|z^2+1|^2=(1+s-t)^2+4st=:h(t)$. Clearly, $h(t)$ is convex in $t$ and $h'(0)=2(s-1)\ge16>0$ for $t\ge0$. So, $|z^2+1|^2\ge h(0)=(1+s)^2\ge(1+a^2)^2$. Thus, 
\begin{equation}\label{eq:<vp_1}
z\in\Pi_{-a}^+\implies	\big|\tfrac1{\sqrt{z^2+1}}\big|\le\tfrac1{\sqrt{a^2+1}}<\tfrac13=\vp_1. 
\end{equation}
Similarly, $\big|\frac z{z^2+2}\big|^2$ is a simple rational function of $s$ and $t$, which is easy to maximize over $s\ge a^2$ and $t\ge0$. Actually, in view of the maximum modulus principle, here one may assume that $s=a^2[=9]$. It is then easy to see that the derivative of the mentioned rational function in $t$ is $\le0$. So, the maximum of $\big|\frac z{z^2+2}\big|$ in $z\in\Pi_{-a}^+$ is attained at $z=-a=3$. Hence, 
\begin{equation*}
z\in\Pi_{-a}^+\implies\big|\tfrac z{z^2+2}\big|\le\tfrac{|a|}{a^2+2}=\tfrac3{11}. 
\end{equation*}

Combining this with \eqref{eq:f,erfInv}, \eqref{eq:<vp_1}, \eqref{eq:<vp_2}, and \eqref{eq:subset}, we have 
$
	|f(z)|\le\tfrac3{11}\,\tfrac13\,\de=\tfrac{48}{1100} 
$ 
for all $z\in\Pi_{-a}^+$. 
That is, 
condition \eqref{eq:|f|<} will hold with 
\begin{equation}\label{eq:mu=,erf}
	a=-3,\quad\la=0,\quad\text{and}\quad\mu=\tfrac{48}{1100},  
\end{equation}
which will be assumed in the rest of the consideration of this example. 
The remainder $R_{m,f_c}(\infty)$ in \eqref{eq:series,c} can now be bounded according to \eqref{eq:R_c,infty}, which will hold if 
\begin{equation}\label{eq:m,c>,erf}
\text{$m\ge2$\quad and\quad $c\ge3+(m+3)/2=(m+9)/2$.}	
\end{equation}

By Remark~\ref{rem:m_0}, Corollaries~\ref{cor:series} and \ref{cor:seriesEM} will hold with 
\begin{equation}\label{eq:m>m0=1}
	m>m_0=1, 
\end{equation}
which will also be assumed in this example. 
Then, by \eqref{eq:G} and \eqref{eq:G^EM_m}, one has 
\begin{equation*}
	G_{m_0,F}(n)=
	F(n-1)=
	\frac1{\sqrt\pi}\,\exp\Big\{\erf^{-1}\Big(\arctan\frac1{\sqrt{(n-1)^2+1}}\Big)^2\Big\}
	=\frac1{\sqrt\pi}+o(1)
\end{equation*}
and 
\begin{equation*}
	G^\EM_{m_0,F}(n)=F(n-1)+\frac12\,f(n-1)=\frac1{\sqrt\pi}+o(1). 
\end{equation*} 
So, for this example as well, we confirm the asymptotic relation \eqref{eq:same}. 

In this example, as in the example in Subsection~\ref{sqrt}, 
we have $T_f\approx T_F>>T_\tau$. \big(In fact, here the values of $T_f$ and $T_F$ are much greater than those in the example in Subsection~\ref{sqrt}, so that here the comparison $>>T_\tau$ is even more pronounced.\big) 
So, here we still take $m$ and $c$ according to \eqref{eq:m=} and \eqref{eq:cc=}. 

Tables~\ref{tab:erf} and \ref{tab:erfInv-EM} are similar to Tables~\ref{tab:sqrt} and \ref{tab:sqrt-EM}, respectively. The main difference is that here the values of $d$ are significantly smaller than the corresponding values in the example in Subsection~\ref{sqrt}, because the execution times in this setting are much greater. 
The annotated code and details of these calculations are given in Mathematica notebooks \verb!\erfInvAtan\AErfInvAtan.nb!, \verb!\erfInvAtan\AParErfInvAtan.nb!, \verb!\erfInvAtan\AEMErfInvAtan.nb!, \verb!\erfInvAtan\AEMParErfInvAtan.nb!,  and their pdf images, with file name extension \verb!.pdf! in place of \verb!.nb!. 

\setlength\tabcolsep{3.pt}
\renewcommand{\arraystretch}{1.5}
\begin{table}[h]
	\centering
  \begin{tabular}{c|m{.8cm}||c|c|c|c|c}
\raisebox{-6pt}[-5pt][-5pt]{\large Alt}  & \multicolumn{1}{c||}{\raisebox{-6pt}[-5pt][-5pt]{\# cores}} &\multicolumn{5}{c}{$d$} \\
\cline{3-7}
    \ce $ $ & & $\frac12\,\times 10^3$ & $10^3$ & $2\times 10^3$ & $4\times 10^3$ & $8\times 10^3$  \\ 
    \hline\hline
     \ce time & \ce 1 & $5.4$ & $30$ & $240$ & $1800$ & $-$  
    \\ 
   \hline
     \ce memory & \ce 1  & $8.2\times10^3\,d$  & $11\times10^3\,d$ & $14\times10^3\,d$ & $16\times10^3\,d$ & $-$  
        \\  
   \hline\hline          
     \ce time & \ce 12  & $0.78$  & $4.0$ & $29$ & $210$ & $1900$  
    \\ 
    \hline    
   \ce memory & \ce 12  & $21\times10^3\,d$  & $16\times10^3\,d$ & $20\times10^3\,d$ & $30\times10^3\,d$ & $50\times10^3\,d$  
    \\ 
  \end{tabular}
  	\caption{The wall-clock time (in sec) and memory use (in bytes) to compute $d$ digits of the generalized sum $\sum_{k\ge0}^\Alt f(k)$ in \eqref{eq:series,c}, on one core and on $12$ parallel cores, 
  	for 
  	$f$ as in \eqref{eq:f,erfInv} and 
  	each of the listed values of $d$ -- by using \eqref{eq:series,c}.}
	\label{tab:erf}
\end{table}

\setlength\tabcolsep{3.pt}
\renewcommand{\arraystretch}{1.5}
\begin{table}[h]
	\centering
  \begin{tabular}{c|m{.8cm}||c|c|c|c}
\raisebox{-6pt}[-5pt][-5pt]{\large EM}  & \multicolumn{1}{c||}{\raisebox{-6pt}[-5pt][-5pt]{\# cores}} &\multicolumn{4}{c}{$d$} \\
\cline{3-6}
    \ce $ $ & & $20$ & $50$ & $100$ & $150$ \\ 
    \hline\hline
     \ce time & \ce 1 & $1.1$ & $33$ & $670$ & $-$  
    \\ 
   \hline
     \ce memory & \ce 1  & $61\times10^3\,d$  & $560\times10^3\,d$ & $1100\times10^3\,d$ & $-$  
        \\  
   \hline\hline          
     \ce time & \ce 12  & $0.62$  & $14$ & $300$ & $1900$ 
    \\ 
    \hline    
   \ce memory & \ce 12  & $190\times10^3\,d$  & $1600\times10^3\,d$ & $1600\times10^3\,d$ & $2700\times10^3\,d$ 
    \\ 
  \end{tabular}
  	\caption{The wall-clock time (in sec) and memory use (in bytes) to compute $d$ digits of the generalized sum $\sum_{k\ge0}^\EM f(k)$ in \eqref{eq:series,c,EM} on one core and on $12$ parallel cores, for $f$ as in \eqref{eq:f,erfInv} and 
  	each of the listed values of $d$ -- by using \eqref{eq:series,c,EM}.}
	\label{tab:erfInv-EM}
\end{table}

We see that here the Alt formula operates on a much higher scale than the EM one. Based on Table~\ref{tab:erfInv-EM}, a rather conservative prediction 
of the execution times on the 12 cores using the EM formula for $d=\frac12\,\times 10^3$,  $10^3$, $2\times 10^3$, $4\times 10^3$, and $8\times 10^3$ 
would be about 
$2.2$ days, $35$ days, $1.5$ years, $25$ years, and $390$ years -- versus the corresponding execution times $0.78$ sec, $4.0$ sec, $29$ sec, $210$ sec, and $1900$ sec in Table~\ref{tab:erf}. 
The predicted memory use for the same values of $d$ would be about $7$ GB, $33$ GB,  $150$ GB, $730$ GB, and $3400$ GB -- versus the corresponding memory use numbers $21\times10^3\times\frac12\,\times 10^3\text{ B}\approx0.011$ GB, $16\times10^3\times 10^3\text{ B}\approx0.016$ GB, $20\times10^3\times 2\times 10^3\text{ B}\approx0.04$ GB, $30\times10^3\times 4\times 10^3\text{ B}\approx0.12$ GB, and $50\times10^3\times 8\times 10^3\text{ B}\approx0.4$ GB in Table~\ref{tab:erf}. 
Details on these predictions can be found in 
Mathematica notebook \verb!\erfInvAtan\prediction.nb!
and its pdf image \verb!\erfInvAtan\prediction.pdf!.

\subsection{Example (Calculation of an array of values of the Hurwitz generalized zeta function): complicated \texorpdfstring{$f$}{f} and \texorpdfstring{$F$}{F}, comparatively simple derivatives \texorpdfstring{$f^{(2j)}$}{} } \label{zeta}
Both the EM formula and the Alt summation formula are applicable when 
the function $f$ takes values in any normed space, rather than just real values. In particular, $f$ may be complex-valued. For instance, let us consider in this example the vector function 
\begin{equation}\label{eq:f,vect}
	f:=(f_{-1+i,i},\,f_{i,i},\,f_{1+i,i},\,f_{2+i,i}), 
\end{equation}
with values in $\CC^4$, where 
\begin{equation}\label{eq:f_p,de}
	f_{p,\de}(x):=(x+\de)^{-p}    
\end{equation}
for $x\ge0$, with $\de\in\CC\setminus\{0,-1,-2,\dots\}$ and $p\in\CC$; 
in this context, of course 
$i$ denotes the imaginary unit. 
As usual, for any $z\in\CC\setminus\{0\}$ with $\th:=\arg z\in(-\pi,\pi]$, we let $z^p:=|z|^p e^{i\th p}$. 
So, the vector series corresponding to the vector function $f$ in \eqref{eq:f,vect} is 
\begin{equation*}
	\sum_{k=0}^\infty f(k)=\Big(\sum_{k=0}^\infty(k+i)^{1-i},\sum_{k=0}^\infty(k+i)^{-i},\sum_{k=0}^\infty(k+i)^{-1-i},\sum_{k=0}^\infty(k+i)^{-2-i}\Big). 
\end{equation*}
Each of the first three components of this vector series is a divergent series in $\CC$, whereas the fourth component is a convergent series in $\CC$ -- whose value is $\zeta(2+i,i)$, where 
$\zeta$ is the Hurwitz generalized zeta function. The expression $\zeta(p,1)$ defines the Riemann zeta function. 

Usually (see e.g.\ \cite[page~15]{andrews}),
$\zeta(p,\de)$ is defined only for real $\de>0$ -- initially, as $\sum_{k=0}^\infty f_{p,\de}(k)$ for $p\in\Pi_{1+}^+$ 
\big(where $\Pi_{a+}^+:=\{z\in\CC\colon\Re z>a\}$ for real $a$; cf.\ the definition of $\Pi_\cdot^+$ in \eqref{eq:S=Pi}\big), and then extended by analytic continuation to all $p\in\CC\setminus\{1\}$. However, $\zeta(p,\de)$ can actually be defined 
for all $\de\in\CC\setminus(-\infty,0]$ and all $p\in\CC\setminus\{1\}$. One way to do this is as follows. 

Take any $\de\in\CC\setminus(-\infty,0]$. Again, consider first the case when $p\in\Pi_{1+}^+$. Then, by Proposition~\ref{prop:seriesEM} with $m\ge m_0=1$, 
\begin{equation}\label{eq:zeta=Z}
	\zeta(p,\de)=Z_m(p,\de):=f_{p,\de}(0)-G^\EM_{m,F_{p,\de}}(1)+R^\EM_{m,f_{p,\de}}(\infty), 
\end{equation}
where 
\begin{equation}\label{eq:F_p,de}
	F_{p,\de}(x):=\tfrac1{1-p}\,(x+\de)^{1-p}.   
\end{equation}
It is easy to see that $Z_m(p,\de)$ is analytic in $p\in\Pi_{(2-2m)+}^+\setminus\{1\}$. 
Note also that 
\begin{equation}\label{eq:p,m iff}
	p\in\Pi_{(2-2m)+}^+\iff m>1-\tfrac12\,\Re p. 
\end{equation}
So, the formula $\zeta(p,\de):=Z_m(p,\de)$ for any $p\in\CC\setminus\{1\}$ and any natural $m>1-\tfrac12\,\Re p$ provides a well-defined analytic continuation $\CC\setminus\{1\}\ni p\mapsto\zeta(p,\de)$, for each $\de\in\CC\setminus(-\infty,0]$. 
In fact, this extension is analytic in $\de\in\CC\setminus(-\infty,0]$ as well. 

Moreover, again by Proposition~\ref{prop:seriesEM} -- but now with any 
\begin{equation}\label{eq:m,m_0,zeta}
	m_0>1-\tfrac12\,\Re p\quad\text{and}\quad m\ge m_0
\end{equation}
(cf.\ \eqref{eq:p,m iff}), we deduce from \eqref{eq:zeta=Z}, Corollary~\ref{cor:seriesEM},  Proposition~\ref{prop:same}, and Corollary~\ref{cor:series} 
that 
\begin{align*}
	\zeta(p,\de) 
=&	\sum_{k\ge0}^\EM f_{p,\de}(k)
=
\sum_{k=0}^c f_{p,\de}(k)-G^\EM_{m,F_{p,\de}}(c+1)+R^\EM_{m,(f_{p,\de})_c}(\infty) 
\\ 
&\quad\quad \resizebox{4pt}{6pt}{$||$} \notag \\ 
&
	\sum_{k\ge0}^\Alt f_{p,\de}(k)=\sum_{k=0}^{c-1} f_{p,\de}(k)-G_{m,F_{p,\de}}(c)+R_{m,(f_{p,\de})_c}(\infty) \notag
\end{align*}
for any $p\in\CC\setminus\{1\}$, any $\de\in\CC\setminus(-\infty,0]$, any natural $c$, and any $m$ and $m_0$ as in \eqref{eq:m,m_0,zeta}. 

So,
\begin{align}
	\big(\zeta(-1+i,i),\zeta(i,i),\zeta(1+i,i),\zeta(2+i,i)\big) 
&=
\sum_{k=0}^c f(k)-G^\EM_{m,F}(c+1)+R^\EM_{m,f_c}(\infty) 
\label{eq:zeta=EM} 
\\ 
&=
\sum_{k=0}^{c-1} f(k)-G_{m,F}(c)+R_{m,f_c}(\infty), 
\label{eq:zeta=Alt} 	
\end{align}
where $f$ is the vector function as in \eqref{eq:f,vect},  
\begin{equation*}
	F:=(F_{-1+i,i},\,F_{i,i},\,F_{1+i,i},\,F_{2+i,i})  
\end{equation*}
is a vector function with values in $\CC^4$ which is 
an antiderivative of $f$, $F_{p,\de}$ is as in \eqref{eq:F_p,de}, 
$c$ is any natural number, and $m$ and $m_0$ are as in \eqref{eq:m>m0=2}. 
Thus, formulas \eqref{eq:zeta=EM} and \eqref{eq:zeta=Alt} present as a way to compute an array of values of the Hurwitz generalized zeta function. 

Suppose that a real number $a$ is such that $a\le\Re\de-1$. 
Recall then \eqref{eq:S=Pi} and take any $z\in\Pi_{-a}^+$, so that $\Re z\ge-a$, whence 
$|z+a+1|\ge\Re z+a+1\ge1$, 
$$|z+\de|\le|z+a+1|
+|\de-a-1|\le(1+|\de-a-1|)|z+a+1|,$$ 
$|z+\de|\ge\Re(z+\de)\ge-a+\Re\de\ge1$.  
Therefore, if $\Re p<0$, then  
\begin{equation*}
	|f_{p,\de}(z)|\le |z+\de|^{-\Re p}\, e^{|\Im p|\,\pi/2}
	\le e^{\pi\,|\Im p|/2}\,(1+|\de-a-1|)^{-\Re p}|z+a+1|^{-\Re p}, 
\end{equation*}
so that 
condition \eqref{eq:|f|<} holds with $\mu=e^{\pi\,|\Im p|/2}(1+|\de-a-1|)^{-\Re p}$, $\la=-\Re p$, and $f_{p,\de}$ in place of $f$. 
If $\Re p\ge0$, then 
\begin{equation*}
	|f_{p,\de}(z)|\le e^{\pi\,|\Im p|/2}, 
\end{equation*}
so that 
condition \eqref{eq:|f|<} holds with $\mu=e^{\pi\,|\Im p|/2}$, $\la=0$, and $f_{p,\de}$ in place of $f$. 
Thus, for each coordinate of the function $f$ in \eqref{eq:f,vect}, condition \eqref{eq:|f|<} holds with 
\begin{equation}\label{eq:mu=,vect}
a=-1,\quad\la=1,\quad	\mu=2e^{\pi/2},  
\end{equation}
which will be assumed in the rest of the consideration of this example. 
The remainder $R_{m,f_c}(\infty)$ in \eqref{eq:series,c} can now be bounded according to \eqref{eq:R_c,infty}, which will hold if \eqref{eq:m,c>,erf} holds. 
By Remark~\ref{rem:m_0}, Corollaries~\ref{cor:series} and \ref{cor:seriesEM} will hold if \eqref{eq:m>m0=2} holds,   
which will also be assumed in this example. 
Writing the terms of the form $(n+\th)^{1-p}$ in the expressions in \eqref{eq:G} and \eqref{eq:G^EM_m} for $G_{m_0,F}(n)$ and $n^i G^\EM_{m_0,F}(n)$ as $n^{1-p}(1+\th\vp)^{1-p}$ with $\vp:=\frac1n$, and then expanding $(1+\th\vp)^{1-p}$ into the powers of $\vp$, 
%
we have 
\begin{equation*}
	G_{m_0,F}(n)=\tfrac{4}{3} F(n-\tfrac{1}{2})-\tfrac{1}{6}F(n-1)-\tfrac{1}{6} F(n)=\tilde G(n)+O\big(\tfrac1n\big)	
\end{equation*}
and 
\begin{equation*}
	G^\EM_{m_0,F}(n)=F(n-1)+\tfrac12\,f(n-1)+\tfrac{B_2}2\,f'(n-1)=\tilde G(n)+O\big(\tfrac1n\big),	
\end{equation*}
where 
\begin{equation*}
	\tilde G(n):=\tfrac1{n^i}\,\big(
\tfrac{12 n^2 + 30 i\, n - 23 + 9 i}{12 (2 - i)},\ 
\tfrac{2 n + 1 + 3 i}{2 (1 - i)},\ 
\tfrac1{- i},\, 
0
\big)
\end{equation*}
and $O(\frac1n)$ stands for a vector of the form $\big(O(\frac1n),\,O(\frac1n),\,O(\frac1n),\,O(\frac1n)\big)$. 
So, for this example as well, we confirm the asymptotic relation \eqref{eq:same}. 
One may also note here that the factor $\frac1{n^i}=e^{-i\ln n}$ in the above expression for $\tilde G(n)$ is of modulus $1$, and it oscillates in $n$ with frequency $\frac{\ln n}{2\pi n}\underset{n\to\infty}\longrightarrow0$. 

In this example, as in the example in Subsection~\ref{sqrt}, 
we have $T_f\approx T_F>>T_\tau$. \big(In fact, here the values of $T_f$ and $T_F$ are much greater than those in the example in Subsection~\ref{sqrt}, so that here the comparison $>>T_\tau$ is even more pronounced.\big) 
Therefore, here we still take $m$ and $c$ according to \eqref{eq:m=} and \eqref{eq:cc=}. 

Tables~\ref{tab:vect} and \ref{tab:vect-EM} are similar to Tables~\ref{tab:sqrt} and \ref{tab:sqrt-EM}, respectively. As in the example in Subsection~\ref{erfInv}, here the values of $d$ are significantly smaller than the corresponding values in the example in Subsection~\ref{sqrt}, because the execution times in this setting are much greater.  
The annotated code and details of these calculations are given in 
Mathematica notebooks \verb!\vector\AVect.nb!, \verb!\vector\AParVect.nb!, \verb!\vector\AEMVect.nb!, \verb!\vector\AEMParVect.nb!,  and their pdf images, with file name extension \verb!.pdf! in place of \verb!.nb!. 

\setlength\tabcolsep{3.pt}
\renewcommand{\arraystretch}{1.5}
\begin{table}[h]
	\centering
  \begin{tabular}{c|m{.8cm}||c|c|c|c|c}
\raisebox{-6pt}[-5pt][-5pt]{\large Alt}  & \multicolumn{1}{c||}{\raisebox{-6pt}[-5pt][-5pt]{\# cores}} &\multicolumn{5}{c}{$d$} \\
\cline{3-7}
    \ce $ $ & & $10^3$ & $2\times 10^3$ & $4\times 10^3$ & $8\times 10^3$ & $16\times 10^3$  \\ 
    \hline\hline
     \ce time & \ce 1 & $3.0$ & $16$ & $110$ & $780$ & $-$  
    \\ 
   \hline
     \ce memory & \ce 1  & $15\times10^3\,d$ & $26\times10^3\,d$ & $48\times10^3\,d$ & $93\times10^3\,d$ & $-$  
        \\  
   \hline\hline          
     \ce time & \ce 12  &  $0.70$ & $2.8$ & $16$ & $100$ & $620$  
    \\ 
    \hline    
   \ce memory & \ce 12  &  $33\times10^3\,d$ & $42\times10^3\,d$ & $64\times10^3\,d$ & $110\times10^3\,d$ & $200\times10^3\,d$  
    \\ 
  \end{tabular}
  	\caption{The wall-clock time (in sec) and memory use (in bytes) to compute $d$ digits of the generalized sum $\sum_{k\ge0}^\Alt f(k)=\big(\zeta(-1+i,i),\zeta(i,i),\zeta(1+i,i),\zeta(2+i,i)\big)$ in \eqref{eq:series,c}, on one core and on $12$ parallel cores, 
  	for 
  	$f$ as in \eqref{eq:f,vect} and 
  	each of the listed values of $d$ -- by using \eqref{eq:series,c}.}
	\label{tab:vect}
\end{table}

\setlength\tabcolsep{4pt}
\renewcommand{\arraystretch}{1.5}
\begin{table}[h]
	\centering
  \begin{tabular}{c|m{.8cm}||c|c|c|c|c}
\raisebox{-6pt}[-5pt][-5pt]{\large EM}  & \multicolumn{1}{c||}{\raisebox{-6pt}[-5pt][-5pt]{\# cores}} &\multicolumn{5}{c}{$d$} \\
\cline{3-7}
    \ce $ $ & & $10^3$ & $2\times10^3$ & $4\times10^3$ & $8\times 10^3$ & $16\times 10^3$ \\ 
    \hline\hline
     \ce time & \ce 1 & $2.7$ & $10$ & $60$ & $400$ & $-$  
    \\ 
   \hline
     \ce memory & \ce 1  & $11\times10^3\,d$  & $17\times10^3\,d$ & $27\times10^3\,d$ & $59\times10^3\,d$ & $-$  
        \\  
   \hline\hline          
     \ce time & \ce 12  & $1.2$  & $5.1$ & $27$ & $160$ & $1200$ 
    \\ 
    \hline    
   \ce memory & \ce 12  & $42\times10^3\,d$  & $54\times10^3\,d$ & $94\times10^3\,d$ & $190\times10^3\,d$ & $450\times10^3\,d$ 
    \\ 
  \end{tabular}  	
  \caption{The wall-clock time (in sec) and memory use (in bytes) to compute $d$ digits of the generalized sum $\sum_{k\ge0}^\EM f(k)=\big(\zeta(-1+i,i),\zeta(i,i),\zeta(1+i,i),\zeta(2+i,i)\big)$ in \eqref{eq:series,c,EM} on one core and on $12$ parallel cores, 
  	for $f$ as in \eqref{eq:f,vect} and 
  	each of the listed values of $d$ -- by using \eqref{eq:series,c,EM}.}
	\label{tab:vect-EM}
\end{table}

The main difference between the present example and the examples in Subsections~\ref{sqrt} and \ref{erfInv}
is that in this exceptional case the values of the higher-order derivatives of $f$ are about as easy to compute as the values of the function $f$ itself, since 
\begin{equation}\label{eq:der-vect}
f_{p,\de}^{(2j-1)}(x)=-p(p+1)\dots(p+2 j-2) (x+\de)^{-p-2 j+1} 	
\end{equation}
for $j=1,2,\dots$ and $x\ge0$. 
Accordingly, we see that the numbers in Table~\ref{tab:vect} are of the same order of magnitude as the corresponding numbers in Table~\ref{tab:vect-EM}; the latter numbers are somewhat better for the one-core calculations and somewhat worse for the 12-core ones.

The execution time and memory use measurements reported in Tables~\ref{tab:vect} and \ref{tab:vect-EM} may be compared with the corresponding measurements for the Mathematica built-in command 
\texttt{HurwitzZeta[]}, which is one of the Mathematica commands that cannot be parallelized (in Mathematica). 
The execution time and memory use for \texttt{HurwitzZeta[]} are reported in Table~\ref{tab:vect-Ma}; for details, see again Mathematica notebook \verb!\vector\AEMVect.nb! and its pdf image \verb!\vector\AEMVect.pdf!.  
The label {\large Ma} in the upper left cell in 
Table~\ref{tab:vect-Ma} indicates the use of Mathematica. 

\setlength\tabcolsep{4pt}
\renewcommand{\arraystretch}{1.5}
\begin{table}[h]
	\centering
  \begin{tabular}{c|m{.8cm}||c|c|c|c}
\raisebox{-6pt}[-5pt][-5pt]{\large Ma}  & \multicolumn{1}{c||}{\raisebox{-6pt}[-5pt][-5pt]{\# cores}} &\multicolumn{4}{c}{$d$} \\
\cline{3-6}
    \ce $ $ & & $10^3$ & $2\times10^3$ & $4\times10^3$ & $8\times 10^3$  \\ 
    \hline\hline
     \ce time & \ce 1 & $6.9$ & $44$ & $310$ & $2300$   
    \\ 
   \hline
     \ce memory & \ce 1  & $730\,d$  & $310\,d$ & $410\,d$ & $710\,d$ 
    \\ 
  \end{tabular}  	
  \caption{The wall-clock time (in sec) and memory use (in bytes) to compute $d$ digits of  $\sum_{k\ge0}^\Alt f(k)=\sum_{k\ge0}^\EM f(k)=\big(\zeta(-1+i,i),\zeta(i,i),\zeta(1+i,i),\zeta(2+i,i)\big)$ on one core, 
  	for $f$ as in \eqref{eq:f,vect} and 
  	each of the listed values of $d$ -- by using the (non-parallelizable) built-in Mathematica command \texttt{HurwitzZeta[]}.}
	\label{tab:vect-Ma}
\end{table}

One can see that, in terms of the execution time, our implementation of the EM formula performs significantly better than the built-in command \texttt{HurwitzZeta[]}, being about $6$ times as fast for $d=1000$ digits of accuracy and about $14$ times as fast for $d=8000$. In particular, this suggests that our code of the calculations by the EM formula is rather efficiently optimized for the execution time. 

The Alt summation formula performs even better, especially in the parallelized version, being about $10$ times as fast as \texttt{HurwitzZeta[]} for $d=1000$ and about $23$ times as fast for $d=8000$.  

On the other hand, \texttt{HurwitzZeta[]} is substantially better in terms of memory use than our implementations of the EM and Alt summation formulas. 

It was also suggested to the author of this paper to consider the performance of the Richardson extrapolation process (REP), described e.g.\ in \cite{sidi}. 
Results of the corresponding numerical experiment are reported in Table~\ref{tab:vect-R}. 
Details are given in Mathematica notebook \verb!\vector\richardsonVect.nb! and its pdf image \verb!\vector\richardsonVect.pdf!. 
The label {\large REP} in the upper left cell in 
Table~\ref{tab:vect-R} (as well as in Table~\ref{tab:euler-R} to follow) indicates the use of the REP. 

\setlength\tabcolsep{4pt}
\renewcommand{\arraystretch}{1.5}
\begin{table}[h]
	\centering
  \begin{tabular}{c|m{.8cm}||c|c|c|c}
\raisebox{-6pt}[-5pt][-5pt]{\large REP}  & \multicolumn{1}{c||}{\raisebox{-6pt}[-5pt][-5pt]{\# cores}} &\multicolumn{4}{c}{$d$} \\
\cline{3-6}
    \ce $ $ & & $23$ & $31$ & $40$ & $49$  \\ 
    \hline\hline
     \ce time & \ce 1 & $9.3$ & $36$ & $160$ & $670$   
    \\ 
   \hline
     \ce memory & \ce 1   & $2.0\times10^6\,d$ & $5.7\times10^6\,d$ & $19\times10^6\,d$   & $62\times10^6\,d$ 
    \\ 
  \end{tabular}  	
  \caption{The wall-clock time (in sec) and memory use (in bytes) to compute $d$ digits of 
$\big(\zeta(-1+i,i),\zeta(i,i),\zeta(1+i,i),\zeta(2+i,i)\big)$ (on one core), 
  	for $f$ as in \eqref{eq:f,vect} and 
  	each of the listed values of $d$ -- by using the (non-parallelizable) REP.}
	\label{tab:vect-R}
\end{table}

One can see that that the convergence of the REP is very slow, in comparison: to gain just $8$ or $9$ digits of accuracy, one needs to quadruple the execution time.  
At this rate, it would take the REP about $\frac{670}{365\times 24\times 3600}\times 4^{95}
\approx3\times 10^{52}$ years to compute $1000$ digits of the generalized sum  
-- whereas the corresponding calculation using the Alt summation formula takes only $0.70$ sec, as shown in Table~\ref{tab:vect}. 

The main underlying reason for this stark contrast seems to be the fact that the REP takes into account only the exponents -- but not the coefficients -- in 
the asymptotic expansions on which that method is based; see lines 2--3 after formula (1.1.2) on page~21 in \cite{sidi}. 
Thus, much of the available information is neglected in the REP. 

Also, because of its recursive nature, it appears that calculations by the REP cannot be parallelized. 

Approximations to the value $\zeta(2)=\zeta(2,1)=\pi^2/6$ of the Riemann zeta function were also computed in \cite{sidi} by the 
GREP (Generalization of the Richardson Extrapolation Process). However, as stated on page~138 in \cite{sidi}, with one choice of the relevant parameters of the GREP, ``[a]dding more terms to the process does not improve
the accuracy; to the contrary, the accuracy dwindles quite quickly. With the [other choice of the parameters], we are able to improve the accuracy to almost machine
precision.'' So, at least in this case, GREP provides much less accuracy than even the original REP -- cf.\ Table~\ref{tab:vect-R}. 

As noted on page~57 in \cite{sidi} concerning generalizations of the REP, ``[the] problems of convergence and stability [...]
turn out to be very difficult mathematically, especially because of this generality. As a
result, the number of the meaningful theorems that have been obtained and that pertain
to convergence and stability has remained small.'' 

In contrast with the explicit and rather easy to use upper bounds on the remainders for EM formula and the Alt summation formula -- such as the ones given by \eqref{eq:R_m^EM<} and 
\eqref{eq:R bound}, only error bounds of generic form $O(\cdot)$ seem to be available for the REP, without specification of the corresponding constant factors. This makes it more difficult to design calculations by the REP or its generalizations. 

No specific execution time or memory use data on the REP or its generalizations seem to be given in \cite{sidi} or elsewhere.

\subsection{Example (Calculation of the Euler constant): simple \texorpdfstring{$f$}{f}, comparatively complicated \texorpdfstring{$F$}{F}, and simple derivatives \texorpdfstring{$f^{(2j)}$}{} } \label{euler}

The Euler constant is defined by the formula 
\begin{equation}\label{eq:ga:=}
	\gaeu:=\lim_{n\to\infty}(\H_n-\ln n),  
\end{equation}
where 
\begin{equation}\label{eq:H_c}
	\H_n:=\sum_{\al=1}^n\frac1\al,   
\end{equation}
the $n$th harmonic number. 
%
Let here 
\begin{equation}\label{eq:f,euler}
	f(x):=\frac1{x+1}, 
\end{equation}
with the antiderivative 
\begin{equation}\label{eq:F,euler}
	F(x):=\ln(x+1)  
\end{equation}
for real $x\ge0$. 
Note that here values of $F$ are significantly harder to compute with high accuracy than values of $f$. 

Take any natural $c\ge(m+3)/2$. 
Since $\Re z\ge0$ implies $|z+1|\ge\Re z+1\ge1$, all the conditions of Proposition~\ref{prop:cauchy} will hold for $f_c$ and $a_c:=c+a$ in place of $f$ and $a$ (respectively) if $a=0$, $\la=0$, and $\mu=1$. So, by Remark~\ref{rem:m_0}, Corollaries~\ref{cor:series} and \ref{cor:seriesEM} will hold with $m>m_0=1$, which 
will be assumed in this example. 
In view of Corollary~\ref{cor:series}, \eqref{eq:G_m,alt2}, \eqref{eq:F,euler}, and \eqref{eq:sum ga_j}, here 
the generalized sum $\sum_{k\ge0}^\Alt f(k)$ in \eqref{eq:series,c} equals 
\begin{equation*}
	\gaeu=\H_c+\ln2-
	\tau_{m,1}\,\ln(2c+1)
	-\sum_{j=2}^{m}\tau_{m,j}\,\ln\big((2c+1)^2-(j-1)^2\big)-R_{m,f_c}(\infty). 
\end{equation*}
To obtain the latter expression, we rewrote the term $F(c-1/2-\al/2)+F(c-1/2+\al/2)=\break 
\ln(c+1/2-\al/2)+\ln(c+1/2+\al/2)$ in the expression of $G_{m,F}(c)$ in \eqref{eq:series,c} (cf.\ \eqref{eq:G_m,alt2}) as  $-2\ln2
+\ln\big((2c+1)^2-(j-1)^2\big)$ for $j:=1+\al$, which allows one to almost halve the number of harder-to-compute values of the logarithm function. 

By Remark~\ref{rem:R}, one may take 
$
	M_{2m}=\dfrac{(2m-1)!}{(c-m/2+1/2)^{2m}}, 
$ 
whence, by \eqref{eq:|R|<} and \eqref{eq:sum gaj}, here we have 
\begin{align}
	|R_{m,f_c}(\infty)|\le R^{**}_{m,c}
	&:=
	\frac{1.001\pi}{(2m+1)2m}\,\Big(\frac{\La_*}4\Big)^m 
	\frac{m^{2m+1}}{(c-m/2+1/2)^{2m}} \label{eq:R_c,infty,eu},   
\end{align}
which in this particular case yields a slight improvement on the general bound $R^*_{m,c}$ in \eqref{eq:R_c,infty} on $|R_{m,f_c}(\infty)|$, and the previously assumed condition $c\ge(m+3)/2$ can now be relaxed to $c>(m-1)/2$. 

Similarly, in view of Corollary~\ref{cor:seriesEM}, \eqref{eq:G^EM_m}, and \eqref{eq:F,euler},  
the generalized sum $\sum_{k\ge0}^\EM f(k)$ in \eqref{eq:series,c,EM} equals 
\begin{equation*}
	\gaeu=\H_{c+1}-\ln(c+1)-\frac1{2(c+1)}+\sum_{j=1}^{m-1}\frac{B_{2j}}{2j\,(c+1)^{2j}}
+R^\EM_{m,f_c}(\infty). 
\end{equation*}
By \eqref{eq:R^EM<}, assuming $m\ge4$, here we have 
\begin{align}
	|R^\EM_{m,f_c}(\infty)|\le R^{\EM,**}_{m,c}
	&:=
	\frac{2.02(2m-2)!}{(2\pi)^{2m-1}(c+1)^{2m-1}} \label{eq:REM_c,infty,eu},   
\end{align}
which in this particular case yields a slight improvement on the general bound $R^{\EM,*}_{m,c}$ in \eqref{eq:REM_c,infty} on $|R^\EM_{m,f_c}(\infty)|$. 

In this example, in contrast with the previous ones,  
time measurements  
show that in the considered range of values of $d$ one has 
$T_F\approx10d^{1/2}T_f$; 
see 
Mathematica notebooks \verb!\euler\1-over-k-timing.nb! and \verb!\euler\log-timing.nb! and the corresponding pdf images  for details. 
So, here $T_F>>T_f\vee T_\tau$. 
Also, by the mentioned ``halving'' the number of ``costly'' values of $F$ to compute, here we should replace the term $2T_F$ in the expression of $K$ in \eqref{eq:om} by $T_F$. So, 
\begin{equation*}
\om\approx\om_d:=\frac\ka{10d^{1/2}}\quad\text{and}\quad m\approx m_{\om_d},  
\end{equation*}
in accordance with \eqref{eq:m_om}. 

Accordingly, we will take here 
\begin{equation*}
	m=2\lceil \tfrac12\,m_{\om_d}\rceil,  
\end{equation*}
for $m$ to be an even integer, as was assumed to be in Section~\ref{par}. 
Formula \eqref{eq:cc=} is replaced here by 
\begin{equation*}
	c=\Big\lceil \frac{m-1}2
	+\Big(2\times10^d\,\frac{1.001 \ka^{2m}m^{2m+1}}{(2m+1)2m}\Big)^{1/(2m)} \Big\rceil,  
\end{equation*}
since the bound $R^{**}_{m,c}$ in \eqref{eq:R_c,infty,eu} here replaces the bound $R^*_{m,c}$ in \eqref{eq:R_c,infty} on $|R_{m,f_c}(\infty)|$. 

Then 
conditions $c>(m-1)/2$ 
and $m>m_0=1$ will hold. 

Tables~\ref{tab:euler} and \ref{tab:euler-EM} are similar to Tables~\ref{tab:sqrt} and \ref{tab:sqrt-EM}, respectively. 
%
The annotated code and details of these calculations are given in
Mathematica notebooks \verb!\euler\AEuler.nb!, \verb!\euler\AParEuler.nb!, \verb!\euler\AEMEuler.nb!, \verb!\euler\AEMParEuler.nb!,  and their pdf images, with file name extension \verb!.pdf! in place of \verb!.nb!.  

\setlength\tabcolsep{3.pt}
\renewcommand{\arraystretch}{1.5}
\begin{table}[h]
	\centering
  \begin{tabular}{c|m{.8cm}||c|c|c|c|c|c|c|c}
\raisebox{-6pt}[-5pt][-5pt]{\large Alt}  & \multicolumn{1}{c||}{\raisebox{-6pt}[-5pt][-5pt]{\# cores}} &\multicolumn{8}{c}{$d$} \\
\cline{3-10}
    \ce $ $ & & $10^3$ & $2\times10^3$ & $4\times10^3$ & $8\times 10^3$ & $16\times 10^3$ & $32\times 10^3$ & $64\times 10^3$ & $128\times 10^3$ \\ 
    \hline\hline
     \ce time & \ce 1 & $0.09$ & $0.29$ & $1.1$ & $6.0$ & $31$ & $150$ & $820$ & $-$  
    \\ 
   \hline
     \ce memory & \ce 1  & $270\,d$  & $270\,d$ & $320\,d$ & $420\,d$ & $620\,d$ & $1500\,d$ & $680\,d$ & $-$ 
        \\  
   \hline\hline          
     \ce time & \ce 12  & $0.10$  & $0.13$ & $0.26$ & $0.87$ & $3.9$  & $21$  & $100$  & $510$ 
    \\ 
    \hline    
   \ce memory & \ce 12  & $1900\,d$  & $1700\,d$ & $2200\,d$ & $2200\,d$ & $6300\,d$ & $13000\,d$ & $240\,d$ & $230\,d$ 
    \\ 
  \end{tabular} 
    	\caption{The wall-clock time (in sec) and memory use (in bytes) to compute $d$ digits of the Euler constant, on one core and on $12$ parallel cores, 
  	for 
  	$f$ as in \eqref{eq:f,euler} and 
  	each of the listed values of $d$ -- by using \eqref{eq:series,c}.}
	\label{tab:euler}
\end{table}

\setlength\tabcolsep{3pt}
\renewcommand{\arraystretch}{1.5}
\begin{table}[h]
	\centering
  \begin{tabular}{c|m{.8cm}||c|c|c|c|c|c|c|c}
\raisebox{-6pt}[-5pt][-5pt]{\large EM}  & \multicolumn{1}{c||}{\raisebox{-6pt}[-5pt][-5pt]{\# cores}} &\multicolumn{8}{c}{$d$} \\
\cline{3-10}
    \ce $ $ & & $10^3$ & $2\times10^3$ & $4\times10^3$ & $8\times 10^3$ & $16\times 10^3$ & $32\times 10^3$ & $64\times 10^3$ & $128\times 10^3$ \\ 
    \hline\hline
     \ce time & \ce 1 & $0.06$ & $0.26$ & $1.0$ & $5.7$ & $14$ & $70$ & $390$ & $-$  
    \\ 
   \hline
     \ce memory & \ce 1  & $130\,d$  & $110\,d$ & $140\,d$ & $240\,d$ & $1700\,d$ & $2000\,d$ & $3100\,d$ & $-$ 
        \\  
   \hline\hline          
     \ce time & \ce 12  & $0.07$  & $0.17$ & $0.66$ & $3.0$ & $15$  & $71$  & $470$  & $2600$ 
    \\ 
    \hline    
   \ce memory & \ce 12  & $150\,d$  & $170\,d$ & $160\,d$ & $240\,d$ & $500\,d$ & $1000\,d$ & $25000\,d$ & $ 39000\,d$ 
    \\ 
  \end{tabular}  	
  \caption{The wall-clock time (in sec) and memory use (in bytes) to compute $d$ digits of the Euler constant on one core and on $12$ parallel cores, 
  	for $f$ as in \eqref{eq:f,euler} and 
  	each of the listed values of $d$ -- by using \eqref{eq:series,c,EM}.}
	\label{tab:euler-EM}
\end{table}

We have also considered the performance of the Richardson extrapolation process (REP) in the present example. Results of the corresponding numerical experiment are reported in Table~\ref{tab:euler-R}. 
Details are given in 
Mathematica notebook \verb!\euler\richardsonEuler.nb! and its pdf image \verb!\euler\richardsonEuler.pdf!.

\setlength\tabcolsep{4pt}
\renewcommand{\arraystretch}{1.5}
\begin{table}[h]
	\centering
  \begin{tabular}{c|m{.8cm}||c|c|c|c|c|c|c}
\raisebox{-6pt}[-5pt][-5pt]{\large REP}  & \multicolumn{1}{c||}{\raisebox{-6pt}[-5pt][-5pt]{\# cores}} &\multicolumn{7}{c}{$d$} \\
\cline{3-9}
    \ce $ $ & & $74$ & $92$ & $113$ & $136$ & $161$ & $188$ & $218$  \\ 
    \hline\hline
     \ce time & \ce 1 & $0.36$ & $1.4$ & $2.6$ & $19$  & $75$ & $320$ & $1300$   
    \\ 
   \hline
     \ce memory & \ce 1   & $590\,d$ & $520\,d$ & $510\,d$   & $420\,d$  & $400\,d$ & $380\,d$   & $360\,d$ 
    \\ 
  \end{tabular}  	
  \caption{The wall-clock time (in sec) and memory use (in bytes) to compute $d$ digits of the Euler constant on one core, 
  	for $f$ as in \eqref{eq:f,euler} and 
  	each of the listed values of $d$ -- by using the (non-parallelizable) REP.}
	\label{tab:euler-R}
\end{table}

Projections based on the data presented in Table~\ref{tab:euler-R} 
suggest that it would take the REP about $18\times10^4$ years to compute $1000$ digits of $\gaeu=\sum_{k\ge0}^\Alt f(k)=\sum_{k\ge0}^\EM f(k)$ (see again Mathematica notebook \verb!\euler\richardsonEuler.nb! and its pdf image \verb!\euler\richardsonEuler.pdf! for details) -- compared with $0.06$ sec by the EM formula itself, without the Richardson extrapolation. 

The present example is to an extent similar to the example in Subsection~\ref{zeta}, the main difference between these two examples being that here the time $T_F$ to compute a value of $F$ is much greater for large $d$ than the time $T_f$ to compute a value of $f$.  
On the other hand, the main difference between the present example and 
the examples in Subsections~\ref{sqrt} and \ref{erfInv}
is that in this exceptional case as well the values of the higher-order derivatives of $f$ are about as easy to compute as the values of the function $f$ itself; see \eqref{eq:der-vect} with $p=1$ and $\de=1$. 
Thus,  
the present example 
represents one of the few least favorable situations for the Alt summation formula in comparison with the EM one. 

Yet, we see that, as in 
the example in Subsection~\ref{zeta}, the execution time numbers in Table~\ref{tab:euler} are of the same order of magnitude as the corresponding numbers in Table~\ref{tab:euler-EM}; the latter numbers are somewhat better for the one-core calculations and somewhat worse for the 12-core ones. 
However, according to \cite{brent-log}, $d$ values of the logarithmic function can be computed with $\asymp d$ digits of precision in time $T_{\log} \asymp M(d)\,d\,\log d\asymp d^2\, \log^2 d\, \log\log d$, where $M(d)\asymp d\, \log d\, \log\log d$ is the time needed to perform precision $d$ multiplication. On the other hand, the best known algorithms for Bernoulli numbers \cite{fillebrown,harvey10} will compute the first $d$ Bernoulli numbers with with $\asymp d$ digits of precision in time
$\;\asymp (d^2/\log d)\,M(d\,\log d)\asymp d^3\,\log d\,\log\log d>>T_{\log}$.  
So, for very large $d$, it should be expected that even on one core the Alt summation formula will perform faster than the EM one.

The memory use numbers in Table~\ref{tab:euler} are more or less similar to the corresponding numbers in Table~\ref{tab:euler-EM} for the one-core calculations, but about 10 times as large for the 12-core ones. One may note the big drop in the reported memory use from $13000\,d$ to $240\,d$ in the last row of Table~\ref{tab:euler} (and also the smaller drop $1500\,d$ to $680\,d$ in the one-core memory use row of the same table) when $d$ increases from $32\times 10^3$ to $64\times 10^3$. 
 
Quite a similar drop occurs in a much simpler, distilled version of this situation, when just the sum of a large number of values of the logarithmic function is computed with high accuracy; see details in Mathematica notebook \verb!\euler\memoryMatter.nb! and its pdf image \verb!\euler\memoryMatter.pdf!.  
Such a drop has not been observed in Maple \cite{ierley-maple}. 
According to a representative of Wolfram Research (the developer of Mathematica), the drop occurs because Mathematica switches from one method to another depending on various parameters of the computational process. 
Somewhat similar drops (in that case not only in the memory use but also in the execution time) occur in Mathematica calculations of the Bernoulli numbers; see details in 
Mathematica notebook \verb!B-time,mem.nb! and its pdf image \verb!B-time,mem.pdf!.
It is possible that thresholds for the values of parameters of the computational process determining the mentioned switches in methods have not been chosen or updated to be near their optimal values. 
This suggests some potential for further improvement in the execution time and memory use in Mathematica.



\section{Proofs
}\label{proof}
\begin{normalsize}
\begin{proof}[Proof of Theorem~\ref{th:}]
Take any $k=0,\dots,n-1$ and consider the Taylor expansion 
\begin{equation*}
	f(x)=\sum_{i=0}^{2m-1}\frac{f^{(i)}(k)}{i!}\,u^i
	+\frac{u^{2m}}{(2m-1)!}\,\int_0^1\dd s\,(1-s)^{2m-1}f^{(2m)}(k+su)
\end{equation*}
for all $x\in(k-m/2,k+m/2]$, where $u:=x-k$. 
Integrating both sides of this identity in $x\in(k-j/2,\break 
k+j/2]$ (or, equivalently, in $u\in(-j/2,+j/2]$) for each $j=1,\dots,m$, then multiplying by $\ga_{m,j}$, and then summing in $j$, one has 
\begin{equation}\label{eq:A=S+R}
	A_{m,k}=S_{m,k}+R_{m,k},
\end{equation}
where 
\begin{align}
	A_{m,k}&:=\sum_{j=1}^m\ga_{m,j}\int_{k-j/2}^{k+j/2}\dd x\,f(x), \notag \\ 
	S_{m,k}&:=\sum_{\al=0}^{m-1}\frac{f^{(2\al)}(k)}{(2\al+1)!\,2^{2\al}}\,\sum_{j=1}^m\ga_{m,j} j^{2\al+1}, 
	\label{eq:S_mk}\\ 
	R_{m,k}&:=\frac1{(2m-1)!}\,\int_0^1\dd s\,(1-s)^{2m-1}
	\sum_{j=1}^m\ga_{m,j} \int_{-j/2}^{j/2}\dd u\,u^{2m} f^{(2m)}(k+su) \notag \\ 
	&\,=\frac1{(2m-1)!\,2^{2m+1}}\,\int_0^1\dd s\,(1-s)^{2m-1}\int_{-1}^{1}\dd v\,v^{2m}
	\sum_{j=1}^m\ga_{m,j}j^{2m+1} f^{(2m)}(k+jsv/2). \notag 
\end{align}
Clearly, by \eqref{eq:R_m}, 
\begin{equation}\label{eq:sum R=R}
	\sum_{k=0}^{n-1}R_{m,k}=R_m. 
\end{equation}

Next, take any $\al=0,\dots,m-1$. Then, by \eqref{eq:ga_j},   
\begin{equation}\label{eq:sum_j}
\begin{gathered}
	-\binom{2m}{m}\,\sum_{j=1}^m\ga_{m,j} j^{2\al+1}
	=2\sum_{j=1}^m (-1)^j\binom{2m}{m+j} j^{2\al}
	=2\sum_{j=-m}^{-1} (-1)^j\binom{2m}{m+j} j^{2\al} \\ 
	=\sum_{j=-m}^{m} (-1)^j\binom{2m}{m+j} j^{2\al}-\binom{2m}{m}\,\ii\{\al=0\}. 
\end{gathered}	
\end{equation}
Here and elsewhere, 
$\ii\{\cdot\}$ denotes the indicator function. 
%
The power function $\psi_\al$ defined by the formula $\psi_\al(z):=z^{2\al}$ for real $z$ is obviously a polynomial of degree $2\al<2m$. Hence, $\De^{2m}\psi_\al=\psi_\al^{(2m)}=0$, where (for any natural $p$) 
$\De^p$ is the $p$th power of the symmetric difference operator $\De$ defined by the formula $(\De\phi)(z):=\phi(z+1/2)-\phi(z-1/2)$ for all functions $\phi\colon\R\to\R$ and all real $z$, so that 
\begin{equation}\label{eq:De^p}
	(\De^p\phi)(z)=\sum_{\be=0}^p(-1)^\be\binom p\be\phi(z+p/2-\be).  
\end{equation}
Therefore,
\begin{equation*}
	(-1)^m\sum_{j=-m}^{m} (-1)^j\binom{2m}{m+j} j^{2\al}
	=\sum_{\be=0}^{2m}(-1)^\be\binom{2m}{\be}(m-\be)^{2\al}
	=(\De^{2m}\psi_\al)(0)=0. 
\end{equation*}
It follows from \eqref{eq:sum_j} that 
\begin{equation}\label{eq:canceling}
\sum_{j=1}^m\ga_{m,j} j^q=\ii\{q=1\}\quad\text{for }q=1,3,\dots,2m-1, 	
\end{equation}
which in particular confirms 
the equality of the first two sums in 
\eqref{eq:sum ga_j}, involving the $\ga_{m,j}$'s, to $1$.  
The second equality in \eqref{eq:sum ga_j} now follows from, say, yet to be proved equality \eqref{eq:A_m,alt1} by taking there any natural $n\ge m$ and letting $f=\ii_{[m/2-1,n-m/2]}$; then each of the integrals in \eqref{eq:A_m}--\eqref{eq:A_m,alt1} equals $n-m+1\ne0$. 

In view of \eqref{eq:S_mk}, it also follows from \eqref{eq:canceling} that  
\begin{equation}\label{eq:S=f}
	S_{m,k}=f(k). 
\end{equation}

Take 
any $j=1,\dots,m$ and let $G(y):=G_j(y):=F(y-j/2)$ for all real $y$ -- where, recall, $F$ is any antiderivative of the function $f$. Then 
\begin{align*}
	\sum_{k=0}^{n-1}\int_{k-j/2}^{k+j/2}
	&=\sum_{k=0}^{n-1}[F(k+j/2)-F(k-j/2)] \\ 
	&=\sum_{k=0}^{n-1}G(k+j)-\sum_{k=0}^{n-1}G(k) \\ 
	&	=\sum_{k=j}^{n-1+j}G(k)-\sum_{k=0}^{n-1}G(k) \\ 
	&	=\sum_{k=n}^{n-1+j}G(k)-\sum_{k=0}^{j-1}G(k) \\ 
	&	=\sum_{i=0}^{j-1}G(n-1+j-i)-\sum_{i=0}^{j-1}G(i) \\ 
	&	=\sum_{i=0}^{j-1}[F(n-1+j/2-i)-F(i-j/2)]
	=\sum_{i=0}^{j-1}\int_{i-j/2}^{n-1+j/2-i} 
	=\sum_{i=0}^{j-1}\int_{-1+j/2-i}^{n-1+j/2-i},  
\end{align*}
since $\{i-j/2\colon i=0,\dots,j-1\}=\{-1+j/2-i\colon i=0,\dots,j-1\}$. 
So, 
\begin{equation}\label{eq:sum A=A}
	\sum_{k=0}^{n-1}A_{m,k}=\sum_{k=0}^{n-1}\sum_{j=1}^m\ga_{m,j}\int_{k-j/2}^{k+j/2}
	=\sum_{j=1}^m\ga_{m,j}\sum_{k=0}^{n-1}\int_{k-j/2}^{k+j/2}
		=\sum_{j=1}^m\ga_{m,j}\sum_{i=0}^{j-1}\int_{i-j/2}^{n-1+j/2-i}=A_m   
\end{equation}
by 
the definition of $A_m$ in 
\eqref{eq:A_m}, 
and the second equality in \eqref{eq:A_m} also follows.  
Now \eqref{eq:} follows from \eqref{eq:A=S+R}, \eqref{eq:sum R=R}, \eqref{eq:S=f}, and \eqref{eq:sum A=A}. 

To show that the 
first expression in \eqref{eq:A_m,alt1} equals that in \eqref{eq:A_m}, note that the conjunction of the conditions $j\in\{1,\dots,m\}$ and $i\in\{0,\dots,j-1\}$ is equivalent to the conjunction of the conditions $\al\in\{1-m,\dots,m-1\}$, $j\in\{1+|\al|,\dots,m\}$, and $j=1+|\al|$ mod $2$, where $\al:=2i-j+1$. So, in view of \eqref{eq:A_m}, 
\begin{equation*}
	A_m=\sum_{\al=1-m}^{m-1}\,\int_{\al/2-1/2}^{n-1/2-\al/2}{\kern-5pt}\dd x\,f(x)
	\;\sum_{j=1+|\al|}^m\ga_{m,j}\ii\{j=1+|\al|\text{ mod }2\}. 
\end{equation*}
In view of 
\eqref{eq:tau_j}, the first expression of $A_m$ in 
\eqref{eq:A_m,alt1} equals that in \eqref{eq:A_m}.  
The second equality in \eqref{eq:A_m,alt1} follows because $\{\al/2-1/2\colon\al=1-m,\dots,m-1\}=\{-1/2-\al/2\colon\al=1-m,\dots,m-1\}$.  
As for 
the two equalities in \eqref{eq:A_m,alt2}, they are obvious. 

Inequality \eqref{eq:|R|<} 
is obvious. 

Theorem~\ref{th:} is completely proved. 
\end{proof}

\begin{proof}[Proof of Proposition~\ref{lem:=B_p}]
Let $f(x)=x^p$. Then, in view of the condition 
$p=0,\dots,2m-1$, we have $f^{(2m)}=0$ and hence $R_m=0$.  
So, by \eqref{eq:}, with $C:=-\frac1{p+1}\sum_{j=1}^m\sum_{i=0}^{j-1}(i-j/2)^{p+1}$,  
\begin{align*}
	\sum_{k=0}^{n-1}k^p&=C+\frac1{p+1}\sum_{j=1}^m\ga_{m,j}\sum_{i=0}^{j-1}\Big(n-1+\frac j2-i\Big)^{p+1} \\ 
	&=C+\frac1{p+1}\sum_{j=1}^m\ga_{m,j}\sum_{i=0}^{j-1}
	\sum_{\al=0}^{p+1}\binom{p+1}\al\Big(\frac j2-i-1\Big)^\al n^{p+1-\al} \\ 
	&=
	C+\frac1{p+1}\sum_{\al=0}^{p+1}\binom{p+1}\al n^{p+1-\al}
	\sum_{j=1}^m\ga_{m,j}\sum_{i=0}^{j-1}\Big(\frac j2-i-1\Big)^\al,  
\end{align*}
which is a polynomial in $n$. Comparing the coefficient of $n^{p+1-\al}$ for $\al=p$ in this polynomial with the corresponding coefficient in the Faulhaber formula \eqref{eq:faulhaber}, we obtain the first equality in \eqref{eq:=B_p}. The second equality there is obtained quite similarly to the equalities in \eqref{eq:A_m,alt1} and \eqref{eq:A_m,alt2}. 
\qed 

\rule{0pt}{0pt}\big(A different, longer proof of Proposition~\ref{lem:=B_p}, which does not use \eqref{eq:}, was given by Amdeberhan \cite{amdeber}.\big)
\end{proof}

\begin{proof}[Proof of Proposition~\ref{prop:cauchy}]
Take any nonnegative integer $\al$ and any real $x\ge-m/2-1/2$. Then the condition $a\ge(m+3)/2$ yields $x+a\ge1$. Let $C_x(r_x)$ denote the circle $\{z\in\CC\colon
|z-x|=r_x\}$, 
where 
$
r_x:=(x+a)\sin\th_0. 	
$ 
Then it is easy to see that $C_x(r_x)$ is contained in the convex set $S$. So, by the Cauchy integral formula, 
\begin{equation*}
	|f^{(\al)}(x)|=\Big|\frac{\al!}{2\pi i}\,\oint_{C_x(r_x)}\frac{\dd z\,f(z)}{(z-x)^{\al+1}}\Big|
	=\Big|\frac{\al!}{2\pi {r_x^\al}}\,\int_0^{2\pi}\dd t\,e^{-i\al t}\,f(x+r_xe^{it})\Big|. 
\end{equation*}
Next, for any $t\in[0,2\pi]$ the conditions $x+a\ge1$ and $r_x=(x+a)\sin\th_0$ imply 
by \eqref{eq:|f|<} and the conditions $\mu\ge0$ and $\la\ge0$, 
\begin{equation*}
	|f(x+r_xe^{it})|\le\mu\,|x+r_xe^{it}+a+1|^\la
	\le\mu\,(x+a)^\la(2+\sin\th_0)^\la. 
\end{equation*} 
Thus, for all real $x\ge-m/2-1/2$ 
\begin{equation}\label{eq:der-bound}
	|f^{(\al)}(x)|
	\le\mu\al!\frac{(2+\sin\th_0)^\la}{\sin^\al\th_0}\frac1{(x+a)^{\al-\la}}=:g_\al(x). 
\end{equation}
To complete the proof of \eqref{eq:R_m<}, it remains to refer to \eqref{eq:|R|<} and \eqref{eq:M,eu-const}, and to do straightforward calculations. 
Inequality \eqref{eq:R_m^EM<} is obtained similarly, using \eqref{eq:R^EM<} and the inequality $\zeta(2m-1)<1.01$ for $m\ge4$ instead of \eqref{eq:|R|<} and \eqref{eq:M,eu-const}.  
Proposition~\ref{prop:cauchy} is now proved.  
\end{proof}

\begin{proof}[Proof of Proposition~\ref{prop:sum gaj}]
The key to this proof is 

\begin{lemma}\label{lem:gaj}
For $j=1,\dots,m-1$ 
\begin{equation*}
	|\ga_{m,j}|j
	<\trho_{m,j}:=2m^{2m+1}(m-j)^{j-m-1/2}(m+j)^{-m-j-1/2}. 
\end{equation*}
\end{lemma}

The bound $\trho_{m,j}$ on $|\ga_{m,j}|j$ was obtained by using the definition \eqref{eq:ga_j} of $\ga_{m,j}$ together with the Stirling approximation formula. 
Therefore, this bound is asymptotically tight when $m-j$ is large. 

On the other hand, the values of $|\ga_{m,j}|j$ are comparatively small when $m$ is large but $m-j$ is not, and so, the contribution of these terms into the sum $\sum_{j=1}^m|\ga_{m,j}|j^{2m+1}$ in \eqref{eq:sum gaj} is relatively small. 


\begin{proof}[Proof of Lemma~\ref{lem:gaj}] 
Consider the ratio 
\begin{equation*}
	r_{m,j}:=\frac{\trho_{m,j}}{|\ga_{m,j}|j} 
\end{equation*}
and then the ratio 
\begin{equation*}
	q_{m,j}:=\frac{r_{m,j+1}}{r_{m,j}}=\left(\frac{m-j}{m-1-j}\right)^{m-j-1/2} \left(\frac{m+j}{m+j+1}\right)^{m+j+1/2}.    
\end{equation*}
for $j=1,\dots,m-2$. Let us take here the liberty to deal with $q_{m,j}$ as a function of real arguments $m\in(1,\infty)$ and and $j\in[0,m-1)$, with the value of the function at the point $(m,j)$ defined as the value of the latter displayed expression. Then we have 
\begin{equation*}
	\pd 2j\ln q_{m,j}=\frac{2 m(2 j+1) \left(m^2+j^2+j\right)}{(m-1-j)^2(m-j)^2 (m+j)^2 (m+j+1)^2}>0, 
\end{equation*}
so that $\ln q_{m,j}$ is strictly convex in $j$. 

Next, $\pd 2m\ln q_{m,0}=\frac2{m (m^2-1)^2}>0$, so that $\ln q_{m,0}$ is strictly convex in $m$. Moreover, $\ln q_{m,0}\to0$ as $m\to\infty$. So, $\ln q_{m,0}>0$. 

Further, $\fd{}m\big[\big(\pd{}j\ln q_{m,j}\big)\big|_{j=0}\big]=-\frac{m^2+1}{m^2 (m^2-1)^2}<0$, so that $\big(\pd{}j\ln q_{m,j}\big)\big|_{j=0}$ is decreasing in $m$, with $\big(\pd{}j\ln q_{m,j}\big)\big|_{j=0}\to0$ as $m\to\infty$. So, $\big(\pd{}j\ln q_{m,j}\big)\big|_{j=0}>0$. 

Recalling now that $\ln q_{m,0}>0$ and $\ln q_{m,j}$ is convex in $j$, we conclude that $\ln q_{m,j}>0$ and $q_{m,j}>1$. 
So, in view of the definition of $q_{m,j}$, we see that $r_{m,j}$ is increasing in $j=1,\dots,m-1$. 

We also have $\fd2m\ln r_{m,1}=\frac2{m (m^2-1)^2}>0$, so that $\ln r_{m,1}$ is strictly convex in $m$. Moreover, $\ln r_{m,1}\to0$ as $m\to\infty$. So, $\ln r_{m,1}>0$ and $r_{m,1}>1$. Since $r_{m,j}$ is increasing in $j=1,\dots,m-1$, we now have $r_{m,j}>1$ for all $j=1,\dots,m-1$. Thus, in view of the definition of $r_{m,j}$, the proof of Lemma~\ref{lem:gaj} is complete. \qed
\end{proof}

Let us now turn back to the the proof of Proposition~\ref{prop:sum gaj}. 
The last sentence of Proposition~\ref{prop:sum gaj} is trivial. So, assume that $m\ge2$. 
By Lemma~\ref{lem:gaj} and \eqref{eq:La}, 
\begin{equation*}
	|\ga_{m,j}|j^{2m+1}<\trho_{m,j}j^{2m}
	=2m^{2m}\La(j/m)^m(1-j^2/m^2)^{-1/2}
	\le2m^{2m}\La_*^m(1-j^2/m^2)^{-1/2}  
\end{equation*}
for $j=1,\dots,m-1$, 
whence 
\begin{equation}\label{eq:sum ga<}
	\sum_{j=1}^m|\ga_{m,j}|j^{2m+1}
	<2m^{2m}\La_*^m\,\sum_{j=1}^{m-1}\frac1{\sqrt{1-j^2/m^2}}+|\ga_{m,m}|m^{2m+1}. 
\end{equation}
Since $\frac1{\sqrt{1-x^2}}$ is convex in $x\in[0,1)$, we have 
\begin{equation}\label{eq:sum 1/sq<}
	\sum_{j=1}^{m-1}\frac1{\sqrt{1-j^2/m^2}}\le\sum_{j=1}^{m-1}\int_{j-1/2}^{j+1/2}\frac{\dd x}{\sqrt{1-x^2/m^2}}
	<\int_0^m\frac{\dd x}{\sqrt{1-x^2/m^2}}=\frac\pi2\,m. 
\end{equation}
Next, let us show that for natural $m$ 
\begin{equation}\label{eq:ka>1}
	\ka_m:=\binom{2m}m\Big/\Big(2^{2 m} e^{-1/(8 m)}/\sqrt{\pi m}\Big)>1. 
\end{equation}
Let 
\begin{equation*}
	K_m:=\frac{\ka_{m+1}}{\ka_m}=\frac{(2 m+1)}{2 \sqrt{m (m+1)}}\,e^{-\frac{1}{8 m(m+1)}}.  
\end{equation*}
Then $\fd{}m\ln K_m=1/(8 m^2 (m+1)^2 (2m+1))>0$, so that $K_m$ is increasing in $m$, with $K_m\to1$ as $m\to\infty$. Hence, $K_m<1$ for all $m$. That is, $\ka_m$ is decreasing in $m$, with $\ka_m\to1$ as $m\to\infty$. Thus, the inequality in \eqref{eq:ka>1} is checked. It now follows from \eqref{eq:sum ga<}, \eqref{eq:sum 1/sq<}, and \eqref{eq:ga_j} that 
\begin{equation}
		\sum_{j=1}^m|\ga_{m,j}|j^{2m+1}
	<2m^{2m+1}\Big(\frac\pi2\La_*^m+2^{-2 m} e^{1/(8 m)} \sqrt{\pi/m}\Big)
	=\pi m^{2m+1}\La_*^m\,(1+\vp_m),  
\end{equation}
where $\vp_m:=2(4\La_*)^{-m} e^{1/(8 m)}/\sqrt{\pi m}$. Clearly, $\vp_m$ is decreasing in $m\ge1$. Also, 
$\vp_{26}<0.001$. So, $\vp_m<0.001$ for $m\ge26$.   
To complete the proof of Proposition~\ref{prop:sum gaj}, it remains to note that, by direct calculations, inequality \eqref{eq:sum gaj} holds for all $m=2,\dots,25$. 
\end{proof}

\begin{proof}[Proof of Proposition~\ref{prop:series}]
Take any natural $n$. Let 
\begin{equation}\label{eq:R_{m,f}}
	R_{m,f}(n):=R_m, 
\end{equation}
with $R_m$ as defined in \eqref{eq:R_m}.
Then, by \eqref{eq:R unif}, 
\begin{equation}\label{eq:R to R}
	R_{m,f}(n)\underset{n\to\infty}\longrightarrow R_{m,f}(\infty) 
\end{equation}
and, by \eqref{eq:}--\eqref{eq:A_m}, 
\begin{equation}\label{eq:=G-G-R}
	\sum_{k=0}^{n-1} f(k)=G_{m,F}(n)-G_{m,F}(0)-R_{m,f}(n).  
\end{equation}
So, 
\begin{equation}\label{eq:sum-G=}
	\sum_{k=0}^{n-1} f(k)-G_{m_0,F}(n)=[G_{m,F}(n)-G_{m_0,F}(n)]-G_{m,F}(0)-R_{m,f}(n). 
\end{equation}
Next, by the linearity of $G_{m,F}$ in $F$,   
\begin{equation}\label{eq:G-G=}
	G_{m,F}(n)-G_{m_0,F}(n)=[G_{m,T}(n)-G_{m_0,T}(n)]+[G_{m,F-T}(n)-G_{m_0,F-T}(n)],  
\end{equation}
where $T=T_{n,m_0,F}$ is the Taylor polynomial of order $p:=2m_0-1$ for the function $F$ at the point $n-1$, so that 
\begin{equation*}
	T(x)=\sum_{\al=0}^p\frac{F^{(\al)}(n-1)}{\al!}\,(x-n+1)^\al
\end{equation*}
for real $x$. 

By Proposition~\ref{lem:=B_p}, $G_{m,P_\be}(0)=B_\be=G_{m_0,P_\be}(0)$ for all $\be=0,\dots,p$, where $P_\be(x):=x^\be$ for real $x$. 
%
Since $T$ is a polynomial of degree $\le p$, it is a linear combination of the polynomials $P_0,\dots,P_p$.  
Therefore and because $G_{m,F}$ is linear in $F$, one has $G_{m,T}(0)=G_{m_0,T}(0)$. Since $p=2m_0-1$ and $m\ge m_0$, it also follows that $T^{(2m)}=0$ and hence, by \eqref{eq:R_{m,f}}, $R_{m,T'}(n)=0$. Thus, by \eqref{eq:=G-G-R}, 
\begin{equation*}
	\sum_{k=0}^{n-1} T'(k)=G_{m,T}(n)-G_{m,T}(0)=G_{m_0,T}(n)-G_{m_0,T}(0).   
\end{equation*}
Therefore and because $G_{m,T}(0)=G_{m_0,T}(0)$, we have 
\begin{equation}\label{eq:G=G}
G_{m,T}(n)=G_{m_0,T}(n), 	
\end{equation}
for all natural $n$. 
Further, the remainder $(F-T)(n-1+w)$ at point $n-1+w$ of the Taylor approximation $T$ of $F$ at $n$ equals $$F^{(p+1)}(n-1+\th w)w^{p+1}/(p+1)!=f^{(2m_0-1)}(n-1+\th w)w^{p+1}/(p+1)!$$ 
for all real $w$ and some $\th=\th_{f,n,w,m_0}\in(0,1)$. So, by \eqref{eq:f^ to0}, 
\begin{equation*}
(F-T)(n-1+w)\underset{n\to\infty}\longrightarrow0 	
\end{equation*}
for each real $w$, and so, by \eqref{eq:G}, 
\begin{equation}\label{eq:G_F-T to0}
G_{m,F-T}(n)\underset{n\to\infty}\longrightarrow0. 	
\end{equation}
It follows now from \eqref{eq:G-G=} and \eqref{eq:G=G} that $G_{m,F}(n)-G_{m_0,F}(n)\underset{n\to\infty}\longrightarrow0$. 
To complete the proof of \eqref{eq:series}, it remains to refer to \eqref{eq:sum-G=} and \eqref{eq:R to R}. 
As for the equalities \eqref{eq:G_m,alt1} and \eqref{eq:G_m,alt2}, their proof is quite similar to that of \eqref{eq:A_m,alt1} and \eqref{eq:A_m,alt2}. 
Proposition~\ref{prop:series} is now completely proved. 
\end{proof}

\begin{proof}[Proof of Proposition~\ref{prop:seriesEM}]
Take any natural $n$. Let 
\begin{equation}\label{eq:R^EM_{m,f}}
	R^\EM_{m,f}(n):=R^\EM_m, 
\end{equation}
with $R^\EM_m$ as defined in \eqref{eq:R^EM}.
Then, by \eqref{eq:REM unif}, 
\begin{equation}\label{eq:REM to REM}
	R^\EM_{m,f}(n)\underset{n\to\infty}\longrightarrow R^\EM_{m,f}(\infty) 
\end{equation}
and, by \eqref{eq:EM}--\eqref{eq:A^EM}, 
$
	\sum_{k=0}^{n-1} f(k)=f(0)+G^\EM_{m,F}(n)-G^\EM_{m,F}(1)+R^\EM_{m,f}(n).  
$ 
So, 
\begin{equation}\label{eq:sum-G^EM=}
	\sum_{k=0}^{n-1} f(k)-G^\EM_{m_0,F}(n)=f(0)+[G^\EM_{m,F}(n)-G^\EM_{m_0,F}(n)]-G^\EM_{m,F}(1)+R^\EM_{m,f}(n). 
\end{equation}
But 
\begin{equation*}
	G^\EM_{m,F}(n)-G^\EM_{m_0,F}(n)=\sum_{j=m_0}^{m-1}\frac{B_{2j}}{(2j)!}F^{(2j)}(n-1)
	=\sum_{j=m_0}^{m-1}\frac{B_{2j}}{(2j)!}f^{(2j-1)}(n-1)\underset{x\to\infty}\longrightarrow0
\end{equation*}
by \eqref{eq:f^^ to0}. 
To complete the proof of Proposition~\ref{prop:seriesEM}, it now remains to refer to \eqref{eq:sum-G^EM=} and \eqref{eq:REM to REM}. 
\qed 
\end{proof}

\begin{proof}[Proof of Proposition~\ref{prop:same}]
Let us first verify the last sentence of this proposition. Here it is enough to assume that $P(x)=P_\be(x):=(x-n)^\be$ for an arbitrary $\be=0,\dots,2m_0-1$. 
Then, by \eqref{eq:G} and Proposition~\ref{lem:=B_p},   
\begin{equation}\label{eq:=B}
	G_{m,P_\be}(n)=\sum_{j=1}^m\ga_{m,j}\sum_{i=0}^{j-1}(-1+j/2-i)^\be=B_\be, 
\end{equation}
whereas, in view of \eqref{eq:G^EM_m} and because $B_0=1$, $B_1=-1/2$, and $B_3=B_5=\dots=0$,  
\begin{multline*}
	G^\EM_{m,P_\be}(n)=(-1)^\be 
    + \frac\be2\,(-1)^{\be-1}  
    +
    \sum_{j=1}^{m-1} B_{2j} \binom\be{2j} (-1)^\be
    =\sum_{\al=0}^{2m-1} B_\al \binom\be\al (-1)^\be \\ 
    =\sum_{\al=0}^\be B_\al \binom\be\al (-1)^\be 
    =\Big(B_\be+\sum_{\al=0}^{\be-1} B_\al \binom\be\al\Big) (-1)^\be.   
\end{multline*}
So, if $\be\ne1$, then 
\begin{equation*}
	(-1)^\be[G^\EM_{m,P_\be}(n)-G_{m,P_\be}(n)]=[1-(-1)^\be]B_\be+\sum_{\al=0}^{\be-1} B_\al \binom\be\al
	=[1-(-1)^\be]B_\be,    
\end{equation*}
by a well-known identity for the Bernoulli numbers -- see e.g.\ \cite[formula~(2), page~229]{ireland-rosen}. 
If $\be$ is even, then $[1-(-1)^\be]B_\be=0$. If $\be=3,5,\dots$, then again $[1-(-1)^\be]B_\be=0$. So, $G^\EM_{m,P_\be}(n)=G_{m,P_\be}(n)$ for all $\be\in\{0,\dots,2m_0-1\}\setminus\{1\}$. 
Also, $G^\EM_{m,P_1}(n)=-1/2=B_1=G_{m,P_1}(n)$, by \eqref{eq:=B}. This completes the verification of the last sentence of Proposition~\ref{prop:same}. 

Let now $T=T_{n,m_0,F}$ be the Taylor polynomial of order $p=2m_0-1$ for the function $F$ at the point $n-1$, as in the proof of Proposition~\ref{prop:series}. Then, by \eqref{eq:G_F-T to0}, $G_{m_0,F-T}(n)\underset{n\to\infty}\longrightarrow0$. Also, 
by \eqref{eq:G^EM_m}, 
$G^\EM_{m_0,F-T}(n)
=0$. 
On the other hand, $T$ is a polynomial of degree $\le2m_0-1$, and so, by the already proved last sentence of Proposition~\ref{prop:same}, $G_{m_0,T}(n)=G^\EM_{m_0,T}(n)$. 
To complete the proof of Proposition~\ref{prop:same}, it remains to use the linearity of $G_{m,F}$ and $G^\EM_{m,F}$ in $F$, which yields $G_{m_0,F}(n)-G^\EM_{m_0,F}(n)=[G_{m_0,T}(n)-G^\EM_{m_0,T}(n)]+G_{m_0,F-T}(n)-G^\EM_{m_0,F-T}(n)
=G_{m_0,F-T}(n)\underset{n\to\infty}\longrightarrow0$. 
\qed
\end{proof}


\bibliographystyle{abbrv}      

\bibliography{P:/mtu_pCloud_02-02-17/bib_files/citations12.13.12}

\def\cprime{$'$} \def\polhk#1{\setbox0=\hbox{#1}{\ooalign{\hidewidth
  \lower1.5ex\hbox{`}\hidewidth\crcr\unhbox0}}}
  \def\polhk#1{\setbox0=\hbox{#1}{\ooalign{\hidewidth
  \lower1.5ex\hbox{`}\hidewidth\crcr\unhbox0}}}
  \def\polhk#1{\setbox0=\hbox{#1}{\ooalign{\hidewidth
  \lower1.5ex\hbox{`}\hidewidth\crcr\unhbox0}}} \def\cprime{$'$}
  \def\polhk#1{\setbox0=\hbox{#1}{\ooalign{\hidewidth
  \lower1.5ex\hbox{`}\hidewidth\crcr\unhbox0}}} \def\cprime{$'$}
  \def\polhk#1{\setbox0=\hbox{#1}{\ooalign{\hidewidth
  \lower1.5ex\hbox{`}\hidewidth\crcr\unhbox0}}} \def\cprime{$'$}
  \def\cprime{$'$}
\begin{thebibliography}{10}

\bibitem{andrews}
G.~E. Andrews, R.~Askey, and R.~Roy.
\newblock {\em Special functions}, volume~71 of {\em Encyclopedia of
  Mathematics and its Applications}.
\newblock Cambridge University Press, Cambridge, 1999.

\bibitem{c++vsMathca}
S.~B. Aruoba and J.~Fern\'andez-Villaverde.
\newblock A comparison of programming languages in economics.
\newblock Working Paper 20263, National Bureau Of Economic Research, June 2014.

\bibitem{bailey-borwein}
D.~H. Bailey and J.~M. Borwein.
\newblock Experimental mathematics: recent developments and future outlook.
\newblock In {\em Mathematics unlimited---2001 and beyond}, pages 51--66.
  Springer, Berlin, 2001.

\bibitem{bornemann}
F.~Bornemann.
\newblock The {SIAM} 100-digit challenge: a decade later. {I}nspirations,
  ramifications, and other eddies left in its wake.
\newblock {\em Jahresber. Dtsch. Math.-Ver.}, 118(2):87--123, 2016.

\bibitem{100digit}
F.~Bornemann, D.~Laurie, S.~Wagon, and J.~Waldvogel.
\newblock {\em The {SIAM} 100-digit challenge}.
\newblock Society for Industrial and Applied Mathematics (SIAM), Philadelphia,
  PA, 2004.
\newblock A study in high-accuracy numerical computing, With a foreword by
  David H. Bailey.

\bibitem{brent-log}
R.~P. Brent.
\newblock Multiple-precision zero-finding methods and the complexity of
  elementary function evaluation.
\newblock \url{https://arxiv.org/abs/1004.3412}, see also
  \url{http://maths-people.anu.edu.au/~brent/pub/pub028.html}, 2010.

\bibitem{butzer-etal11}
P.~L. Butzer, P.~J. S.~G. Ferreira, G.~Schmeisser, and R.~L. Stens.
\newblock The summation formulae of {E}uler-{M}aclaurin, {A}bel-{P}lana,
  {P}oisson, and their interconnections with the approximate sampling formula
  of signal analysis.
\newblock {\em Results Math.}, 59(3-4):359--400, 2011.

\bibitem{candel-etal}
B.~Candelpergher, H.~G. Gadiyar, and R.~Padma.
\newblock Ramanujan summation and the exponential generating function
  {$\sum^\infty_{k=0}{z^k\over k!}\zeta'(-k)$}.
\newblock {\em Ramanujan J.}, 21(1):99--122, 2010.

\bibitem{fillebrown}
S.~Fillebrown.
\newblock Faster computation of {B}ernoulli numbers.
\newblock {\em J. Algorithms}, 13(3):431--445, 1992.

\bibitem{gould}
H.~W. Gould.
\newblock Explicit formulas for {B}ernoulli numbers.
\newblock {\em Amer. Math. Monthly}, 79:44--51, 1972.

\bibitem{hardy-ramanuj}
G.~H. Hardy and S.~Ramanujan.
\newblock Asymptotic formul\ae \ in combinatory analysis [{P}roc. {L}ondon
  {M}ath. {S}oc. (2) {\bf 17} (1918), 75--115].
\newblock In {\em Collected papers of {S}rinivasa {R}amanujan}, pages 276--309.
  AMS Chelsea Publ., Providence, RI, 2000.

\bibitem{harvey10}
D.~Harvey.
\newblock A multimodular algorithm for computing {B}ernoulli numbers.
\newblock {\em Math. Comp.}, 79(272):2361--2370, 2010.

\bibitem{ierley-maple}
G.~Ierley.
\newblock Private communication, 2017.

\bibitem{ireland-rosen}
K.~Ireland and M.~Rosen.
\newblock {\em A classical introduction to modern number theory}, volume~84 of
  {\em Graduate Texts in Mathematics}.
\newblock Springer-Verlag, New York, second edition, 1990.

\bibitem{kleinert}
H.~Kleinert.
\newblock {\em Path integrals in quantum mechanics, statistics, polymer
  physics, and financial markets}.
\newblock World Scientific Publishing Co. Pte. Ltd., Hackensack, NJ, fifth
  edition, 2009.

\bibitem{knopp51}
K.~Knopp.
\newblock {\em Theory and application of infinite series}.
\newblock Blackie, London, 1951.

\bibitem{knuth62}
D.~E. Knuth.
\newblock Euler's constant to {$1271$} places.
\newblock {\em Math. Comp.}, 16:275--281, 1962.

\bibitem{amdeber}
MathOverflow.
\newblock A representation of the {B}ernoulli numbers.
\newblock
  \url{http://mathoverflow.net/questions/252404/a-representation-of-the-bernoulli-numbers#252410},
  2016.

\bibitem{dichot_published}
I.~Pinelis.
\newblock A topological dichotomy with applications to complex analysis.
\newblock {\em Colloq. Math.}, 139(1):137--146, 2015.

\bibitem{sums-via-ints-multivar}
I.~Pinelis.
\newblock Approximating sums by integrals only: multiple sums and sums over
  lattice polytopes.
\newblock \url{https://arxiv.org/abs/1705.09159}, 2017.

\bibitem{richardson}
L.~Richardson.
\newblock The approximate arithmetical solution by finite differences of
  physical problems involving differential equations, with an application to
  the stress in a masonry dam.
\newblock {\em Phil. Trans. Roy. Soc. London, Series A}, 210:307--357, 1910.

\bibitem{sadovskii}
M.~V. Sadovskii.
\newblock {\em Quantum field theory}, volume~17 of {\em De Gruyter Studies in
  Mathematical Physics}.
\newblock De Gruyter, Berlin, 2013.

\bibitem{sidi}
A.~Sidi.
\newblock {\em Practical extrapolation methods}, volume~10 of {\em Cambridge
  Monographs on Applied and Computational Mathematics}.
\newblock Cambridge University Press, Cambridge, 2003.
\newblock Theory and applications.

\bibitem{titchmarsh}
E.~C. Titchmarsh.
\newblock {\em The theory of functions}.
\newblock Oxford University Press, Oxford, 1958.
\newblock Reprint of the second (1939) edition.

\end{thebibliography}

\end{normalsize}

\end{document}